\documentclass{article}

\usepackage[utf8]{inputenc}
\usepackage[T1]{fontenc}
\usepackage[british]{babel}
\usepackage{amsmath}
\usepackage{amsthm}
\usepackage{amsfonts}
\usepackage{amssymb}
\usepackage{mathscinet}
\usepackage{bm}
\usepackage{makeidx}
\usepackage[pdftex]{graphicx}

\theoremstyle{plain}
\newtheorem{theorem}{Theorem}[section]
\newtheorem{lemma}{Lemma}[section]
\theoremstyle{definition}
\newtheorem{definition}{Definition}[section]
\newtheorem{example}{Example}[section]
\theoremstyle{remark}
\newtheorem{remark}{Remark}[section]

\DeclareMathOperator{\id}{id}
\DeclareMathOperator{\Hom}{Hom}
\DeclareMathOperator{\tr}{tr}

\DeclareMathOperator{\res}{res}
\DeclareMathOperator{\RE}{Re}
\DeclareMathOperator{\IM}{Im}
\DeclareMathOperator{\diverge}{div}
\DeclareMathOperator{\curl}{\mathbf{curl}}
\DeclareMathOperator{\Var}{\mathsf{Var}}

\newcommand{\pcoor}[1]{%
  \begingroup\lccode`~=`: \lowercase{\endgroup
  \edef~}{\mathbin{\mathchar\the\mathcode`:}\nobreak}%
  [
  \begingroup
  \mathcode`:=\string"8000
  #1%
  \endgroup
  ]
}

\makeindex

\begin{document}

\title{Tensor- and spinor-valued random fields with applications to continuum physics and cosmology}

\author{Anatoliy Malyarenko\thanks{Corresponding author. Division of Mathematics and Physics, Mälardalen University, Box 883, 721 23 Västerås, Sweden. E-mail: \texttt{anatoliy.malyarenko@mdh.se}.} \and Martin Ostoja-Starzewski\thanks{Department of Mechanical Science \& Engineering, also Institute for Condensed Matter Theory and Beckman Institute, University of Illinois at Urbana-Champaign, Urbana, IL, 61801-2906 USA. E-mail: \texttt{martinos@illinois.edu}.}}

\date{\today}

\maketitle

\begin{abstract}
In this paper, we review the history, current state-of-art, and physical applications of the spectral theory of two classes of random functions. One class consists of homogeneous and isotropic random fields defined on a Euclidean space and taking values in a real finite-di\-men\-si\-o\-nal linear space. In applications to continuum physics, such a field describes physical properties of a homogeneous and isotropic continuous medium in the situation, when a microstructure is attached to all medium points. The range of the field is the fixed point set of a symmetry class, where two compact Lie groups act by orthogonal representations. The material symmetry group of a homogeneous medium is the same at each point and acts trivially, while the group of physical symmetries may act nontrivially. In an isotropic random medium, the rank $1$ (resp. rank $2$) correlation tensors of the field transform under the action of the group of physical symmetries according to the above representation (resp. its tensor square), making the field isotropic.

Another class consists of isotropic random cross-sections of homogeneous vector bundles over a coset space of a compact Lie group. In applications to cosmology, the coset space models the sky sphere, while the random cross-section models a cosmic background. The Cosmological Principle ensures that the cross-section is isotropic.

For convenience of the reader, a necessary material from multilinear algebra, representation theory, and differential geometry is reviewed in Appendix.
\end{abstract}

\clearpage

\tableofcontents

\listoffigures

\section{Introduction}

Random functions of several variables appeared for the first time as
mathematical models for physical phenomena like turbulence, see
\cite{Keller1925,MR0001702,doi:10.1098/rspa.1935.0158,Karman1937b,Karman1937a,MR28160,MR26884,Karman1938,MR0042255}.
Since then, the theory developed further and a lot of new applications
appeared, see historical accounts in
\cite{MR2977490,MR2406667,MR2406668,MR2341287,MR2676223,MR697386,MR893393,MR915557}.

In this survey paper, we describe two particular directions in the theory of
random functions of several variables.
\begin{enumerate}
  \item The first one studies random fields defined on the space $\mathbb{R}^d$ and taking values in a real finite-dimensional space, say $U$. The case, when the elements of $U$ are tensors of a fixed rank over $\mathbb{R}^d$, turns out to be very interesting for applications in continuum physics.
  \item The second direction studies random cross-sections of vector, tensor and spinor bundles over manifolds. The case, when the base of the bundle is the sphere $S^2$, found applications in cosmology.
\end{enumerate}

It turns out that the two directions described above are linked to several parts of mathematics. In Section~\ref{sec:physical}, we give a physical motivation for introducing random fields and random cross-sections of vector and tensor bundles. Following a physical tradition, we place physical media into an affine space without a fixed origin, and later vectorise the above space. We explain the \emph{Curie principle}\index{Curie principle} that gives a relation between \emph{medium}\index{symmetry group!medium} and \emph{physical}\index{symmetry group!physical} symmetry groups of a physical medium. It is here where we give a rigorous mathematical definition of a homogeneous and isotropic random field, the main object for investigations in Sections~\ref{sec:Euclidean} and \ref{sec:continuum}. An explanation to the fact that the Cosmic Microwave Background can be modelled mathematically as an isotropic random cross-section of a bundle over the sky sphere, is given.

In Section~\ref{sec:history}, we give a short historical account of the topic and describe how the two directions mentioned above developed in time and how the links between the theory of random functions of several variables and several parts of deterministic mathematics have been established.

Section~\ref{sec:Euclidean} occupies a significant part of the paper. In applications to continuum physics, a random field model is applied to different physical phenomena like turbulence, linear strain, piezoelectricity, linear elasticity, etc. How many strain, piezoelectric, elasticity,\dots classes exist? In Subsection~\ref{sub:formulation}, we answer to this question, attach a homogeneous and isotropic random field to the orthogonal representation of the physical symmetry group that acts in the fixed point set of a symmetry class of the medium symmetry group, and formulate a general problem: how to calculate the general form of the one- and two-point correlation tensors of the introduced field? How to find its spectral expansion?

As usual, there are several ways to solve a given mathematical problem. In Subsection~\ref{sub:principles} we explain our approach. In short: we divide our solution into two parts; the first one is coordinate-free, the second one starts from choosing the most convenient basis in the fixed point set $U$ of a symmetry class and finishes by writing the two-point correlation tensor and the spectral expansion of the field in the chosen basis.

The material given in Subsection~\ref{sub:homogeneous} is well-known and given there for the sake of completeness. We describe a homogeneous random field $T(\mathbf{x})$ defined on the \emph{space domain}\index{space domain} $\mathbb{R}^d$ and taking values in a real finite-dimensional linear space $U$, in terms of a measure $F$ defined on the Borel $\sigma$-field $\mathfrak{B}(\hat{\mathbb{R}}^d)$ of the \emph{wavenumber domain}\index{wavenumber domain} $\hat{\mathbb{R}}^d$ and taking values in the cone of Hermitian nonnegative-definite linear operators on the complexification $cU$ of the space $U$. It remains to find all such measures $F$ that the corresponding random field is not only homogeneous, but also isotropic.

A preliminary answer to this question is given in Subsection~\ref{sub:preliminary}. We identify the closed subgroup $\tilde{G}$ of the group $\mathrm{O}(d)$, the subspace $\tilde{U}$ of the real linear space of Hermitian linear operators in $cU$, and the orthogonal representation $\theta^{\tilde{U}}$ of the group $\tilde{G}$ such that: a $U$-valued homogeneous random field $T(\mathbf{x})$ is isotropic if and only if the corresponding measure $F$ satisfies Equation~\eqref{eq:15} below. This statement is immediately formulated in the form of Equation~\eqref{eq:16} which describes more simple objects: an ordinary (not operator-valued) measure $\mu$, and a function $f$ defined on the wavenumber domain and taking values in a \emph{finite-dimensional convex compact set}. This link between the theory of homogeneous and isotropic random fields and the theory of finite-dimensional convex compacta was established by the authors.

In Subsection~\ref{sub:general}, we sketch a proof for Equation~\eqref{eq:18}, which gives the two-point correlation tensor of a homogeneous and isotropic random field in a coordinate-free form. We refer to the papers, when the above equation is transformed to a coordinate form.

Instead of reproducing the above form with a lot of indices, we consider several examples important for continuum physics. All examples follow the same $6$-steps scheme, described in the beginning of Subsection~\ref{sub:examples}. In different examples, the steps are described with different degree of going into details. We conclude in Subsection~\ref{sub:concluding}.

In Section~\ref{sec:continuum}, we describe two major applications of the theory developed in Section~\ref{sec:Euclidean}, to continuum physics. One has to do with imposing restrictions on dependent tensor fields such as displacement, velocity or stress. These restrictions are dictated by the field equations of continuum theories such as mechanics, conductivity, or electrostatics.

The second application relates to positive-definiteness property of tensor fields of constitutive properties in continuum mechanics. First and foremost, it appears in linear constitutive theories such as classical elasticity or electrical permeability, which are defined in terms of the free energy functionals. Secondly, positive-definiteness is crucial in formulating models of irreversible material behavior based on the dissipation functionals, such as exemplified by conductivity and viscous fluids.  Thus, if statistical continuum mechanics theories are built on stochastic functionals of free energy and/or dissipation function, tensor-valued random fields of rank $2$ and $4$ are essential.

Fluctuations of random fields of elasticity and viscosity provide statistical information about the extrema.  The minima correspond to the loss of positive-definiteness, which is a well-known possibility in elasticity, e.g. \cite{MR1318388}. On the other hand, and in light of the non-equilibrium statistical mechanics \cite{doi:10.1146/annurev-conmatphys-062910-140506}, there is also a possibility of violations of the Clausius--Duhem inequality when the continuum point is set up on length and/or time scales comparable to characteristic features (mean-free paths and/or collision times) of microstructure \cite{OstojaStarzewski2014,MR4203086}.  When the number of elements (atoms or grains) in this statistical volume element becomes sufficiently large, and when the observation time windows grow, the probability of spontaneous violations of the Second Law vanishes.

The exposition in Section~\ref{sec:sections} is motivated by the idea formulated in \cite{MR2928833}. On the one hand, in the current standard model of particle physics, both six quarks and six leptons are divided into three \emph{generations}:\index{generations of particles} pairs of particles that demonstrate a similar physical behaviour. For example, the first generation of leptons includes electron and electron neutrino, the second muon and muon neutrino, the third tau and tau neutrino.

On the other hand, modern quantum physics is based on the theory of complex unitary representations of topological groups. The latter come in three flavours: the representations of \emph{real}, \emph{complex}, and \emph{quaternionic} type. It is supposed that there exists a link between three generations of particles and three types of complex representations.

Following this idea, we consider in a \emph{unified} way homogeneous vector bundles $(E,B,\pi)$ whose fibers are (right) finite-dimensional linear spaces over three (skew) fields: the field $\mathbb{R}$ of reals, the field $\mathbb{C}$ of complex numbers, and the skew field $\mathbb{H}$ of quaternions. Theorem~\ref{th:section} is new and gives the spectral expansion of an isotropic random section of a homogeneous vector bundle with the above described fibres. The expansion is performed with respect to a \emph{special} basis in the Hilbert space of square-integrable cross-sections of the bundle $(E,B,\pi)$. Here, we explain how such a basis appears, using a simple homogeneous vector bundle.

Let $G=\mathrm{SO}(3)$ be the group of orthogonal $3\times 3$ matrices with unit determinant, and let $H=\mathrm{SO}(2)$ be the subgroup of $G$ that leave the point $(0,0,1)^{\top}\in\mathbb{R}^3$ fixed. It is well-known that the set $G/H$ of left cosets is identified with the sphere $B=S^2$. Put $E=B\times\mathbb{C}^1$ and $\pi(x,z)=x$ for all $x\in S^2$ and $z\in\mathbb{C}^1$. The Hilbert space $L^2(S^2)$ of square-integrable cross-sections of the bundle $(E,B,\pi)$ with respect to the Lebesgue measure $\mathrm{d}x$ consists of functions on the sphere. The unitary representation $(g\cdot f)(x)=f(g^{-1}x)$ for $g\in G$ and $f\in L^2(S^2)$ is the direct sum of the irreducible unitary representations $cU_{\ell}$ of the group $G$ over nonnegative integers $\ell$. The space $cU_{\ell}$ of dimension $2\ell+1$ is the famous space of harmonic polynomials of degree~$\ell$.

Let $s(x)$ be a centred second-order random cross-section of the line bundle $(E,B,\pi)$, and let $\{\,f_n(x)\colon n\geq 1\,\}$ be an arbitrary orthonormal basis of the space $L^2(S^2)$. The random variables
\[
a_n=\int_{S^2}s(x)\overline{f_n(x)}\,\mathrm{d}x
\]
are the Fourier coefficients of the random field $s(x)$. In general, they are correlated. However, if $s(x)$ is \emph{isotropic} and the basis can be divided into countably many subsets such that the $\ell$th one constitute a basis of the space $cU_{\ell}$, then the random variables $a_n$ are uncorrelated. Moreover, the variance of $a_n$ depends only on the subset to which $f_n(x)$ belongs. In particular, if one chooses the spherical harmonics $Y_{\ell,m}(x)$ as the basis, we recover the result of \cite{Obukhov1947}:
\[
s(x)=\sum_{\ell=0}^{\infty}\sum_{m=-\ell}^{\ell}a_{\ell m}Y_{\ell,m}(x).
\]

The general algorithm for choosing such a basis is described in Subsection~\ref{sub:construction}. We discuss it below.

Note that the random fields considered in Sections~\ref{sec:Euclidean} and \ref{sec:continuum} are random sections of \emph{trivial} tensor bundles over Euclidean spaces.

Section~\ref{sec:cosmology} is a collection of examples. We collect different spectral expansions of the Stokes parameters of the CMB using different orthonormal bases in the Hilbert spaces of square-integrable sections of various homogeneous vector bundles, where the fibers are linear spaces over various (skew) fields. See also similar calculations in \cite{MR3746005}.

Finally, Appendix~\ref{sec:language} is intended for the readers who do not specialise in the areas of mathematics linked to the theory of random fields. The material here is standard, but some results are new and some approaches require a short explanation.

In vast majority of mathematical and physical literature, three equivalent approaches to the definition of a tensor are most popular. An axiomatic approach using some universal property due to Bourbaki \cite{MR1727844} is probably the most elegant of them. A constructive approach through multidimensional arrays goes back to \cite{R1,MR1511109} and is well-known to all physicists. See also a comprehensive historical account in \cite{MR1058203} and a modern one in \cite{MR4298222}. We choose an intermediate approach through multi-linear maps, which goes back at least to the first edition of \cite{MR2105604} in 1960. A discussion concerning tensor product of quaternionic spaces goes to \cite{MR1045637}.

There exists a lot of excellent books concerning group representations. A few of them consider real, complex, and quaternionic representations simultaneously, \cite{MR0252560,MR1410059,MR1153249} are among them. A closely connected area, Invariant Theory, is described in \cite{Spencer}.

The definition of a manifold as the level set of a continuously differentiable function goes back to \cite{Poincare}, while the definition in terms of patching Euclidean pieces together was given in 1913 in the first edition of \cite{MR1440406}. These definitions are equivalent by the \emph{Whitney Embedding Theorem} \cite{MR1503303}. We choose to use the latter; nevertheless, the dispute aimed to argue which definition is better continues nowadays, see, e.g., \cite{MR1618209}.

Fiber bundles can also be defined in several ways. We give a short description of two of them. The first one is gluing up a bundle from trivial patches $\mathcal{U}\times F$, where $\mathcal{U}$ is the domain of a chart on the base manifold $B$, and $F$ is the fiber. The second one is the principal bundle point of view. See a detailed comparison of the two approaches in \cite{MR2136212} as well as a very condensed explanation in \cite{MR484262}.

We choose the latter approach. With the help of principal bundles, we give a simple description of spin groups, see Example~\ref{ex:spin} below, and of spin bundles in Example~\ref{ex:spinweight}. Compare the latter explanation with \cite[Section~2.1]{MR3410545} which uses the former approach. See also \cite{MR2136212,MR2384313} for equivalence of the two approaches.

The most important class of fiber bundles for applications to random fields are the so called \emph{homogeneous vector bundles}. In our approach, we start from a principal fibre bundle $(G,H,B,\tilde{\pi})$, where $G$ is a Lie group, $H$ its closed subgroup, $B=G/H$ is the set of left cosets, and $\tilde{\pi}(g)=gH$. We restrict ourselves to the case of \emph{compact} group $G$. A homogeneous vector bundle\index{bundle!vector!homogeneous} is the bundle $(E,B,\pi)$ associated to $(G,H,B,\tilde{\pi})$ by a representation of $H$ in a finite-dimensional linear space $L_0$ over a (skew) field $\mathbb{K}$. The Hilbert space $L^2(E)$ of square-integrable cross-sections of such a bundle carries the representation of the group $G$ \emph{induced} by the representation of $H$ described above. The space $L^2(E)$ can be uniquely represented as the direct sum of finite-dimensional \emph{isotypic subspaces},\index{isotypic subspace} where the representations of $G$, multiple to the irreducible ones, act. The multiplicities of the irreducible components are calculated with the help of the \emph{Frobenius reciprocity}\index{Frobenius reciprocity} explained in Subsection~\ref{sub:structure}. Note that the Frobenius reciprocity can be formulated for the case when the Lie group $G$ is not necessarily compact, see \cite{MR889252}.

As we have seen above, using the trivial homogeneous line bundle over the sphere $S^2$, it is important to construct orthonormal bases in the isotypic subspaces of the space $L^2(E)$. To construct such a basis, we use the following fact. Let $L$ be an irreducible representation of the group $G$. Either one copy or multiple copies of $L$ act in the nonzero isotypic subspace of $L^2(E)$ corresponding to $L$ if and only if the dimension of the linear space $\Hom_{\mathbb{K}G}(L,L^2(E))$ of intertwining operators between $L$ and $L^2(E)$ is positive. By the Frobenius reciprocity, the above space is isomorphic to the linear space $\Hom_{\mathbb{K}H}(\res^G_HL,L_0)$ of intertwining operators between the representation $\res^G_HL$ of the group $G$ given by the restriction of the representation $L$ from $G$ to $H$, and the representation $L_0$. Moreover, the isomorphism is constructed \emph{explicitly}.

In Subsection~\ref{sub:construction}, we start by construction of a basis in the linear space $\Hom_{\mathbb{K}H}(\res^G_HL,L_0)$. The basis is constructed with the help of Schur's Lemma. Up to our knowledge, the last item in our formulation of Lemma~\ref{lem:Schur} about the structure of the real linear space of intertwining operators between two copies of an irreducible quaternionic representation is new. Proof is included.

The images of the elements of the above described basis under the explicit Frobenius reciprocity isomorphism form an orthogonal basis in the isotypic subspace of $L^2(E)$ corresponding to a given irreducible representation $L$ of the group~$G$. After eventual normalisation, the obtained basis is suitable for a spectral expansion of an isotropic cross-section of the homogeneous vector bundle $(E,B,\pi)$. In particular, both ordinary and spin-weighted complex-valued spherical harmonics can be considered as particular cases of our construction. In Section~\ref{sec:cosmology}, we construct their real-valued counterparts and use both complex- and real-valued harmonics for spectral expansion of the Stokes parameters of the Cosmic Microwave Background.

Finally, it is known that the current standard cosmological model predicts the existence of \emph{three} cosmic backgrounds. In addition to the Cosmic Microwave Background, the gravitational and neutrino backgrounds should exist, see, e.g., \cite{doi:10.1142/S0218301317400080}. The spectral expansions of the corresponding random fields may be constructed using the methods described in this paper. The spectral expansion of the CMB starts from $\ell=2$, because CMB consists of photons of spin~$1$. The expansion for the case of neutrino background should start from $\ell=1$, because neutrinos have spin $\frac{1}{2}$. Likewise, the expansion for the case of gravitational background should start from $\ell=4$, because the hypothetical quants of the gravitational radiation, gravitons, have spin $2$.

\section{Physical motivation}\label{sec:physical}

The reader may consult Appendix~\ref{sec:language} for mathematical terms used below.

\begin{example}\label{ex:continuum}
Let $D$ be a bounded domain in the affine Euclidean space $E^d$, filled with a continuous medium. Following a physics tradition, we call the points in $D$ \emph{places}.\index{place} Let $U$ be a real finite-dimensional linear space, and let $T\colon D\to U$ be a function that describes a particular physical parameter of the medium. For example:
\begin{itemize}
  \item $U=\mathbb{R}^1$ and $T(A)$ is the temperature at the place $A\in D$;
  \item $d=3$, $U=\mathbb{R}^3$ and $T(A)$ is the velocity of the fluid at $A$;
  \item $U=\mathsf{S}^2(\mathbb{R}^3)$, the linear space of symmetric rank $2$ tensors over $\mathbb{R}^3$, and $T(A)$ is the strain tensor of a deformable body at $A$;
  \item $U=\mathsf{S}^2(\mathbb{R}^3)\otimes\mathbb{R}^3$ and $T(A)$ is the piezoelectricity tensor at $A$;
  \item $U=\mathsf{S}^2(\mathsf{S}^2(\mathbb{R}^3))$ and $T(A)$ is the elasticity tensor at $A$.
\end{itemize}
More examples can be found in \cite{MR3692350,MR4045367,MR3208052}.

If the medium is gaseous or liquid, its movement may become turbulent. In a deformable body, a spatially random material microstructure may be present. In all these cases, we speak of a \emph{random medium}.\index{random medium} More explanation can be found in \cite[Section~1.1]{MR3930601} and \cite{MR2341287}. The function $T$ becomes a \emph{random field}.\index{random field} There is a probability space $(\Omega,\mathfrak{F},\mathsf{P})$ and a function $T\colon D\times\Omega\to V$ such that for any fixed $A\in D$ and a Borel set $B\subseteq V$, the set $\{\,\omega\in\Omega\colon T(A,\omega)\in B\,\}$ is an event. As usual, the random field $T(A,\omega)$ is completely determined by its finite-dimensional distributions $(T(A_1,\omega),\dots,T(A_n,\omega))$, where $n$ is a positive integer, and $A_1$, \dots, $A_n$ are $n$ \emph{distinct} places in $D$.

Fix a place $O\in D$. There exists a unique structure of a \emph{linear} space in $E^d$ such that the map $E^d\to\mathbb{R}^d$, $A\mapsto A-O$, is an isomorphism of linear spaces. In what follows, we identify the spaces $E^d$ and $\mathbb{R}^d$ with the help of the above map. To simplify the exposition, we suppose that the random field $\{\,T(\mathbf{x})\colon\mathbf{x}\in D\subset\mathbb{R}^d\,\}$ is the restriction to $D$ of another random field defined on all of $\mathbb{R}^d$, and denote the new field by the same symbol $T(\mathbf{x})$.

Assume that the physical properties of the random medium do not depend on the choice of the origin $O$, that is, the medium is homogeneous. Mathematically, the random field $T(\mathbf{x})$ is strictly homogeneous.

\begin{definition}
A random field $T(\mathbf{x})$ is called \emph{strictly homogeneous}\index{random field!strictly homogeneous} if for any positive integer $n$, for any $n$ distinct points $\mathbf{x}_1$, \dots, $\mathbf{x}_n$ in $\mathbb{R}^d$, and for arbitrary $\mathbf{x}\in\mathbb{R}^d$, the finite-dimensional distributions $(T(\mathbf{x}_1),\dots,T(\mathbf{x}_n))$ and $(T(\mathbf{x}_1+\mathbf{x}),\dots,T(\mathbf{x}_n+\mathbf{x}))$ are identical.
\end{definition}

Assume that the random field $T(\mathbf{x})$ is \emph{second-order},\index{random field!second-order} that is, $\mathsf{E}[\|T(\mathbf{x})\|^2]<\infty$, $\mathbf{x}\in\mathbb{R}^d$. If, in addition, such a field is strictly homogeneous, then it is wide-sense homogeneous.

\begin{definition}
A second-order random field  $T(\mathbf{x})$ is called \emph{wide-sense homogeneous}\index{random field!wide-sense homogeneous} if and only if its \emph{one-point correlation tensor}\index{correlation tensor!one-point} $\langle T(\mathbf{x})\rangle=\mathsf{E}[T(\mathbf{x})]$ and the \emph{two-point correlation tensor}\index{correlation tensor!two-point}
\begin{equation}\label{eq:2}
\langle T(\mathbf{x}),T(\mathbf{y})\rangle=\mathsf{E}[(T(\mathbf{x})-\langle T(\mathbf{x})\rangle)\otimes((T(\mathbf{y})-\langle T(\mathbf{y})\rangle))]
\end{equation}
are shift-invariant, that is, for any $\mathbf{z}\in\mathbb{R}^d$ we have
\begin{equation}\label{eq:7}
\langle T(\mathbf{x}+\mathbf{z})\rangle=\langle T(\mathbf{x})\rangle,\qquad
\langle T(\mathbf{x}+\mathbf{z}),T(\mathbf{y}+\mathbf{z})\rangle=\langle T(\mathbf{x}),T(\mathbf{y})\rangle.
\end{equation}
\end{definition}

In what follows, we consider only wide-sense homogeneous random fields and call them just homogeneous. We also assume, that a random field $T(\mathbf{x})$ is \emph{mean-square continuous},\index{random field!mean-square continuous} that is, for any $\mathbf{x}\in\mathbb{R}^d$ we have
\begin{equation}\label{eq:33}
\lim_{\mathbf{y}-\mathbf{x}\to \mathbf{0}}\mathsf{E}[\|T(\mathbf{x})-T(\mathbf{y})\|^2]=0.
\end{equation}
Note under some weak conditions a measurable second-order random field is automatically mean-square continuous, see \cite{MR3064996}.

What happens with the random field $T(\mathbf{x})$ under rotation and/or reflection? A microstructure is attached to any material point $\mathbf{x}\in D$. At the macroscopic scale all the details of the microstructure are lost, all what remains is the \emph{material symmetry group}\index{symmetry group!material} $G$. The above group is the same for all material points, because the medium is homogeneous. The group $G$ is a closed subgroup of the group $\mathrm{O}(d)$ of orthogonal $d\times d$ matrices. On the other hand, the physical properties of the media are encoded by the \emph{physical symmetry group},\index{symmetry group!physical} denote it by $H$. The \emph{Curie principle}\index{Curie principle} states that $G\subseteq H$, see \cite{AUFFRAY2019197,MR1251809}.

Under the action of a matrix $g\in G$, a point $\mathbf{x}\in\mathbb{R}^d$ becomes the point $g\mathbf{x}$. The random element $T(\mathbf{x})$ of the linear space $U$ becomes the element $\theta^U(g)T(\mathbf{x})$. The random fields $T(g\mathbf{x})$ and $\theta^U(g)T(\mathbf{x})$ must be identical.

\begin{definition}
A random field $T(\mathbf{x})$ is called \emph{strictly isotropic}\index{random field!strictly isotropic} if for any positive integer $n$, for any $n$ distinct points $\mathbf{x}_1$, \dots, $\mathbf{x}_n$ in $\mathbb{R}^d$, and for arbitrary $g\in G$, the finite-dimensional distributions
\[
(T(g\mathbf{x}_1),\dots,T(g\mathbf{x}_n))\quad\text{and}\quad (\theta^U(g)T(\mathbf{x}_1),\dots,\theta^U(g)T(\mathbf{x}_n))
\]
are identical.
\end{definition}

In particular, we have
\[
\langle T(g\mathbf{x})\rangle=\langle\theta^U(g)T(\mathbf{x})\rangle,\qquad\mathbf{x}\in\mathbb{R}^d.
\]
The continuous linear operator $\theta^U(g)$ commutes with the mathematical expectation, and we obtain
\begin{equation}\label{eq:4}
\langle T(g\mathbf{x})\rangle=\theta^U(g)\langle T(\mathbf{x})\rangle.
\end{equation}
Similarly, the two-point correlation tensors of the two above random fields must be equal:
\[
\langle T(g\mathbf{x}),T(g\mathbf{y})\rangle=\langle\theta^U(g)T(\mathbf{x}),\theta^U(g)T(\mathbf{y})\rangle.
\]
By definition of the two-point correlation tensor \eqref{eq:2} and by continuity of the tensor product, we have
\[
\langle\theta^U(g)T(\mathbf{x}),\theta^U(g)T(\mathbf{y})\rangle=(\theta^U\otimes\theta^U)
(g)\langle T(\mathbf{x}),T(\mathbf{y})\rangle
\]
which gives
\begin{equation}\label{eq:3}
\langle T(g\mathbf{x}),T(g\mathbf{y})\rangle=(\theta^U\otimes\theta^U)
(g)\langle T(\mathbf{x}),T(\mathbf{y})\rangle.
\end{equation}

\begin{definition}\label{def:isotropic}
A second-order random field  $T(\mathbf{x})$ is called \emph{wide-sense isotropic}\index{random field!wide-sense isotropic} if and only if its one-point correlation tensor satisfies Equation~\eqref{eq:4} and its two-point correlation tensor satisfies Equation~\eqref{eq:3}.
\end{definition}

In what follows, we consider only wide-sense isotropic random fields and call them just isotropic.
\end{example}

\begin{example}\label{ex:standard}
The current standard model of cosmology assumes that the universe is a $4$-dimensional differentiable manifold, say $M$, called a \emph{spacetime}.\index{spacetime} For any two intersecting charts $(\mathcal{U}_{\alpha},\varphi_{\alpha})$ and $(\mathcal{U}_{\beta},\varphi_{\beta})$ of the manifold $M$, and for any $x\in \mathcal{U}_{\alpha}\cap \mathcal{U}_{\beta}$, denote by $w_{\alpha\beta}(x)$ the Jacobi matrix of the map $\varphi_{\alpha}\circ\varphi^{-1}_{\beta}$ at the point $x$. Glue up a topological space from Cartesian products $\mathcal{U}_{\alpha}\times\mathbb{R}^4$ by identifying a point $(x,\mathbf{y}_{\alpha})\in \mathcal{U}_{\alpha}\times\mathbb{R}^4$ with a point $(x,\mathbf{y}_{\beta})\in \mathcal{U}_{\beta}\times\mathbb{R}^4$ whenever $\mathbf{y}_{\alpha}=w_{\alpha\beta}\mathbf{y}_{\beta}$. The obtained topological space is denoted by $TM$. The projection $p\colon TM\to M$ that maps a point $(x,\mathbf{y})\in TM$ to the point $x\in M$, defines a \emph{tangent bundle}\index{bundle!tangent} $(TM,p,M)$. It is assumed that a symmetric non-degenerate bilinear form $g_x$ with signature $(-,+,+,+)$ in $T_xM=p^{-1}(x)$ is assigned in a differentiable way at each $x\in M$.

About 380,000 years after the Big Bang, the electromagnetic waves decoupled from the rest of the universe. Now, they are observable as the \emph{Cosmic Microwave Background},\index{Cosmic Microwave Background} or just the CMB. The waves of the CMB spread in the manifold $(M,g)$ according to the Maxwell equations. The \emph{geometric optic approximation},\index{geometric optic approximation} see \cite{MR0418833}, says that the trajectories of the CMB waves are very close to the so called \emph{null geodesics},\index{null geodesic} that is, the geodesic lines with null tangent vectors. A tangent vector $X\in T_xM$ is null\index{null vector} if $g_x(X,X)=0$. The manifold $(M,g)$ is assumed to be \emph{time-oriented},\index{manifold!time-oriented} that is, one can continuously choose a negative component
\[
N^-(x)=\{\,X\in T_xM\colon g_x(X,X)=0,X^0<0\,\}
\]
in each \emph{null cone}\index{null cone} $N(x)=\{\,X\in T_xM\colon g_x(X,X)=0\,\}$. An observer at the point $x\in M$ observes the waves of the CMB at each point of the \emph{celestial sphere}\index{celestial sphere}
\[
S^2=\{\,X\in N^-(x)\colon (X^1)^2+(X^2)^2+(X^3)^2=1\,\}.
\]
A CMB detector measures an \emph{electric field}\index{electric field} $\mathbf{E}$ perpendicular to the direction of observation $\mathbf{n}\in S^2$. Mathematically, $\mathbf{E}(\mathbf{n})\in T_{\mathbf{n}}S^2$, the tangent plane at the point $\mathbf{n}$. The \emph{intensity tensor}\index{intensity tensor} $P(\mathbf{n})$ is proportional to the tensor product $\mathbf{E}(\mathbf{n})\otimes\mathbf{E}^*(\mathbf{n})$. It is a cross-section of the tensor bundle $(TM\otimes T^*M,p\otimes p^*,M)$.

The Cosmological Principle states that at least on large scales, the Universe is homogeneous and isotropic. Therefore, the standard model of cosmology assumes that the CMB is an \emph{isotropic random cross-section} of the above tensor bundle. We develop the theory of isotropic random cross-sections in Section~\ref{sec:sections}.
\end{example}

\section{A short history of the topic}\label{sec:history}

Recall that a continuous function $f\colon\mathbb{R}^d\times\mathbb{R}^d\to\mathbb{R}$ is called \emph{positive-definite}\index{positive-definite function} if for any positive integer $n$, for any $n$ points $\mathbf{x}^1$, \dots, $\mathbf{x}_n$ in $\mathbb{R}^d$, and for arbitrary real numbers $c_1$, \dots, $c_n$ we have
\[
\sum_{i,j=1}^{n}c_jc_jf(\mathbf{x}_i,\mathbf{x}_j)\geq 0.
\]
In 1938, Isaac Jacob Schoenberg\index{Schoenberg, Isaac Jakob} \cite{MR1503439} proved the following result.

\begin{theorem}[I.J. Schoenberg, \cite{MR1503439}]\label{th:Schoenberg}
Equation
\begin{equation}\label{eq:5}
f(\mathbf{x},\mathbf{y})=2^{(d-2)/2}\Gamma(d/2)\int_{0}^{\infty}
\frac{J_{(d-2)/2}(\lambda\|\mathbf{x}-\mathbf{y}\|)}
{\|\mathbf{x}-\mathbf{y}\|^{(d-2)/2}}\,\mathrm{d}\Phi(\lambda)
\end{equation}
establishes a one-to-one correspondence between the class of continuous positive-definite functions $f(\mathbf{x},\mathbf{y})$ whose values depend only on the distance $\|\mathbf{x}-\mathbf{y}\|$ between the points $\mathbf{x}$ and $\mathbf{y}$, and the class of finite Borel measures $\Phi$ on the set $[0,\infty)$.
\end{theorem}

In Theorem~\ref{th:Schoenberg}, the symbol $\Gamma$ denotes the gamma function, while the symbol $J_{(d-2)/2}$ denoted the Bessel function of the first kind of order $(d-2)/2$. It turns out that Theorem~\ref{th:Schoenberg} is equivalent to the following result. Put $G=\mathrm{O}(d)$, $U=\mathbb{R}^1$ and $\theta^U(g)=1$.

\begin{theorem}
Equation~\emph{\eqref{eq:5}} establishes a one-to-one correspondence between the class of two-point correlation tensors of homogeneous and isotropic random fields and the class of finite Borel measures $\mu$ on the set $[0,\infty)$.
\end{theorem}

See the proof of equivalence of the above Theorems in \cite{MR2058259}. This result clearly shows a link between the theory of homogeneous and isotropic random fields and the theories of positive-definite kernels and of special functions.

\begin{remark}\label{rem:1}
The spectral representations of type \eqref{eq:5} are not unique. Indeed, one cam multiply the function under the integral by a positive constant and divide the measure $\Phi$ by the same constant. In Equation \eqref{eq:5}, the integrand is equal to $1$ at the point $0$ in the following sense:
\[
\lim_{u\downarrow 0}2^{(d-2)/2}\Gamma(d/2)\frac{J_{(d-2)/2}(u)}
{u^{(d-2)/2}}=1.
\]
\end{remark}

In 1938, Theodore von K\'{a}rm\'{a}n\index{von K\'{a}rm\'{a}n, Theodore} and Leslie Howarth\index{Howarth, Leslie} \cite{Karman1938} provided a physical proof of the following result. Let $G=\mathrm{O}(3)$, $U=\mathbb{R}^3$ and $\theta^U(g)=g$, see Example~\ref{ex:continuum}, the case of a turbulent fluid. Denote by $\mathbf{v}(\mathbf{x})$ the velocity of the fluid. Let $\mathbf{x}$ and $\mathbf{y}$ be two vectors in $\mathbb{R}^3$, denote $\mathbf{r}=\mathbf{y}-\mathbf{x}$, and $r=\|\mathbf{r}\|$. Let $v_l(\mathbf{x})$ be the projection of the vector $\mathbf{v}(\mathbf{x})$ in the direction $\mathbf{r}$, and let $v_k(\mathbf{x})$ be the projection of the same vector in a direction perpendicular to $\mathbf{r}$. The two-point correlation tensor of the $\mathbb{R}^1$-valued homogeneous and isotropic random field $v_l(\mathbf{x})$,
\[
B_{ll}(r)=\mathsf{E}[v_l(\mathbf{x})v_l(\mathbf{x}+\mathbf{r})]
\]
is called the \emph{longitudinal correlation function}\index{correlation function!longitudinal} of the random field $\mathbf{v}(\mathbf{x})$. Similarly, the two-point correlation tensor of the random field $v_k(\mathbf{x})$,
\[
B_{kk}(r)=\mathsf{E}[v_k(\mathbf{x})v_k(\mathbf{x}+\mathbf{r})]
\]
is called the \emph{transversal correlation function}.\index{correlation function!transversal} Introduce the matrix-valued functions
\begin{equation}\label{eq:27}
L^1_{ij}(\mathbf{r})=\delta_{ij},\qquad L^2_{ij}(\mathbf{r})=r_jr_j.
\end{equation}
The two-point correlation tensor of the random field $\mathbf{v}(\mathbf{x})$ has the form
\begin{equation}\label{eq:6}
\langle v_i(\mathbf{x}),v_j(\mathbf{x}+\mathbf{r})\rangle=
[B_{ll}(r)-B_{kk}(r)]\frac{L^2_{ij}(\mathbf{r})}{r^2}+B_{kk}(r)L^1_{ij}(\mathbf{r}).
\end{equation}

In 1940, Howard Percy Robertson\index{Robertson, Howard Percy} \cite{MR0001702} gave a mathematical proof of the above result, using the invariant theory. Thus, a link between the theory of homogeneous and isotropic random fields and the invariant theory was established.

Equation~\eqref{eq:6} gives only \emph{necessary} conditions for a function to be the two-point correlation tensor of the velocity of a turbulent fluid. In 1948, Akiva Moiseevich Yaglom\index{Yaglom, Akiva Moiseevich} \cite{MR0033702} found the necessary and sufficient conditions for the case of $d=3$. This result was proved independently in 1952 by Jos\'{e} Enrique Moyal\index{Moyal, Jos\'{e} Enrique} in \cite{MR49014}. We formulate their result for arbitrary $d$ proved by Yaglom in \cite{MR0094844}. Equations
\begin{equation}\label{eq:28}
\begin{aligned}
B_{ll}(r)&=2^{(d-2)/2}\Gamma(d/2)\left[\int_{0}^{\infty}\left(\frac{J_{d/2}(\lambda r)}{(\lambda r)^{d/2}}-\frac{J_{(d+2)/2}(\lambda r)}{(\lambda r)^{(d-2)/2}}\right)\,\mathrm{d}\Phi_1(\lambda)\right.\\
&\quad+\left.(d-1)\int_{0}^{\infty}\frac{J_{d/2}(\lambda r)}{(\lambda r)^{d/2}}\,\mathrm{d}\Phi_2(\lambda)\right],\\
B_{kk}(r)&=2^{(d-2)/2}\Gamma(d/2)\left[\int_{0}^{\infty}\frac{J_{d/2}(\lambda r)}{(\lambda r)^{d/2}}\,\mathrm{d}\Phi_1(\lambda)\right.\\
&\quad+\left.\int_{0}^{\infty}\left(\frac{J_{(d-2)/2}(\lambda r)}{(\lambda r)^{(d-2)/2}}-\frac{J_{d/2}(\lambda r)}{(\lambda r)^{d/2}}\right)\,\mathrm{d}\Phi_2(\lambda)\right]
\end{aligned}
\end{equation}
establish a one-to-one correspondence between the class of longitudinal and transverse correlation functions of $\mathbb{R}^d$-valued homogeneous and isotropic random fields and the class of pairs $(\Phi_1,\Phi_2)$ of finite Borel measures on $[0,\infty)$ satisfying the condition
\begin{equation}\label{eq:25}
\Phi_1(\{0\})=\Phi_2(\{0\}).
\end{equation}

In 1961, A.M.~Yaglom \cite{MR0146880} and independently Mykha\u{\i}lo \u{I}osypovych Yadrenko\index{Yadrenko,Mykha\u{\i}lo \u{I}osypovych} in his unpublished PhD thesis proved the following result.

\begin{theorem}[Yadrenko--Yaglom]\label{th:YY}
A centred $\mathbb{R}^1$-valued homogeneous and isotropic random field has the form
\begin{equation}\label{eq:8}
\begin{aligned}
T(r,\theta_1,\dots,\theta_{d-2},\varphi)&=\left(2^{d-1}\Gamma(d/2)\pi^{d/2}\right)^{1/2}
\sum_{\ell=0}^{\infty}\sum_{m=1}^{h(\ell,d)}
Y^m_{\ell}(\theta_1,\dots,\theta_{d-2},\varphi)\\
&\quad\times\int_{0}^{\infty}\frac{J_{\ell+(d-2)/2}(\lambda r)}{(\lambda r)^{(d-2)/2}}\,\mathrm{d}Z^m_{\ell}(\lambda).
\end{aligned}
\end{equation}
\end{theorem}

In Theorem~\ref{th:YY}, the symbols $(r,\theta_1,\dots,\theta_{d-2},\varphi)$ denote the spherical coordinates in $\mathbb{R}^d$, the functions $Y^m_{\ell}(\theta_1,\dots,\theta_{d-2},\varphi)$ are \emph{real-valued spherical harmonics},\index{spherical harmonics!real-valued} the number
\[
h(\ell,d)=(2\ell+d-2)\frac{(\ell+d-3)!}{(d-2)!\ell!}
\]
is the total number of spherical harmonics of degree~$\ell$ on the centred unit sphere $S^{d-1}$, and $Z^m_{\ell}$ are Borel measures on $[0,\infty)$ taking values in the Hilbert space $L^2_0(\Omega)$ of centred real-valued random variables with finite variance on the probability space $(\Omega,\mathfrak{F},\mathsf{P})$ with control measure $\Phi$ given in Equation~\eqref{eq:5}, that is,
\[
\mathsf{E}[Z^m_{\ell}(A_1)]=0,\qquad
\mathsf{E}[Z^{m_1}_{\ell_1}(A_1)Z^{m_2}_{\ell_2}(A_2)]
=\delta_{\ell_1\ell_2}\delta_{m_1m_2}\Phi(A_1\cap A_2)
\]
for arbitrary Borel subsets $A_1$ and $A_2$ of the set $[0,\infty)$. The real-valued spherical harmonics $Y^m_{\ell}$ are not so known as their complex-valued cousins $Y_{\ell,m}$, see \cite{MR698780}, \cite[Subsection~A.4.3]{MR2884225}, or \cite[Section~2.10]{MR3930601}. Our notation follows \cite{MR2723248}, which is nowadays a de facto standard.

In 1964, Victor Aleksandrovich Lomakin\index{Lomakin, Victor Aleksandrovich} \cite{MR0163467} provided a physical proof of the following result. Let $G=\mathrm{O}(3)$, $U=\mathsf{S}^2(\mathbb{R}^3)$ and $\theta^U(g)=\mathsf{S}^2(g)$, see Example~\ref{ex:continuum}, the case of a strain tensor of a random deformable body. Denote by $T(\mathbf{x})$ the strain tensor of the body. To simplify notation, introduce the \emph{Ogden tensors}\index{Ogden tensor} after \cite{MR342001}. For a nonnegative integer $\nu$, the Ogden tensor $I$ of rank $2\nu+2$ is defined inductively by
\[
\begin{aligned}
I_{ij}&=\delta_{ij},\qquad I_{ijkl}=\frac{1}{2}(\delta_{ik}\delta_{jl}
+\delta_{il}\delta_{jk}),\\
I_{i_1\cdots i_{2\nu+2}}&=\nu^{-1}(I_{i_1pi_3i_4}I_{pi_2i_5\cdots i_{2\nu+2}}+\cdots+I_{i_1pi_{2\nu+1}i_{2\nu+2}}I_{pi_2\cdots i_{2\nu}}),
\end{aligned}
\]
where there is a summation over $p$. The two-point correlation tensor of the random field $T(\mathbf{x})$ has the form
\begin{equation}\label{eq:31}
\begin{aligned}
\langle T_{ij}(\mathbf{x}),T_{kl}(\mathbf{y})\rangle&=P_4(r)L^1_{ijkl}(\mathbf{r})
+P_6(r)L^2_{ijkl}(\mathbf{r})+[P_5(r)-P_6(r)]L^3_{ijkl}(\mathbf{r})\\
&\quad+[P_3(r)-P_4(r)]L^4_{ijkl}(\mathbf{r})\\'
&\quad+[P_1(r)+P_2(r)-2P_3(r)-4P_5(r)]L^5_{ijkl}(\mathbf{r})
\end{aligned}
\end{equation}
with $P_4(r)+2P_6(r)-P_2(r)=0$, where
\[
\begin{aligned}
L^1_{ijkl}(\mathbf{r})&=I^0_{ij}I^0_{kl}, & L^2_{ijkl}(\mathbf{r})&=2I_{ijkl}, \\
L^3_{ijkl}(\mathbf{r})&=r_jr_kI_{il}+r_ir_lI_{jk}+r_ir_kI_{jl}
+r_jr_lI_{ik}, & L^4_{ijkl}(\mathbf{r})&=r_jr_jI_{kl}+r_kr_lI_{ij},\\
L^5_{ijkl}(\mathbf{r})&=r_ir_jr_kr_l.
\end{aligned}
\]

In 1965, V.A. Lomakin \cite{LOMAKIN19651048} found the general form of the two-point correlation tensor of the random field $T_{ijkl}(\mathbf{x})$ that corresponds to the case of an elasticity tensor of a deformable body, see Example~\ref{ex:continuum}. His formula includes $15$~terms and will not be reproduced here.

In 2014--2020 the authors of the survey published a series of papers and books \cite{MR4031379,MR3336288,MR3630567,MR3621303,MR3493458,MR3732655,MR3903345,MR4200126,MR3327730} in collaboration with colleagues, where they extended the results described above. We describe the above extensions in Sections~\ref{sec:Euclidean} and \ref{sec:continuum}.

The Cosmic Microwave Background was discovered serendipitously in 1964--1965 by Arno Allan Penzias\index{Penzias, Arno Allan} and Robert Woodrow Wilson,\index{Wilson, Robert Woodrow} see \cite{1965ApJ}. As the authors wrote
\begin{quote}
This excess temperature is, within the limits of our observations, isotropic, unpolarised, and free from seasonal variations (July, 1964--April, 1965).
\end{quote}

Later on, with the help of more advanced instruments, it was found that the CMB deviates from isotropy and is polarised. Two mathematical models of the CMB in the form of a random field were constructed in 1997 independently by two research groups, see \cite{PhysRevD.55.7368} and \cite{Zaldarriaga:1996xe}. More mathematical studies appeared in \cite{MR3170229}, \cite{MR2737761}, and \cite{MR2884225}. We describe these results in Sections~\ref{sec:sections} and \ref{sec:cosmology}.

\section{Random fields defined on an Euclidean space}\label{sec:Euclidean}

\subsection{The formulation of a problem}\label{sub:formulation}

Let $d$ be a positive integer. Let $G$ be a closed subgroup of the \emph{orthogonal group}\index{group!orthogonal} $\mathrm{O}(d)$ of $d\times d$ orthogonal matrices acting on a real finite-dimensional linear space $U$ with an inner product $(\cdot,\cdot)$ by an orthogonal representation with translation $\theta^U(g)$. Let $T(\mathbf{x})$ be a $U$-valued mean-square continuous random field on $\mathbb{R}^d$ which is homogeneous according to Equation~\eqref{eq:7} and isotropic according to Equation~\eqref{eq:3}. We would like to find the general form of its one- and two-point correlation tensors similar to \eqref{eq:5} and the representation of the field itself similar to Equation~\eqref{eq:8}.

Before we continue, the following question will be considered. Which fields of the above described class are most interesting for continuum physics?

First, we give an answer to this question, using a particular example. In Example~\ref{ex:continuum} we have seen that the representation with translation $\theta^{\mathsf{S}^2(\mathsf{S}^2(\mathbb{R}^3))}$ corresponds to linear elasticity. How many classes of elastic bodies are there?

The mathematically correct answer to this question was given by Sandra Forte and Maurizio Vianello in 1996, see \cite{MR1405284}. Fix a tensor $T$ in the linear space $U=\mathsf{S}^2(\mathsf{S}^2(\mathbb{R}^3))$. Let $K_T$ be group defined by
\[
K_T=\{\,g\in\mathrm{O}(3)\colon g\cdot T=T\,\}.
\]
Mathematicians call $K_T$ the \emph{stationary subgroup}\index{stationary subgroup} of the point $T$. Physicists call it the \emph{material symmetry group}\index{symmetry group!material} of the microstructure attached to a physical medium.

As $g$ runs over $\mathrm{O}(3)$, the point $g\cdot T$ runs over the \emph{orbit}\index{orbit} $O_T$ of the point~$T$:
\[
O_T=\{\,g\cdot T\colon g\in\mathrm{O}(3)\,\}.
\]
It is easy to see that the material symmetry group of a point $g\cdot T\in O_T$ is \emph{conjugate} to $K_T$, that is, it is equal to $gK_Tg^{-1}$. Call two tensors $T_1$ and $T_2$ \emph{equivalent} if their material symmetry groups are conjugate. This relation partitions the space $U$ into equivalence classes called \emph{symmetry classes}.\index{symmetry class} Each symmetry class is completely determined by the conjugacy class $[K]$, that is, the set of all closed subgroups of the group $\mathrm{O}(3)$ that are conjugate to the material symmetry group $K$ of a particular tensor in the symmetry class. Forte and Vianello found $8$ symmetry classes for elasticity, or \emph{elasticity classes}.

For a group~$K$ that defines a symmetry class, define its \emph{fixed point set}\index{fixed point set} or the \emph{linear slice}\index{linear slice} as
\[
U_K=\{\,T\in U\colon g\cdot T=T\quad\text{for all}\quad g\in K\,\}.
\]
The linear slice $U_K$ is the linear space that meets all the orbits of tensors which have \emph{at least} the symmetry class $[K]$. It is also the isotypical subspace of $U$ where the copies of the trivial representation of the group~$K$ act. Moreover, in many cases there exists a group $G$ such that $K$ is a proper subgroup of $G$ but $U_K$ is still an invariant subspace for $G$. The maximal group with this property is the \emph{normaliser}\index{normaliser} of $K$:
\[
N(K)=\{\,g\in\mathrm{O}(3)\colon kKg^{-1}=K\,\}.
\]
The normaliser $N(K)$ is the maximal physical symmetry group that corresponds to the material symmetry group $K$.

The most interesting case for continuum physics is as follows: $d=2$ (the so called \emph{plane problems}\index{plane problem}) or $d=3$ (\emph{space problems}),\index{space problem} $U$ is the linear slice of a symmetry class $[K]$, and $G$ is a closed subgroup of the normaliser $N(K)$ such that $K$ is a subgroup of $G$. A continuous medium occupies a compact subset $D$ of the linear space $\mathbb{R}^d$. A microstructure is attached to any material point $\mathbf{x}\in D$. The group $K$ is the group of symmetries of the above microstructure. We assume that the medium is homogeneous and this group is the same for all material points. At the macroscopic scale all the details of the microstructure are lost, all what remains is the material symmetry group $K$.

On the other hand, the physical properties of the media are encoded by the \emph{physical symmetry group},\index{symmetry group!physical} the group $G$. Such defined group satisfies the \emph{Curie principle} explained above.

Note that an algorithm for determining all symmetry classes for a given finite-dimensional representation of the group $\mathrm{O}(3)$, is described in \cite{MR4045367,Olive2013,MR3208052}.

In all examples that follow below, $G$ is a closed subgroup of the group $\mathrm{O}(d)$ satisfying the condition $K\subseteq G\subseteq N(K)$, where the group $K$ defines a symmetry class, and $U$ is the fixed point set of the symmetry class $[K]$, where an orthogonal representation of the group $G$ acts.

\subsection{Principles for finding a solution}\label{sub:principles}

First of all, we can easily find the one-point correlation tensor $\langle T(\mathbf{x})\rangle$ of a homogeneous and isotropic random field. It follows from the first equation in \eqref{eq:7} that this quantity does not depend on $\mathbf{x}$.  Denote it just by $T$. It follows from Equation~\eqref{eq:4} that $\theta^U(g)T=T$ for all $g\in G$, which is equivalent to the following: \emph{if the isotypical space of $U$ that corresponds to the trivial irreducible representation of $G$, has positive dimension, then $T$ is an arbitrary element of the above isotypical space, otherwise $T=0$}. In what follows, we always consider the centred random field $T(\mathbf{x})-T$ and denote it again by $T(\mathbf{x})$.

The general form of the two-point correlation tensor of the field $T(\mathbf{x})$ clearly depends on the choice of the basis in the space $U$. Therefore, the above choice should be included to the proof. We divide the proof into two parts. The first, coordinate-free part, will be sketched in Subsections~\ref{sub:homogeneous}--\ref{sub:general}. Equations of this part will be written in a simple, coordinate-free form. The price one has to pay for such a simplification is that while deducing the result, we introduce a couple of notions which are not always familiar to some readers.

It is possible to write down the second, coordinate part of the proof for a general group and representation. This leads to a Bacchanalia of indices, see \cite[Theorem~0]{MR3621303}, \cite[Theorem~13]{MR3930601}. Instead, we consider examples of introducing coordinates for different representations of the group $\mathrm{O}(3)$ in Subsection~\ref{sub:examples}.

Finally, the problem will be sold in such a way: first we describe all homogeneous fields and then find the conditions under which a particular homogeneous field is isotropic.

\subsection{The description of homogeneous fields}\label{sub:homogeneous}

Let the group $G$ act in $U$ by an orthogonal representation. Consider the set $\mathbb{C}$ of complex numbers as a two-dimensional real linear spaces and define $cU$ as the tensor product $\mathbb{C}\otimes_{\mathbb{R}}U$. There is a unique representation of the group $G$ in the space $cU$ such that $g\cdot(z\otimes u)=z\otimes(g\cdot u)$ for all $g\in G$, $z\in\mathbb{C}$ and $u\in U$. Moreover, there exists a unique positive-definite Hermitian form $H(v_1,v_2)$ on $cU$ such that
\[
H(z_1\otimes u_1,z_2\otimes u_2)=z_1^*z_2(u_1,u_2),
\]
and the above form is $G$-invariant. In other words, the constructed representation is a \emph{unitary}\index{representation!unitary} one.

The set $U$ can be treated as a subset of $cU$ by identifying $u\in U$ with $1\otimes u\in cU$. Thus, a $U$-valued random field $T(x)$ can be treated as a $cU$-valued one. It is easy to see that a mean-square continuous $U$-valued random field is a mean-square continuous  $cU$-valued random field in the obvious sense.

Denote by $\id$ the identical map in $cU$. There exists a unique map $j_{cU}\colon cU\to cU$ with the following properties: $j_{cU}(zv)=z^*j(v)$ for all $z\in\mathbb{C}$ and all $v\in cU$, $j_{cU}^2=\id$, and $j_{cU}(z\otimes u)=z^*\otimes u$. This map is a \emph{real structure}\index{real structure} on $cU$, a coordinate-free form of the complex conjugation.

Define the two-point correlation tensor of the \emph{$cU$-valued} random field $T(\mathbf{x})$ by
\[
\langle T(\mathbf{x}),T(\mathbf{y})\rangle=\mathsf{E}[j_{cU}T(\mathbf{x})\otimes T(\mathbf{y})].
\]
If, in addition, the \emph{$U$-valued} random field $T(\mathbf{x})$ is homogeneous, then $jT(\mathbf{y})=T(\mathbf{y})$, and the second equation in \eqref{eq:7} holds true for the \emph{$cU$-valued} random field $T(\mathbf{x})$. That is, the $cU$-valued random field $T(\mathbf{x})$ is homogeneous. The two-point correlation tensor of such a field is described by the following result.

\begin{theorem}
Equation
\begin{equation}\label{eq:11}
\langle T(\mathbf{x}),T(\mathbf{y})\rangle=\int_{\hat{\mathbb{R}}^d}
\mathrm{e}^{\mathrm{i}(\mathbf{k},\mathbf{r})}\,\mathrm{d}F(\mathbf{k})
\end{equation}
establishes a one-to-one correspondence between the class of $cU$-valued  mean-square continuous and homogeneous random fields on $\mathbb{R}^d$ and the class of measures on the \emph{wavenumber domain} $\hat{\mathbb{R}}^d$ taking values in the cone of Hermitian nonnegative-definite linear operators on $cU$.
\end{theorem}

Under what conditions does a $cU$-valued homogeneous random field $T(\mathbf{x})$ take values in $U$? Let $\Hom_{\mathbb{C}}(cU,cU)$ be the set of $\mathbb{C}$-linear maps from $cU$ to itself. The complex linear space $\Hom_{\mathbb{C}}(cU,cU)$ admits a real structure $j$ given by
\begin{equation}\label{eq:13}
jF=j_{cU}Fj^{-1}_{cU},\qquad F\in\Hom_{\mathbb{C}}(cU,cU).
\end{equation}

\begin{lemma}\label{lem:reality}
A $cU$-valued homogeneous random field $T(\mathbf{x})$ takes values in $U$ if and only if for any Borel set $A$ in the wavenumber domain $\hat{\mathbb{R}}^d$ we have
\begin{equation}\label{eq:12}
F(-A)=jF(A),
\end{equation}
where $-A=\{\,-\mathbf{k}\colon\mathbf{k}\in A\,\}$.
\end{lemma}

\begin{proof}
Indeed, let $T(\mathbf{x})$ takes values in $U$. Then the random field $j_{cU}T(\mathbf{x})$ has the same two-point correlation tensor as $T(\mathbf{x})$ has:
\[
\langle j_{cU}T(\mathbf{x}),j_{cU}T(\mathbf{y})\rangle=\int_{\hat{\mathbb{R}}^d}
\mathrm{e}^{\mathrm{i}(\mathbf{k},\mathbf{r})}\,\mathrm{d}F(\mathbf{k}).
\]
On the one hand, by definition of the two-point correlation tensor, the left hand side is
\begin{equation}\label{eq:9}
\begin{aligned}
\langle j_{cU}T(\mathbf{x}),j_{cU}T(\mathbf{y})\rangle&=\mathsf{E}[j_{cU}j_{cU}T(\mathbf{x})\otimes j_{cU}T(\mathbf{y})]=\mathsf{E}[T(\mathbf{x})\otimes j_{cU}T(\mathbf{y})]\\
&=\int_{\hat{\mathbb{R}}^d}
\mathrm{e}^{\mathrm{i}(\mathbf{k},-\mathbf{r})}\,\mathrm{d}F(\mathbf{k})
=\int_{\hat{\mathbb{R}}^d}
\mathrm{e}^{\mathrm{i}(\mathbf{k},\mathbf{r})}\,\mathrm{d}F(-\mathbf{k}).
\end{aligned}
\end{equation}
On the other hand, the left hand side is
\begin{equation}\label{eq:10}
\begin{aligned}
\langle j_{cU}T(\mathbf{x}),j_{cU}T(\mathbf{y})\rangle&=j\langle j_{cU}T(\mathbf{x}),j_{cU}T(\mathbf{y})\rangle\\
&=j\int_{\hat{\mathbb{R}}^d}
\mathrm{e}^{\mathrm{i}(\mathbf{k},\mathbf{r})}\,\mathrm{d}F(\mathbf{k})
=\int_{\hat{\mathbb{R}}^d}
\mathrm{e}^{\mathrm{i}(\mathbf{k},\mathbf{r})}\,\mathrm{d}jF(\mathbf{k}).
\end{aligned}
\end{equation}

Conversely, let $F(-A)=jF(A)$. Then the right hand sides of Equations~\eqref{eq:9} and \eqref{eq:10} are equal, and we obtain
\[
\langle
j_{cU}T(\mathbf{x}),j_{cU}T(\mathbf{y})\rangle=j\langle j_{cU}T(\mathbf{x}),j_{cU}T(\mathbf{y})\rangle,
\]
which means that the random field $j_{cU}T(\mathbf{x})$ takes values in $U$, so $T(\mathbf{x})$ does.
\end{proof}

\subsection{A preliminary description of homogeneous and iso\-tro\-pic fields}\label{sub:preliminary}

The linear space $\Hom_{\mathbb{C}}(cU,cU)$ carries a representation of the group $G$ given by
\[
g\cdot F=\theta^{cU}(g)F(\theta^{cU})^{-1}(g),\qquad F\in\Hom_{\mathbb{C}}(cU,cU).
\]
It is trivial to check that $(g\cdot F)^*=g\cdot F^*$, where the index $*$ denotes the conjugation of a linear operator. Split the space $\Hom_{\mathbb{C}}(cU,cU)$ into the linear spaces of Hermitian and skew-Hermitian linear operators. By \cite[Section~3.3]{MR0252560}, the above spaces carry the equivalent representations of $G$ intertwined by the multiplication by $\mathrm{i}$. Denote them by $X$. We have $\Hom_{\mathbb{C}}(cU,cU)=cX$. By \cite[Section~3.9 (i)]{MR0252560}, one possibility for $X$ is $X=\Hom_{\mathbb{R}}(U,U)$, and by \cite[Corollary~3.28 (i)]{MR0252560}, this possibility is unique. Observe that $\Hom_{\mathbb{R}}(U,U)$ and $U\otimes U$ are equivalent. In what follows, we identify the equivalent orthogonal representations acting in $U\otimes U$ and in the real linear space of Hermitian linear operators in $cU$.

Denote the $d\times d$ identity matrix by $I$. Define the closed subgroup $\tilde{G}$ of the group $\mathrm{O}(d)$ by
\[
\tilde{G}=G\cup\{\,(-I)g\colon g\in G\,\}.
\]
Observe, that if $-I\in G$, then the sets in the left hand side are identical. In this case, put $\tilde{U}=\mathsf{S}^2(U)$. Otherwise, if $-I\notin G$, the above sets do not intersect. In that case, put $\tilde{U}=U\otimes U$. Note that the linear space $\mathsf{S}^2(U)$ of symmetric rank~$2$ tensors over $U$ and the linear space $\Lambda^2(U)$ of skew-symmetric rank~$2$ tensors over $U$ are complementary invariant subspaces of the representation $U\otimes U$ of the group~$G$.

In the first case, $\tilde{U}$ is an invariant subspace of the representation $U\otimes U$ and defines the orthogonal representation of $\tilde{G}$. To define an orthogonal representation of $\tilde{G}$ in the second case, put
\[
(-I)g\cdot u=
\begin{cases}
  g\cdot u, & \mbox{if } u\in\mathsf{S}^2(U), \\
  -g\cdot u, & \mbox{if } u\in\Lambda^2(U),
\end{cases}
\]
and extend this action to all of $U\otimes U$ by linearity.

\begin{lemma}
A random field $T(\mathbf{x})$ is $U$-valued, homogeneous and isotropic if and only if the measure $F$ in Equation~\emph{\eqref{eq:11}} takes values in the intersection of the space $\tilde{U}$ with the cone of Hermitian nonnegative-definite linear operators on $cU$, and
\begin{equation}\label{eq:15}
F(\tilde{g}A)=\theta^{\tilde{U}}(\tilde{g})F(A),\qquad\tilde{g}\in\tilde{G}
\end{equation}
for all Borel subsets $A$ of the wavenumber domain $\hat{\mathbb{R}}^d$.
\end{lemma}

\begin{proof}[Sketch of a proof]
Denote the group $\{I,-I\}$ by $Z^c_2$, this notation is taken from \cite{MR950168}. The main idea is that Equation~\eqref{eq:12} in Lemma~\ref{lem:reality} is equivalent to the following condition:
\begin{equation}\label{eq:14}
F(gA)=g\cdot F(A),\qquad g\in Z^c_2,
\end{equation}
where the element $-I\in Z^c_2$ multiplies by $\pm 1$ the elements of the $\pm 1$ eigenspace of the real structure $j$ given by Equation~\eqref{eq:13}.

On the other hand, it is easy to prove that Equation~\eqref{eq:3} in Definition~\eqref{def:isotropic} is equivalent to the condition
\[
F(gA)=(\theta^U\otimes\theta^U)(g)\cdot F(A),\qquad g\in G.
\]

If $-I\in G$, then $-I$ acts trivially in $U\otimes U$. This action coincides with that in the right hand side of condition~\eqref{eq:14} only on the intersection of $U\otimes U$ with the $+1$ eigenspace of the real structure $j$. This intersection is equal to $\tilde{U}=\mathsf{S}^2(U)$, and the above two conditions are equivalent to single condition~\eqref{eq:15}.

Otherwise, if $-I\notin G$, we constructed $\tilde{G}$ as the minimal extension of $G$ by $-I$, and the representation $\tilde{U}$ is constructed in such a way that the above two conditions agree on all of $\tilde{U}=U\otimes U$.

The details of proof can be found in \cite[Section~3.3, Lemma~1]{MR3930601} and in \cite[Lemma~4.1]{MR4200126}.
\end{proof}

Consider a measure $\mu$ given by $\mu(A)=\tr F(A)$, where $\tr$ denotes the trace of an operator. It is well-known that the measure $F$ is absolutely continuous with respect to $\mu$, that is, $\mu(A)=0$ implies $F(A)=0$, and the Radon--Nikodym derivative $f(\mathbf{k})=\mathrm{d}F/\mathrm{d}\mu(\mathbf{k})$ is a measurable function on the wavenumber domain $\hat{\mathbb{R}}^d$ that takes values in the convex compact set of Hermitian nonnegative-definite operators in $cU$ with unit trace. For a proof, see, for instance, \cite[Chapter~5, Theorem~1.1]{MR0222718}.

In new terms, condition~\eqref{eq:15} becomes
\begin{equation}\label{eq:16}
\mu(\tilde{g}A)=\mu(A),\qquad f(\tilde{g}\mathbf{k})=\theta^{\tilde{U}}(\tilde{g})f(\mathbf{k})
\end{equation}
for all $\tilde{g}\in\tilde{G}$.

\subsection{The general form of the two-point correlation tensor}\label{sub:general}

We find all measures $\mu$ satisfying the first equation in \eqref{eq:16}. In general, the action of a topological group $G$ can be very sophisticated, see \cite{MR0413144,MR0379739,MR0201557}. However, in the case when a compact Lie group like $\tilde{G}$ acts in a real finite-dimensional linear space by an orthogonal representation, its action can be easily described, see \cite{MR1738431}.

Let $\hat{\mathbb{R}}^d/\tilde{G}$ be the set of orbits for the action of the group $\tilde{G}$ on $\hat{\mathbb{R}}^d$ by matrix-vector multiplication. For $\mathbf{k}\in\hat{\mathbb{R}}^d$, let $\tilde{G}_{\mathbf{k}}$ be the stationary subgroup of $\mathbf{k}$. We say that the orbits $\tilde{G}\mathbf{k}_1$ and $\tilde{G}\mathbf{k}_2$ are of the same type, or $\tilde{G}\mathbf{k}_1\sim\tilde{G}\mathbf{k}_2$, if and only if $\tilde{G}_{\mathbf{k}_1}$ is conjugate to $\tilde{G}_{\mathbf{k}_2}$ within $\tilde{G}$. There are only finitely many, say $M$, distinct types $(\hat{\mathbb{R}}^d/\tilde{G})_m$, $0\leq m\leq M-1$, called the \emph{strata}.\index{stratum} Moreover, the closure of each stratum is the union of that stratum and strata of smaller dimension.

We say that $\tilde{G}\mathbf{k}_2$ dominates $\tilde{G}\mathbf{k}_1$, or $\tilde{G}\mathbf{k}_1\preceq\tilde{G}\mathbf{k}_2$, if and only if $\tilde{G}_{\mathbf{k}_1}$ is conjugate to a subgroup of $\tilde{G}_{\mathbf{k}_2}$ within $\tilde{G}$. This relation defines a partial ordering in the set of strata. There exists a unique maximal element in the above set. We call this type the \emph{principal stratum}\index{stratum!principal} and denote by $(\hat{\mathbb{R}}^d/\tilde{G})_0$.

Fix a choice of stationary subgroups $\{\,\tilde{G}_m\colon 0\leq m\leq M-1\,\}$ of points in different strata. For simplicity, assume that there exist a chart $\varphi_m$ in the manifold $\tilde{G}/\tilde{G}_m$ and a chart $\lambda_m$ in the manifold $(\hat{\mathbb{R}}^d/\tilde{G})_m$ with dense domains. There exists a unique probabilistic $\tilde{G}$-invariant measure on the orbit of each point $\lambda_m\in(\hat{\mathbb{R}}^d/\tilde{G})_m$, call it $\mathrm{d}\varphi_m$.

According to \cite[Chapter~VII, \S~2, Proposition~4]{MR2098271}, for every finite measure $\mu$ satisfying the first equation in \eqref{eq:16}, there is a unique measure defined on the Borel subsets of the space $(\hat{\mathbb{R}}^d/\tilde{G})_m$, call it $\mathrm{d}\Phi_m(\lambda_m)$, such that the restriction of the measure $\mu$ to the symmetry class $(\tilde{G}/\tilde{G}_m)\times(\hat{\mathbb{R}}^d/\tilde{G})_m$ has the form
\[
\mathrm{d}\mu_m=\mathrm{d}\varphi_m\,\mathrm{d}\Phi_m(\lambda_m).
\]
Conversely, a measure given by the above equation obviously satisfies the first equation in \eqref{eq:16}.

We find all functions $f$ satisfying the second equation in \eqref{eq:16}. Obviously, it is enough to find the restrictions of the function $f$ to the sets $(\hat{\mathbb{R}}^d/\tilde{G})_m$, $0\leq m\leq M-1$. For such a restriction, we have
\begin{equation}\label{eq:17}
f(\tilde{g}\lambda_m)=\theta^{\tilde{U}}(\tilde{g})f(\lambda_m),\qquad\tilde{g}\in\tilde{G}.
\end{equation}
If $\tilde{g}\in\tilde{G}_m$, then $\tilde{g}\lambda_m=\lambda_m$ and
\[
f(\lambda_m)=\theta^{\tilde{U}}(\tilde{g})f(\lambda_m),\qquad\tilde{g}\in\tilde{G}_m.
\]
This condition means that $f(\lambda_m)$ belongs to the isotypical subspace of the linear space $\tilde{U}$, in which the direct sum of several copies of the trivial representation of the group $\tilde{G}_m$ acts. Call this space $\tilde{U}_m$ and observe that it is a linear slice for the group $\tilde{G}_m$. The intersection of the linear slice $\tilde{U}_m$ with the convex compact set of Hermitian nonnegative-definite operators in $cU$ with unit trace is again a convex compact set, call it $\mathcal{C}_m$. The restriction of the function $f$ to the set $(\hat{\mathbb{R}}^d/\tilde{G})_m$ is a measurable function taking values in $\mathcal{C}_m$.

Conversely, let $f(\lambda_m)\colon(\hat{\mathbb{R}}^d/\tilde{G})_m\to\mathcal{C}_m$ be an arbitrary measurable function. Extend it to the symmetry class $(\tilde{G}/\tilde{G}_m)\times(\hat{\mathbb{R}}^d/\tilde{G})_m$ with the help of Equation~\eqref{eq:17}. The result obviously satisfies the second equation in \eqref{eq:16}.

Equation~\eqref{eq:11} takes the form
\begin{equation}\label{eq:18}
\langle T(\mathbf{x}),T(\mathbf{y})\rangle=\sum_{m=0}^{M-1}
\int_{(\hat{\mathbb{R}}^d/\tilde{G})_m}\int_{\tilde{G}/\tilde{G}_m}
\mathrm{e}^{\mathrm{i}(\mathbf{k},\mathbf{r})}
\theta^{\tilde{U}}(\tilde{g})f(\lambda_m)\,\mathrm{d}\varphi_m\,\mathrm{d}\Phi_m(\lambda_m),
\end{equation}
where a point $\mathbf{k}\in(\tilde{G}/\tilde{G}_m)\times(\hat{\mathbb{R}}^d/\tilde{G})_m\subset\hat{\mathbb{R}}^d$ has coordinates $(\varphi_m,\lambda_m)$, and where $\tilde{g}$ is an arbitrary element of the group $\tilde{G}_m$ with $\tilde{g}\lambda_m=\mathbf{k}$.

\subsection{Examples}\label{sub:examples}

Put $d=3$ and $G=\tilde{G}=\mathrm{O}(3)$. The group $\tilde{G}$ acts in the wavenumber domain $\hat{\mathbb{R}}^3$ by matrix-vector multiplication. The set of orbits, $\hat{\mathbb{R}}^3/\mathrm{O}(3)$, has $M=2$ strata:
\[
(\hat{\mathbb{R}}^3/\mathrm{O}(3))_0=(0,\infty),\qquad (\hat{\mathbb{R}}^3/\mathrm{O}(3))_1=\{0\}.
\]
Fix a choice of stationary subgroups: $\tilde{G}_0=\mathrm{O}(2)$, the stationary subgroup of the point $\mathbf{k}=(0,0,1)^{\top}$ and $\tilde{G}_1=\mathrm{O}(3)$. The chart $\varphi_0$ in the manifold $\mathrm{O}(3)/\mathrm{O}(2)=S^2$ is the angular spherical coordinates on the two-dimensional sphere $S^2$ \emph{in the wavenumber domain}. Call them $(\hat{\theta},\hat{\varphi})$, because we reserve the notation $(\theta,\varphi)$ for the angular spherical coordinates \emph{in the space domain}. The chart $\varphi_1$ maps the singleton $\{0\}$ to itself. The measure $\mathrm{d}\varphi_0$ is $\mathrm{d}\varphi_0=\frac{1}{4\pi}\sin\hat{\theta}\,\mathrm{d}\hat{\theta}\,\mathrm{d}\hat{\varphi}$, the measure $\mathrm{d}\varphi_1$ is the Dirac measure on the singleton $\{0\}$. The chart $\lambda_0$ is the radial distance in the wavenumber domain, denote it just by $\lambda$. In spherical coordinates, the measure $\mu$ can be written in the form
\[
\mathrm{d}\mu(\lambda,\hat{\theta},\hat{\varphi})
=\frac{1}{4\pi}\sin\hat{\theta}\,\mathrm{d}\hat{\theta}\,\mathrm{d}\hat{\varphi}
\,\mathrm{d}\Phi(\lambda).
\]

Examples of homogeneous fields that are isotropic with respect to orthogonal representations of proper subgroups of $\mathrm{O}(3)$ with up to $8$ strata can be found in \cite{MR3930601,MR4200126}. Nevertheless, all examples follow the same $6$-steps scheme.

\begin{enumerate}
  \item Splitting the representation $\tilde{U}$ into irreducible components.
  \item Introducing coordinates in the linear slices $\tilde{U}_0$, \dots, $\tilde{U}_{M-1}$.
  \item Description of the sets of extreme points of the convex compacta $\mathcal{C}_0$, \dots, $\mathcal{C}_{M-1}$.
  \item Calculating the matrices $f(\lambda_m)$, $0\leq m\leq M$ in the introduced coordinates.
  \item Calculating the inner integrals in Equation~\eqref{eq:18}.
  \item Calculating the spectral expansion of the field.
\end{enumerate}

\begin{example}\label{ex:1}
Let $U=\mathbb{R}^1$ and $\theta^U(g)=1$. The first four steps are trivial. Indeed, we have $\tilde{U}=\mathrm{S}^2(U)=U$. The linear slice $\tilde{U}_0$ for the group $\tilde{G}_0=\mathrm{O}(2)$ is identical to $\tilde{U}$, and so is $\tilde{U}_1$. The convex compact set of Hermitian linear operators in $cU=\mathbb{C}^1$ with unit trace is the singleton $\{1\}$. Both convex compacta $\mathcal{C}_0$ and $\mathcal{C}_1$ are equal to $\{1\}$. Equation~\eqref{eq:11} takes the form
\begin{equation}\label{eq:20}
\langle T(\mathbf{x}),T(\mathbf{y})\rangle=\frac{1}{4\pi}\int_{0}^{\infty}\int_{S^2}
\mathrm{e}^{\mathrm{i}(\mathbf{k},\mathbf{r})}\sin\hat{\theta}
\,\mathrm{d}\hat{\theta}\,\mathrm{d}\hat{\varphi}\,\mathrm{d}\Phi(\lambda).
\end{equation}

In order to perform the fifth step and calculate the inner integral, use the following general idea. Consider the restriction of the \emph{plane wave}\index{plane wave} $\mathrm{e}^{\mathrm{i}(\mathbf{k},\mathbf{r})}$ to the orbit of a point $\lambda\in(\hat{\mathbb{R}}^3/\mathrm{O}(3))_0$. The Hilbert space of the square-integrable functions on this orbit with respect to the measure $\mathrm{d}\varphi_0$ has an orthonormal basis whose elements are multiples of some matrix entries of the irreducible orthogonal or unitary representations of the group $\tilde{G}$. For a continuous function, the Fourier series with respect to the above basis converges uniformly. In the unitary case, this proposition is known as the \emph{Peter--Weyl Theorem},\index{Peter--Weyl Theorem} in the orthogonal case as the \emph{Fine Structure Theorem},\index{Fine Structure Theorem} see \cite{MR4201900} and Subsection~\ref{sub:structure}.

Let $(r,\theta,\varphi)$ be the spherical coordinates of a point $\mathbf{r}$ in the \emph{space domain}, while $(\lambda,\hat{\theta},\hat{\varphi})$ be those of a point $\mathbf{k}$ in the \emph{wavenumber domain}. In the case of the sphere $S^2$, the Fourier expansion of the plane wave is known as the \emph{Rayleigh expansion}:\index{Rayleigh expansion}
\[
\mathrm{e}^{\mathrm{i}(\mathbf{k},\mathbf{r})}=4\pi\sum_{\ell=0}^{\infty}
\sum_{m=-\ell}^{\ell}\mathrm{i}^{\ell}j_{\ell}(\lambda r)Y^m_{\ell}(\theta,\varphi)Y^m_{\ell}(\hat{\theta},\hat{\varphi}),
\]
where $j_{\ell}(u)=\sqrt{\pi/(2u)}J_{\ell+1/2}(u)$ is the spherical Bessel function.

Insert the Rayleigh expansion into Equation~\eqref{eq:20}. One may calculate the inner integral term by term because of uniform convergence of the above expansion. As a result, all terms with $\ell\neq 0$ disappear, and we obtain
\[
T(\mathbf{x}),T(\mathbf{y})\rangle=\int_{0}^{\infty}\int_{S^2}
j_0(\lambda r)Y^0_0(\theta,\varphi)Y^0_0(\hat{\theta},\hat{\varphi})\sin\hat{\theta}
\,\mathrm{d}\hat{\theta}\,\mathrm{d}\hat{\varphi}\,\mathrm{d}\Phi(\lambda).
\]
We have $Y^0_0(\theta,\varphi)=(4\pi)^{-1/2}$ and the inner integral is equal to $\sqrt{4\pi}$. The two-point correlation function becomes
\[
T(\mathbf{x}),T(\mathbf{y})\rangle=\int_{0}^{\infty}
j_0(\lambda r)\,\mathrm{d}\Phi(\lambda)=\sqrt{2/\pi}\int_{0}^{\infty}
\frac{J_{1/2}(\lambda r)}{\sqrt{\lambda r}}\,\mathrm{d}\Phi(\lambda),
\]
which is the same, as Equation~\eqref{eq:5} in the case of $d=3$.

In the case of arbitrary positive integer $d$, the calculations are similar and can be left to the reader. The uniformly convergent Rayleigh expansion for arbitrary $d$ has the form
\begin{equation}\label{eq:22}
\begin{aligned}
\mathrm{e}^{\mathrm{i}(\mathbf{k},\mathbf{r})}&=(2\pi)^{d/2}\sum_{\ell=0}^{\infty}
\mathrm{i}^{\ell}\frac{J_{\ell+(d-2)/2}(\lambda r)}{(\lambda r)^{(d-2)/2}}\\
&\quad\times\sum_{m=1}^{h(\ell,d)}Y^m_{\ell}(\theta_1,\dots,\theta_{d-2},\varphi)
Y^m_{\ell}(\hat{\theta}_1,\dots,\hat{\theta}_{d-2},\hat{\varphi}),
\end{aligned}
\end{equation}
see \cite{MR1220225}.

How to perform the last step and prove spectral expansions similar to \eqref{eq:8}? One uses a technical tool known as the \emph{Karhunen theorem},\index{Karhunen's theorem} \cite{MR23013}. Let $X$ be a set, and let $T(x)$ be a centred second order random field on $X$ taking values in a complex finite-dimensional linear space $V$ with norm $\|\cdot\|$ and real structure $j$. Let $\Lambda$ be another set, $\mathfrak{L}$ be a $\sigma$-field of subsets of $\Lambda$, $F$ be a measure defined on a measurable space $(\Lambda,\mathfrak{L})$ and taking values in the cone of Hermitian nonnegative-definite linear operators on $V$, $\mu$ be its trace measure, $\mu(A)=\tr F(A)$ for $A\in\mathfrak{L}$. Assume that the two-point correlation function of the random field $T(x)$ can be written in the following form
\begin{equation}\label{eq:21}
\langle T(x),T(y)\rangle=\int_{\Lambda}a(x,\lambda)a(y,\lambda)^*\,
\mathrm{d}F(\lambda),
\end{equation}
where the function $a\colon X\times\Lambda\to\mathbb{C}$ satisfies the condition
\[
\int_{\Lambda}|a(x_0,\lambda)|^2\,\mathrm{d}\mu(\lambda)<\infty,\qquad x_0\in X.
\]
Moreover, assume that the set of finite linear combinations of the functions from the family $\{\,f(x_0,\lambda)\colon x_0\in X\,\}$ is dense in the Hilbert space of $\mu$-square-integrable functions on $\Lambda$.

\begin{theorem}[K.~Karhunen]\label{th:Karhunen}
Under the above conditions, the random field $T(x)$ has the form
\[
T(x)=\int_{\Lambda}a(x,\lambda)\,\mathrm{d}Z(\lambda),
\]
where $Z$ is a measure on $(\Lambda,\mathfrak{L})$ taking values in the Hilbert space of centred $V$-valued random vectors $Y$ with $\mathsf{E}[\|Y\|^2]<\infty$. Moreover, the measure $F$ is the \emph{control measure} of $Z$ in the sense that for all $A$, $B\in\mathfrak{L}$ we have
\[
\mathsf{E}[jZ(A)\otimes Z(B)]=F(A\cap B).
\]
\end{theorem}

Note that Equation~\eqref{eq:5} does not have the form of \eqref{eq:21} and Theorem~\ref{th:Karhunen} cannot be applied. We use another general idea to overcome this difficulty. For a general dimension $d$, Equation~\eqref{eq:20} takes the form
\[
\langle T(\mathbf{x}),T(\mathbf{y})\rangle=\frac{\Gamma(d/2)}{2\pi^{d/2}}\int_{0}^{\infty}\int_{S^{d-1}}
\mathrm{e}^{\mathrm{i}(\mathbf{k},\mathbf{r})}\,\mathrm{d}S\,\mathrm{d}\Phi(\lambda),
\]
where $\mathrm{d}S$ is the Lebesgue measure on the sphere $S^{d-1}$. Recall that $\mathbf{r}=\mathbf{y}-\mathbf{x}$ and write down the integrand as $\mathrm{e}^{\mathrm{i}(\mathbf{k},\mathbf{r})}=\mathrm{e}^{-\mathrm{i}
(\mathbf{k},\mathbf{x})}\mathrm{e}^{\mathrm{i}(\mathbf{k},\mathbf{y})}$. Replace \emph{each} term with its absolutely converging expansion \eqref{eq:22} and calculate the inner integral. We obtain
\[
\begin{aligned}
\langle T(\mathbf{x}),T(\mathbf{y})\rangle&=2^{d-1}\Gamma(d/2)\pi^{d/2}\int_{0}^{\infty}
\sum_{\ell=0}^{\infty}\frac{J_{\ell+(d-2)/2}(\lambda r)}{(\lambda r)^{(d-2)/2}}\frac{J_{\ell+(d-2)/2}(\lambda r')}{(\lambda r')^{(d-2)/2}}\\
&\quad\times\sum_{m=1}^{h(\ell,d)}Y^m_{\ell}(\theta_1,\dots,\theta_{d-2},\varphi)
Y^{m}_{\ell}(\theta'_1,\dots,\theta'_{d-2},\varphi')\,\mathrm{d}\Phi(\lambda),
\end{aligned}
\]
where $(r,\theta_1,\dots,\theta_{d-2},\varphi)$ (resp. $(r',\theta'_1,\dots,\theta'_{d-2},\varphi')$) are the spherical coordinates of the point $\mathbf{x}$ (resp. $\mathbf{y}$). The above equation has the form \eqref{eq:21}, where $\Lambda$ is the disjoint union of countably many copies of the interval $[0,\infty)$ enumerated by the indices $\ell$ and $m$, the restriction of the function $a(\mathbf{x},\lambda)$ to the Cartesian product of $\mathbb{R}^d$ and the $(\ell,m)$th copy of $[0,\infty)$ has the form
\[
a(\mathbf{x},\lambda)=\left(2^{d-1}\Gamma(d/2)\pi^{d/2}\right)^{1/2}
\frac{J_{\ell+(d-2)/2}(\lambda r)}{(\lambda r)^{(d-2)/2}}Y^m_{\ell}(\theta_1,\dots,\theta_{d-2},\varphi),
\]
the restriction of the measure $\mathrm{d}F$ to the above copy is equal to $\Phi$, and $V=c\mathbb{R}^1$. Apply Theorem~\ref{th:Karhunen} and obtain Equation~\eqref{eq:8}. In this equation, the random field $T(\mathbf{x})$ is real-valued if and only if all measures $Z^m_{\ell}$ are real-valued.
\end{example}

Before proceeding to the next example, we describe some elements of the equivalence classes of irreducible orthogonal representations of the group $\mathrm{O}(3)$. Let $H_{\ell}(\mathbb{R}^3)$ and $H_{\ell*}(\mathbb{R}^3)$ be two copies of the real linear space of homogeneous polynomials of degree $\ell$ in three real variables that are harmonic (with null Laplacian). Assume that the group $\mathrm{O}(3)$ acts in the space $H_{\ell}(\mathbb{R}^3)$ by
\[
(g\cdot p)(\mathbf{x})=p(g^{-1}\mathbf{x}),\qquad g\in\mathrm{O}(3),\quad\mathbf{x}\in\mathbb{R}^3,
\]
and in the space $H_{\ell*}(\mathbb{R}^3)$ by
\[
(g\cdot p)(\mathbf{x})=\det(g)p(g^{-1}\mathbf{x}),\qquad g\in\mathrm{O}(3),\quad\mathbf{x}\in\mathbb{R}^3.
\]

First, any irreducible orthogonal representation of the group $\mathrm{O}(3)$ is equivalent to a unique representation in the above described class. Second, it is easy to see that the element $-I\in\mathrm{O}(3)$ acts by multiplication by $1$ in the spaces $H_{0}(\mathbb{R}^3)$, $H_{2}(\mathbb{R}^3)$, \dots, $H_{1*}(\mathbb{R}^3)$, $H_{3*}(\mathbb{R}^3)$, \dots, and acts by multiplication by $-1$ in the remaining spaces.

We denote by $U_{\ell+}$\index{representation!irreducible real of $\mathrm{O}(3)$!$U_{\ell+}$} (resp. $U_{\ell-}$)\index{representation!irreducible real of $\mathrm{O}(3)$!$U_{\ell-}$} a real linear space of dimension $2\ell+1$ where an irreducible representation $\theta^{U_{\ell+}}$ (resp. $\theta^{U_{\ell-}}$) with $\theta^{U_{\ell+}}(-I)x=x$ (resp. $\theta^{U_{\ell-}}(-I)x=-x$) acts. The \emph{Clebsch--Gordan rule}\index{Clebsch--Gordan rule} states that the representations $U_{\ell_1+}\otimes U_{\ell_2+}$ and $U_{\ell_1-}\otimes U_{\ell_2-}$ are equivalent to the direct sum of irreducible components $U_{|\ell_1-\ell_2|+}$, $U_{(|\ell_1-\ell_2|+1)+}$, \dots, $U_{(\ell_1+\ell_2)+}$. Similarly, the representations $U_{\ell_+}\otimes U_{\ell_2-}$ and $U_{\ell_1-}\otimes U_{\ell_2+}$ are equivalent to the direct sum of irreducible components $U_{|\ell_1-\ell_2|-}$, $U_{(|\ell_1-\ell_2|+1)-}$, \dots, $U_{(\ell_1+\ell_2)-}$.

The result \eqref{eq:6} is easy to prove. Indeed, Equation~\eqref{eq:3} states that the two-point correlation tensor of the random field that describes a turbulent fluid, is a form-invariant map for the pair $(U_{1-}\otimes U_{1-},U_{1-})$. To find the general form of the above tensor, we use the Wineman--Pipkin Theorem~\ref{th:WP}. It is well-known that the polynomial $I_1(\mathbf{x})=\|\mathbf{x}\|^2$ constitutes an integrity basis for polynomial invariants of the representation $U_{1-}$. To find an integrity basis for form-invariant polynomials of the pair $(U_{1-}\otimes U_{1-},U_{1-})$, we use the following result by Hermann Weyl\index{Weyl, Hermann} proved by him in 1939 in the first edition of \cite{MR1488158}.

\begin{theorem}[Hermann Weyl]
Any form-invariant polynomial of the group $\mathrm{O}(d)$ is a linear combination of products of Kronecker's deltas $\delta_{ij}$ and second degree homogeneous polynomials $r_ir_j$.
\end{theorem}

In particular, the matrix-valued polynomials \eqref{eq:27} form an integrity basis for form-invariant polynomials of degree not more than $2$. Equation~\eqref{eq:6} immediately follows from Theorem~\ref{th:WP}.

\begin{example}\label{ex:2}
Let $U=U_{1-}$, that is, $U(g)=g$, $g\in\mathrm{O}(3)$. The Clebsch--Gordan rule gives $U_{1-}\otimes U_{1-}\sim U_{0+}\oplus U_{1+}\oplus U_{2+}$, where the symbol $\sim$ means the equivalence of representations. The representation $\tilde{U}=\mathsf{S}^2(U_{1-})$ is
\begin{equation}\label{eq:23}
\mathsf{S}^2(U_{1-})\sim U_{0+}\oplus U_{2+}.
\end{equation}

The second step becomes nontrivial and requires a new concept. There is a basis in the spaces $U_{\ell\pm}$ described in \cite{MR1888117}, denote its elements by $\{\,\mathbf{e}^{\ell}_m\colon -\ell\leq m\leq\ell\,\}$. The \emph{Godunov--Gordienko coefficients}\index{Godunov--Gordienko coefficients} $g^{m[m_1,m_2]}_{\ell[\ell_1,\ell_2]}$ are defined by
\begin{equation}\label{eq:30}
\mathbf{e}^{\ell}_m\sim\sum_{m_1=-\ell_1}^{\ell_1}\sum_{m_2=-\ell_2}^{\ell_2}
g^{m[m_1,m_2]}_{\ell[\ell_1,\ell_2]}\mathbf{e}^{\ell_1}_{m_1}\otimes\mathbf{e}^{\ell_2}_{m_2},
\end{equation}
that is, under the equivalence of representations stated in the Clebsch--Gordan rule, the basis vector $\mathbf{e}^{\ell}_m$ is mapping to the matrix in the right hand side. These coefficients were calculated in \cite{MR2078714}. They are \emph{not} the same as classical Clebsch--Gordan coefficients for unitary representations of the groups $\mathrm{SU}(2)$ and $\mathrm{SO}(3)$ known from quantum mechanics.

In particular, the basis vector $\mathbf{e}^0_0$ in the space $U_{0+}$ in the right hand side of Equation~\eqref{eq:23} is mapping to the matrix $g^0_{0[1,1]}$ with matrix entries $g^{0[m_1,m_2]}_{0[1,1]}$. The algorithm for calculating the Godunov--Gordienko coefficients is described in \cite{MR3308053}, see also the properties of matrices $g^m_{\ell[\ell_1,\ell_2]}$ in \cite{MR3783870}. The algorithm gives $g^0_{0[1,1]}=\frac{1}{\sqrt{3}}\left(
\begin{smallmatrix}
1 & 0 & 0 \\
0 & 1 & 0 \\
0 & 0 & 1
\end{smallmatrix}
\right)$.

Similarly, the basis vector $\mathbf{e}^2_0$ in the space $U_{2+}$ is mapping to the matrix $g^0_{2[1,1]}=\frac{1}{\sqrt{6}}\left(
\begin{smallmatrix}
-1 & 0 & 0 \\
0 & 2 & 0 \\
0 & 0 & -1
\end{smallmatrix}
\right)$.

The linear slice $\tilde{U}_0$ is the linear space generated by the matrices $g^0_{0[1,1]}$ and $g^0_{2[1,1]}$.

Similarly, the linear slice $\tilde{U}_1$ is the linear space generated by the matrix $g^0_{0[1,1]}$.

At the third step, an easy application of the Sylvester Theorem shows that a matrix in $\tilde{U}_0$ is nonnegative-definite and has unit trace if and only if it has the form $\frac{1}{\sqrt{3}}g^0_{0[1,1]}+ug^0_{2[1,1]}$ with $-\frac{1}{\sqrt{6}}\leq u\leq\frac{\sqrt{2}}{\sqrt{3}}$. The convex compact set $\mathcal{C}_0\subset\mathsf{S}^2(\mathbb{R}^3)$ is the interval with extreme points $A_1=\frac{1}{2}\left(
\begin{smallmatrix}
  1 & 0 & 0 \\
  0 & 0 & 0 \\
  0 & 0 & 1
\end{smallmatrix}
\right)$ and $A_2=\left(\begin{smallmatrix}
  0 & 0 & 0 \\
  0 & 1 & 0 \\
  0 & 0 & 0
\end{smallmatrix}
\right)$.

A matrix in $\tilde{U}_1$ is nonnegative-definite and has unit trace if and only if it is equal to $\frac{1}{\sqrt{3}}g^0_{0[1,1]}$. The convex compact set $\mathcal{C}_1$ is a singleton $\{\frac{1}{\sqrt{3}}g^0_{0[1,1]}\}$.

At the fourth step, we introduce the \emph{barycentric coordinates}\index{barycentric coordinates} $u_1(\lambda)$ and $u_2(\lambda)$ of the matrix $f(\lambda,0,0)$ with respect to the simplex $\mathcal{C}_0$, that is,
\[
f(\lambda,0,0)=u_1(\lambda)A_1+u_2(\lambda)A_2,\qquad\lambda>0.
\]
In terms of the basis matrices $g^0_{0[1,1]}$ and $g^0_{2[1,1]}$ this equation becomes
\[
f(\lambda,0,0)=\frac{1}{\sqrt{3}}(u_1(\lambda)+u_2(\lambda))g^0_{0[1,1]}
+\frac{1}{\sqrt{6}}(-u_1(\lambda)+2u_2(\lambda))g^0_{2[1,1]}.
\]
Observe that $u_1(\lambda)+u_2(\lambda)=1$. When $\lambda=0$, we have $f(0,0,0)\in\mathcal{C}_1$, that is, $f(0,0,0)=\frac{1}{3}\delta_{ij}$. The above equation holds true for $\lambda=0$ if and only if
\begin{equation}\label{eq:24}
u_1(0)=\frac{2}{3},\qquad u_2(0)=\frac{1}{3}.
\end{equation}

Let $\theta_{ij}^{\tilde{U}}(g)$ be the matrix entries of the representation $\theta^{\tilde{U}}$, $g\in\mathrm{O}(3)$. Equation~\eqref{eq:16} gives
\[
f_{ij}(\lambda,\hat{\theta},\hat{\varphi})=\frac{1}{3}\delta_{ij}
\theta_{00}^{U_{0+}}(\hat{\theta},\hat{\varphi})
+\frac{1}{\sqrt{6}}(-u_1(\lambda)+2u_2(\lambda))\sum_{m=-2}^{2}g^{m[i,j]}_{2[1,1]}
\theta_{m0}^{U_{2+}}(\hat{\theta},\hat{\varphi}).
\]
Equation~\eqref{eq:11} takes the form
\begin{equation}\label{eq:29}
\begin{aligned}
\langle T_i(\mathbf{x}),T_j(\mathbf{y})\rangle&=\frac{1}{4\pi}\int_{0}^{\infty}\int_{S^2}
\mathrm{e}^{\mathrm{i}(\mathbf{k},\mathbf{r})}\left[\frac{1}{3}\delta_{ij}
\theta_{00}^{U_{0+}}(\hat{\theta},\hat{\varphi})\right.\\
&\quad+\left.\frac{1}{\sqrt{6}}(-u_1(\lambda)+2u_2(\lambda))\sum_{m=-2}^{2}g^{m[i,j]}_{2[1,1]}
\theta_{m0}^{U_{2+}}(\hat{\theta},\hat{\varphi})\right]\,\mathrm{d}S\,\mathrm{d}\Phi(\lambda).
\end{aligned}
\end{equation}

At the fifth step, we use the following. It is well-known that $\theta_{m0}^{U_{\ell+}}(\hat{\theta},\hat{\varphi})=\frac{2\sqrt{\pi}}{\sqrt{2\ell+1}}S^m_{\ell}(\hat{\theta},\hat{\varphi})$. Using the Rayleigh expansion, we calculate the inner integral and obtain
\[
\begin{aligned}
T_i(\mathbf{x}),T_j(\mathbf{y})\rangle&=\int_{0}^{\infty}
\left[\frac{1}{3}(u_1(\lambda)+u_2(\lambda))j_0(\lambda r)\delta_{ij}\right.\\
&\quad+\left.\frac{1}{\sqrt{6}}(u_1(\lambda)-2u_2(\lambda))j_2(\lambda r)\sum_{m=-2}^{2}g^{m[i,j]}_{2[1,1]}
\theta_{m0}^{U_{2+}}(\theta,\varphi)\right]\,\mathrm{d}\Phi(\lambda).
\end{aligned}
\]
Define the measures $\Phi_1$ and $\Phi_2$ by $\mathrm{d}\Phi_k(\lambda)=u_k(\lambda)\,\mathrm{d}\Phi(\lambda)$ and group the terms with the same value of $k$ together. The two-point correlation tensor takes the form
\begin{equation}\label{eq:26}
\begin{aligned}
T_i(\mathbf{x}),T_j(\mathbf{y})\rangle&=\int_{0}^{\infty}
\left[\frac{1}{3}j_0(\lambda r)\delta_{ij}+\frac{1}{\sqrt{6}}j_2(\lambda r)\sum_{m=-2}^{2}g^m_{2[1,1]}
\theta_{m0}^{U_{2+}}(\theta,\varphi)\right]\,\mathrm{d}\Phi_1(\lambda)\\
&\quad+\int_{0}^{\infty}\left[\frac{1}{3}j_0(\lambda r)\delta_{ij}-\frac{\sqrt{2}}{\sqrt{3}}j_2(\lambda r)\sum_{m=-2}^{2}g^m_{2[1,1]}
\theta_{m0}^{U_{2+}}(\theta,\varphi)\right]\,\mathrm{d}\Phi_2(\lambda).
\end{aligned}
\end{equation}

It follows from Equation~\eqref{eq:24} that $\Phi_1(\{0\})=2\Phi_2(\{0\})$. At a first glance, this contradicts condition~\eqref{eq:25}. However, it is not so, see Remark~\ref{rem:1}. This fact was also mentioned in \cite{MR0094844}:
\begin{quote}
\dots without loss of generality, we can always require that the jumps at zero of the functions $\Phi_1(\lambda)$ and $\Phi_2(\lambda)$ be equal\dots; moreover, we can also require that that the jump at zero of one of the functions $\Phi_1(\lambda)$, $\Phi_2(\lambda)$ be zero.
\end{quote}

A more serious problem is that Equation~\eqref{eq:26} contradicts Equation~\eqref{eq:6}. To solve this issue, denote
\[
M^1_{ij}(\mathbf{r})=g^0_{0[1,1]},\qquad M^2_{ij}(\mathbf{r})=r^2\sum_{m=-2}^{2}g^m_{2[1,1]}
\theta_{m0}^{U_{2+}}(\theta,\varphi).
\]
It is well-known that the matrix entries $\theta_{m0}^{U_{\ell+}}(\theta,\varphi)$ are the restrictions to the sphere $S^2$ of homogeneous harmonic polynomials of degree~$\ell$ in three variables given by $r^{\ell}\theta_{m0}^{U_{\ell+}}(\theta,\varphi)$. In particular, $M^1_{ij}(\mathbf{r})$ and $M^2_{ij}(\mathbf{r})$ are form-invariant polynomials. By definition of the integrity basis, they must express through the functions $L^1_{ij}(\mathbf{r})$ and $L^2_{ij}(\mathbf{r})$ given by Equation~\eqref{eq:27}. Indeed, we have
\[
M^1_{ij}(\mathbf{r})=\frac{1}{\sqrt{3}}L^1_{ij}(\mathbf{r}),\qquad M^2_{ij}(\mathbf{r})
=-\frac{1}{\sqrt{6}}L^1_{ij}(\mathbf{r})+\frac{\sqrt{3}}{\sqrt{2}}L^2_{ij}(\mathbf{r}).
\]
The first equality is obvious. Proof of the second one can be found in \cite[p.~149]{MR3930601}. Insert these values to Equation~\eqref{eq:26}. We obtain
\[
\begin{aligned}
T_i(\mathbf{x}),T_j(\mathbf{y})\rangle&=\int_{0}^{\infty}
\left[\frac{1}{6}(2j_0(\lambda r)-j_2(\lambda r))\delta_{ij}+\frac{1}{2}j_2(\lambda r)\frac{r_ir_j}{r^2}\right]\,\mathrm{d}\Phi_1(\lambda)\\
&\quad+\int_{0}^{\infty}\left[\frac{1}{3}(j_0(\lambda r)+j_2(\lambda r))\delta_{ij}+j_2(\lambda r)\frac{r_ir_j}{r^2}\right]\,\mathrm{d}\Phi_2(\lambda).
\end{aligned}
\]
By comparing these equation with \eqref{eq:6} one can calculate the longitudinal and transversal correlation functions and prove that they coincide with the functions~\eqref{eq:28} up to a constant.

At the last step, to find the spectral expansion of the random field $T(\mathbf{x})$, we apply the idea of Example~\ref{ex:1} to Equation~\eqref{eq:29}. But this time we have to calculate the integral over $S^2$ of the product of \emph{three} spherical harmonics. Observe that Equation~\eqref{eq:30} defines the Godunov--Gordienko coefficients $g^{m[m_1,m_2]}_{\ell[\ell_1,\ell_2]}$ only for nonnegative values of $\ell_1$, $\ell_2$, and $\ell$, satisfying the Clebsch--Gordan rule
\[
|\ell_1-\ell_2|\leq\ell\leq\ell_1+\ell_2.
\]
Put $g^{m[m_1,m_2]}_{\ell[\ell_1,\ell_2]}=0$, if this rule is broken.

\begin{lemma}\label{lem:Gaunt}
We have
\[
\int_{S^2}Y^{m_1}_{\ell_1}(\hat{\theta},\hat{\varphi})Y^{m_2}_{\ell_2}(\hat{\theta},\hat{\varphi})
Y^{m_3}_{\ell_3}(\hat{\theta},\hat{\varphi})\mathrm{d}S
=\frac{\sqrt{(2\ell_1+1)(2\ell_2+1)}}{\sqrt{4\pi(2\ell_3+1)}}
g^{m_3[m_1,m_2]}_{\ell_3[\ell_1,\ell_2]}g^{0[0,0]}_{\ell_3[\ell_1,\ell_2]}.
\]
\end{lemma}

\begin{proof}
The complex counterpart of this Lemma is called the \emph{Gaunt integral}\index{Gaunt integral} after \cite{doi:10.1098/rspa.1929.0037}. The real version is proved in exactly the same way, see the proof of the complex counterpart in \cite{MR2840154}.
\end{proof}

After double insertion of the Rayleigh expansion to Equation~\eqref{eq:29} and using Lemma~\ref{lem:Gaunt}, we obtain
\[
\begin{aligned}
\langle T_i(\mathbf{x}),T_j(\mathbf{y})\rangle&=\sum_{n=1}^{2}\sum_{\ell,\ell'=0}^{\infty}
\sum_{m=-\ell}^{\ell}\sum_{m'=-\ell'}^{\ell'}
C^{mm'}_{n\ell\ell'ij}S^m_{\ell}(\theta,\varphi)S^{m'}_{\ell'}(\theta',\varphi')\\
&\quad\times\int_{0}^{\infty}j_{\ell}(\lambda r)j_{\ell'}(\lambda r')\,\mathrm{d}\Phi_n(\lambda),
\end{aligned}
\]
where
\[
\begin{aligned}
C^{mm'}_{1\ell\ell'ij}&=4\pi\mathrm{i}^{\ell'-\ell}\sqrt{(2\ell+1)(2\ell'+1)}
\left(\frac{1}{3}\delta_{ij}g^{0[m,m']}_{0[\ell,\ell']}g^{0[0,0]}_{0[\ell,\ell']}\right.\\
&\quad\left.-\frac{1}{5\sqrt{6}}g^{0[0,0]}_{2[\ell,\ell']}\sum_{k=-2}^{2}
g^{k[i,j]}_{2[1,1]}g^{k[m,m']}_{2[\ell,\ell']}\right),\\
C^{mm'}_{2\ell\ell'ij}&=4\pi\mathrm{i}^{\ell'-\ell}\sqrt{(2\ell+1)(2\ell'+1)}
\left(\frac{1}{3}\delta_{ij}g^{0[m,m']}_{0[\ell,\ell']}g^{0[0,0]}_{0[\ell,\ell']}\right.\\
&\quad\left.+\frac{\sqrt{2}}{5\sqrt{3}}g^{0[0,0]}_{2[\ell,\ell']}\sum_{k=-2}^{2}
g^{k[i,j]}_{2[1,1]}g^{k[m,m']}_{2[\ell,\ell']}\right).
\end{aligned}
\]

The Karhunen Theorem gives
\[
T_i(r,\theta,\varphi)=\sum_{n=1}^{2}\sum_{\ell=0}^{\infty}
\sum_{m=-\ell}^{\ell}S^m_{\ell}(\theta,\varphi)\int_{0}^{\infty}j_{\ell}(\lambda r)\,\mathrm{d}Z^m_{ni\ell}(\lambda),
\]
where
\[
\mathsf{E}[Z^m_{ni\ell}(A)Z^{m'}_{n'j\ell'}(B)]=\delta_{nn'}C^{mm'}_{n\ell\ell'ij}
\Phi_n(A\cap B).
\]
\end{example}

\begin{example}
Let $U=\mathsf{S}^2(U_{1-})=U_{0+}\oplus U_{2+}$. The Clebsch--Gordan formula gives $\mathsf{S}^2(U)=2U_{0+}\oplus 2U_{2+}\oplus U_{4+}$.

It turns out that the real linear space $\tilde{U_0}$ is $5$-dimensional, while $\tilde{U}_1$ is $2$-dimensional. For details, we refer to \cite{MR3493458} and \cite[Section~3.6]{MR3930601}

A new phenomenon appears at the third step. After a reordering of the basis tensors, the matrix $f(\lambda_0)$ contains two pairs of identical $1\times 1$ diagonal blocks and one $2\times 2$ block. Accordingly, the set of extreme points of the set $\mathcal{C}_0$ consists of three connected components: two singletons $D^1$ and $D^2$ \emph{and an ellipse} $D$. The set $\mathcal{C}_1$ is an interval with two extreme points.

The matrix $f(\lambda_0)$ has the form
\[
f(\lambda_0)=u_1(\lambda_0)D^1+u_2(\lambda_0)D^2+(u_3(\lambda_0)+u_4(\lambda_0))
D(\lambda_0),\qquad\lambda_0>0,
\]
where $D(\lambda_0)$ is an arbitrary measurable function on $(0,\infty)$ that takes values in the closed convex span of the ellipse $D$. This equation remains true for $\lambda=0$ if and only if
\begin{equation}\label{eq:19}
u_2(0)=\frac{3}{2}u_1(0),\qquad\frac{2}{7}\leq u_3(0)+u_4(0)\leq 1.
\end{equation}
This result follows from the analysis of the position of the convex compact $\mathcal{C}_1$ inside $\mathcal{C}_0$, like in Example~\ref{ex:2}.

Calculating the inner integral, we obtain a long formula, see \cite[Equation~(3.80)]{MR3930601}, written in terms of five $M$-functions
\[
\begin{aligned}
M^1_{ijkl}(\mathbf{r})&=g^{0[i,j]}_{0[1,1]}\otimes g^{0[k,l]}_{0[1,1]},\\
M^2_{ijkl}(\mathbf{r})&=\sum_{m_1,m_2=-2}^{2}g^{0[m_1,m_2]}_{0[2,2]}g^{m_1[i,j]}_{2[1,1]}
\otimes g^{m_2[k,l]}_{2[1,1]},\\
M^3_{ijkl}(\mathbf{r})&=\frac{r^2}{\sqrt{2}}\sum_{m=-2}^{2}\left(g^{0[i,j]}_{0[1,1]}\otimes g^{m[k,l]}_{2[1,1]}+g^{m[k,l]}_{2[1,1]}\otimes(g^{0[i,j]}_{0[1,1]}\right)\theta_{m0}^{U_{2+}}(\hat{\theta},\hat{\varphi}),\\
M^4_{ijkl}(\mathbf{r})&=r^2\sum_{m,m_1,m_2=-2}^{2}g^{m[m_1,m_2]}_{2[2,2]}g^{m_1[i,j]}_{2[1,1]}
\otimes g^{m_2[k,l]}_{2[1,1]}\theta_{m0}^{U_{2+}}(\hat{\theta},\hat{\varphi}),\\
M^5_{ijkl}(\mathbf{r})&=r^4\sum_{m=-4}^{4}\sum_{m_1,m_2=-2}^{2}g^{0[m_1,m_2]}_{4[2,2]}g^{m_1[i,j]}_{2[1,1]}
\otimes g^{m_2[k,l]}_{2[1,1]}\theta_{m0}^{U_{4+}}(\hat{\theta},\hat{\varphi}).
\end{aligned}
\]
This result does not contradict Equation~\eqref{eq:31}. Indeed, the functions $M^n_{ijkl}(\mathbf{r})$ are form-invariant polynomials and must express through the basis polynomials $L^n_{ijkl}(\mathbf{r})$. They have the form
\[
\begin{aligned}
M^1_{ijkl}(\mathbf{r})&=\frac{1}{3}L^1_{ijkl}(\mathbf{r}),\\
M^2_{ijkl}(\mathbf{r})&=-\frac{1}{3\sqrt{5}}L^1_{ijkl}(\mathbf{r})
+\frac{1}{2\sqrt{5}}L^2_{ijkl}(\mathbf{r}),\\
M^3_{ijkl}(\mathbf{r})&=-\frac{1}{3}L^1_{ijkl}(\mathbf{r})
+\frac{1}{2}L^4_{ijkl}(\mathbf{r}),\\
M^4_{ijkl}(\mathbf{r})&=\frac{2\sqrt{2}}{3\sqrt{7}}L^1_{ijkl}(\mathbf{r})
-\frac{1}{\sqrt{14}}L^2_{ijkl}(\mathbf{r})+\frac{3}{2\sqrt{14}}L^3_{ijkl}(\mathbf{r})
-\frac{\sqrt{2}}{\sqrt{7}}L^4_{ijkl}(\mathbf{r}),\\
M^5_{ijkl}(\mathbf{r})&=\frac{1}{2\sqrt{70}}L^1_{ijkl}(\mathbf{r})
+\frac{1}{2\sqrt{70}}L^1_{ijkl}(\mathbf{r})-\frac{\sqrt{5}}{2\sqrt{14}}L^3_{ijkl}(\mathbf{r})
-\frac{\sqrt{5}}{2\sqrt{14}}L^4_{ijkl}(\mathbf{r})\\
&\quad+\frac{\sqrt{35}}{2\sqrt{2}}L^5_{ijkl}(\mathbf{r}),
\end{aligned}
\]
see \cite{MR3493458} or \cite[Equation~(3.81)]{MR3930601}. The two-point correlation tensor of the random field $T(\mathbf{x})$ takes the form
\begin{equation}\label{eq:32}
\langle T_{ij}(\mathbf{x}),T_{kl}(\mathbf{y})=\sum_{n=1}^{3}\int_{0}^{\infty}
\sum_{q=1}^{5}N_{nq}(\lambda,r)L^q_{ijkl}(\mathbf{r})\,\mathrm{d}\Phi_n(\lambda),
\end{equation}
where the functions $N_{nq}(\lambda,r)$ with $q\in\{1,2\}$ correspond to the singletons $D^1$ and $D^2$ and are linear combinations of spherical Bessel functions. The functions $N_{nq}(\lambda,r)$ include also the component $D(\lambda)$, because they correspond to the ellipse~$D$, see \cite[Table~3.1]{MR3930601}. It follows from Equation~\eqref{eq:19} that an eventual atom $\Phi_3(\{0\})$ occupies at least $\frac{2}{7}$ of the sum of all three atoms, while the rest is divided between $\Phi_1(\{0\})$ and $\Phi_2(\{0\})$ in the proportion $1\colon\frac{3}{2}$.

The spectral expansion of the random field $T(\mathbf{x})$ can be established in the same way as in Example~\ref{ex:2}, see \cite[Section~3.6, Theorem~26]{MR3930601}.
\end{example}

\begin{example}
Let $U=\mathbb{R}^3\otimes\mathsf{S}^2(\mathbb{R}^3)$. The Clebsch--Gordan rule gives
\[
\begin{aligned}
U&=2U_{1-}\oplus U_{2-}\oplus U_{3-},\\
\mathsf{S}^2(U)&=5U_{1+}\oplus U_{1+}\oplus 10U_{2+}\oplus 5U_{3+}\oplus 5U_{4+}\oplus U_{5+}\oplus U_{6+}.
\end{aligned}
\]

The dimensions of linear slices are $\dim\tilde{U}_0=21$, $\dim\tilde{U}_1=5$. The bases in the above spaces are given in \cite[Table~3.5]{MR3930601}.

The set of extreme points of the convex compact $\mathcal{C}_0$ consists of three connected components. No of them are singletons. The convex compact $\mathcal{C}_1$ is a simplex with $5$~vertices.

The nonzero elements of the $18\times 18$ matrix $f(\lambda)$ in terms of $21$ $M$-functions are given in \cite[Table~3.8]{MR3930601}. We have $5$ $L$-functions of degree~$0$ given by \cite[Equation~(2.40)]{MR3930601}, $11$ $L$-functions of degree~$2$ given by \cite[Table~2.3]{MR3930601}, $5$ $L$-functions of degree~$4$ given by \cite[Equation~(2.44)]{MR3930601}, and $1$ $L$-function of degree~$6$, $22$~$L$-functions altogether, which is more that $21$. This phenomenon is well-known in invariant theory. The elements of the basis of form-invariant polynomials are not necessary independent. There may exist polynomial relations between them called \emph{syzygies}.\index{syzygy} In our case, one of the $L$-functions of degree~$2$ is a linear combinations of the remaining $L$-functions, see \cite[p.~105]{MR3930601}. The $M$-functions are expressed as linear combinations of the $21$ linearly independent $L$-functions according to \cite[Table~3.6]{MR3930601}.

The two-point correlation tensor of the random field $T(\mathbf{x})$ is given by an equation similar to \eqref{eq:32}. This time, however, there are $63$~functions $N_{nq}(\lambda,r)$ given by \cite[Table~3.10]{MR3930601}.

The spectral expansion of the random field $T(\mathbf{x})$ can be established in the same way as in Example~\ref{ex:2}, see \cite[Section~3.7, Theorem~34]{MR3930601}.
\end{example}

\begin{example}
Let $U=\mathsf{S}^2(\mathsf{S}^2(\mathbb{R}^3))$. The Clebsch--Gordan rule gives
\[
\begin{aligned}
U&=2U_{0+}\oplus 2U_{2+}\oplus U_{4+},\\
\mathsf{S}^2(U)&=7U_{0+}\oplus U_{1+}\oplus 10U_{2+}\oplus 3U_{3+}\oplus 8U_{4+}\oplus 2U_{5+}\oplus 3U_{6+}\oplus U_{8+}.
\end{aligned}
\]

The dimensions of linear slices are $\dim\tilde{U}_0=29$, $\dim\tilde{U}_1=7$. The bases in the above spaces are given in \cite[Table~3.11]{MR3930601}.

The set of extreme points of the convex compact $\mathcal{C}_0$ consists of three connected components. The first one is $14$-dimensional, the second is $5$-dimensional, the third is $4$-dimensional. For the convex compact $\mathcal{C}_1$, the set of its extreme points contains $1$ singleton and $2$ ellipses.

The nonzero elements of the $21\times 21$ matrix $f(\lambda)$ in terms of $29$ $M$-functions are given in \cite[Table~3.14]{MR3930601}. We have $8$ $L$-functions of degree~$0$ given by \cite[Table~2.2]{MR3930601}, $13$ $L$-functions of degree~$2$ given by \cite[Table~2.4]{MR3930601}, $10$ $L$-functions of degree~$4$ given by \cite[Table~2.5]{MR3930601}, $3$ $L$-functions of degree~$6$ given by \cite[Table~2.6]{MR3930601}, and $1$ $L$-function of degree~$8$, $35$~$L$-functions altogether. Only $29$ of them are linearly independent, and $14$ independent $L$-functions are missing in \cite{LOMAKIN19651048}. There are $35-29=6$ syzygies. The $M$-functions are expressed as linear combinations of the $29$ linearly independent $L$-functions according to \cite[Table~3.12]{MR3930601}.

The two-point correlation tensor of the random field $T(\mathbf{x})$ is given by an equation similar to \eqref{eq:32}. This time, however, there are $87$~functions $N_{nq}(\lambda,r)$ given by \cite[Table~5]{preprint2016}.

The spectral expansion of the random field $T(\mathbf{x})$ can be established in the same way as in Example~\ref{ex:2}, see \cite[Section~3.8, Theorem~36]{MR3930601}.
\end{example}

\subsection{Concluding remarks}\label{sub:concluding}

Note the following connections between the theory of homogeneous and isotropic random fields and other fields of mathematics.

\begin{itemize}
  \item Special functions appear in the spectral theory of homogeneous and isotropic random fields as orthonormal bases in functional spaces and as the coefficients of Fourier expansions with respect to the above bases.
  \item The two-point correlation tensor of a homogeneous and isotropic random field is form-invariant, hence a connection with invariant theory.
  \item After a rearranging of basis vectors, the matrix-valued function $f(\lambda_m)$ has the block-diagonal structure. The number of distinct blocks is equal to the number of connected components in the set of extreme points of the convex compact $\mathcal{C}_m$.
  \item The number of integrals in the spectral expansion of the two-point correlation tensor is equal to the number of connected component in the set of extreme points of the convex compact $\mathcal{C}_0$.
  \item If $\mathcal{C}_0$ is not a simplex, then the above expansion contains arbitrary measurable functions.
  \item The structure of eventual atoms of the measures $\Phi_n$ is determined by the position of $\mathcal{C}_1$ inside $\mathcal{C}_0$.
\end{itemize}

We formulate one more, hypothetical connection. At least, no counter-exam\-ples are known to the authors.
\begin{itemize}
  \item The set $\mathcal{C}_0$ is a simplex if and only if the representation $U$ is irreducible.
\end{itemize}

Is it possible to have a formula for the two-point correlation tensor of a homogeneous and isotropic random field that is similar to Equations~\eqref{eq:5} and \eqref{eq:28}, that is, without arbitrary measurable functions? The answer is: yes, but you have to pay for that. The idea is as follows: consider a simplex $\mathcal{C}$ satisfying $\mathcal{C}_1\subset\mathcal{C}\subset\mathcal{C}_0$. Force the function $f(\mathbf{k})$ to take values in $\mathcal{C}$. We obtain a description of the two-point correlation tensor of a \emph{subclass} of the class of homogeneous and isotropic random fields. In other words, instead of necessary and sufficient conditions, we obtain only sufficient conditions. The closer is the Lebesgue measure of the simplex $\mathcal{C}$ in comparison with that of $\mathcal{C}_0$, the closer are the obtained sufficient conditions to the necessary ones. See details in \cite{MR3930601}.

\section{Applications to continuum physics}\label{sec:continuum}

\subsection{Motivation from stochastic mechanics}\label{sub:motivation}

Tensor-valued random fields (TRFs) are a natural setting for a stochastic
generalization of \emph{continuum physics}. By this we understand
continuum mechanics of fluids and solids, as well as thermal conductivity
and coupled-field models such as thermoelasticity, thermodiffusion, and
electromagnetic interactions in deformable media (e.g., piezoelectricity).
Our scope is the classical physics.

Now, we focus on two types of TRFs appearing in continuum physics:\ (i)
fields of \emph{dependent quantities} (displacement, velocity,
deformation, rotation, stress,\dots) and (ii) fields of \emph{constitutive responses} (conductivity, stiffness, permeability\dots). Five
such fields were listed in Example~\ref{ex:continuum}. All of these fields are tensors of zeroth, first or higher rank and, generally, of random nature (i.e., displaying spatially inhomogeneous, random character), as opposed to
deterministic continuum physics. In the latter case and as a starting point,
we typically have an equation of the form
\begin{equation}
\mathcal{L}\left( \mathbf{u}\right) =f,  \label{deterministic-eq}
\end{equation}%
defined on a subset $\mathcal{D}$ of the $d$-dimensional affine Euclidean
space $E^{d}$, where $\mathcal{L}$ is a differential operator, $f$ is a
source (or forcing function), and $\mathbf{u}$ is a solution field. This
needs to be accompanied by appropriate boundary and/or initial conditions.

\begin{remark}
We use the symbolic ($\mathbf{u}$) or, equivalently, the subscript ($u_{i\dots}$) notations for tensors, as the need arises; also an overdot will mean the derivative with respect to time, $d/dt$.
\end{remark}

A field theory becomes stochastic in two main situations. First, there
appears an apparent randomness of $\mathbf{u}$ due to an inherent
nonlinearity of $\mathcal{L}$ as exemplified by mathematical models (such as
the Navier--Stokes equations) modeling the turbulent fluid motions. This is
the case of statistical fluid mechanics, with the velocity field being the
random dependent quantity, see \cite{MR2406667,MR2406668,MR0001702,MR0033702}.

Alternatively, the field theory becomes stochastic if the coefficients of $\mathcal{L}$, such as the conductivity tensor, are a TRF, so \eqref{deterministic-eq} becomes
\begin{equation}
\mathcal{L}\left( \omega \right) \mathbf{u}=f.  \label{stoch-eq}
\end{equation}
This stochastic equation governs the response of a \textit{random medium}\index{random medium}
\begin{equation}
\mathcal{B}=\left\{ \,B\left( \omega \right) \colon \omega \in \Omega
\,\right\}  \label{eq:set}
\end{equation}%
on an appropriate spatial domain. In principle, each of the realisations $B\left( \omega \right) $ follows deterministic laws of classical mechanics.
Probability is introduced to deal with the set of all possible realisations.
The governing relation\ is a stochastic partial differential equation (SPDE)
and this ensemble picture is termed \textit{stochastic continuum physics}.
There also is a third possibility of SPDE: random forcing and/or random
boundary/initial conditions; this case will not be pursued here.

In what follows, we discuss the TRFs of dependent fields and of constitutive
responses.

\subsection{TRFs of dependent fields}

\subsubsection{Rank 1 TRFs (vector random field)}\label{subsub:rank1}

\paragraph{Restriction imposed by a divergence-free property} Consider a TRF $\boldsymbol{v}$ over $\mathbb{R}^{d}$ ($d=2$ or $3$) to be
solenoidal, i.e. $\diverge\boldsymbol{v}=0$. Then the correlation function
$R_{ij}:=\left\langle v_{i}\left( \boldsymbol{0}\right), v_{j}\left(
\boldsymbol{r}\right) \right\rangle $ satisfies
\[
0=R_{ij},_{i}\left( \boldsymbol{r}\right) \equiv \frac{R_{ij}\left(
\boldsymbol{r}\right) }{\partial r_{i}},
\]
where the index $,i$ denotes differentiation with respect to $r_i$, and where we use the \emph{Einstein convention}:\index{Einstein convention} if an index variable appears twice in a single term, than that term is summed over all the values of the index. Now, writing the representation~\eqref{eq:6} as $R_{ij}\left(\boldsymbol{r}\right)=A\left(r\right)r_{i}r_{j}+B\left(r\right)\delta_{ij}$, we find
\begin{equation}
\left( n+1\right) A+rA^{\prime }+\frac{1}{r}B^{\prime }=0,
\label{correlation restriction}
\end{equation}
where a prime denotes a derivative with respect to $r$. Next, introduce two
new correlation functions
\begin{equation}
\begin{aligned}
\text{longitudinal: }f\left( r\right) &=\displaystyle\frac{\left\langle
v_{p}\left( \boldsymbol{0}\right), v_{p}\left( \boldsymbol{r}\right)
\right\rangle }{\left\langle v_{p}^{2}\right\rangle }, \\
\text{lateral: }g\left( r\right) &=\displaystyle\frac{\left\langle
v_{n}\left( \boldsymbol{0}\right), v_{n}\left( \boldsymbol{r}\right)
\right\rangle }{\left\langle v_{n}^{2}\right\rangle },
\end{aligned}
\label{f and g}
\end{equation}
whereby $p$ (or $n$) denotes parallel (resp., normal) velocity components,
while the summation convention does not apply to the terms in the
denominator. Such a vector random field is encountered in many physical
settings, e.g. in turbulent incompressible flows \cite{MR2406667,MR2406668,MR0001702}. Another example is the anti-plane elasticity in the absence of body force fields.

By ergodicity in the mean, we have that the square ($v^{2}$) of any velocity
component $v$\ equals $v^{2}=\left\langle v_{p}^{2}\right\rangle
=\left\langle v_{n}^{2}\right\rangle =\frac{1}{n}v_{i}v_{i},$so that%
\[
\begin{aligned}
v^{2}f\left( z\right) &=\left\langle v_{p}\left( \boldsymbol{0}\right),
v_{p}\left( \boldsymbol{r}\right) \right\rangle =A\left( r\right)
r^{2}+B\left( r\right) , \\
v^{2}g\left( r\right) &=\left\langle v_{n}\left( \boldsymbol{0}\right),
v_{n}\left( \boldsymbol{r}\right) \right\rangle =B\left( r\right) .%
\end{aligned}%
\]
It follows from \eqref{correlation restriction} that $f$ and $g$ are
related through
\begin{equation}
g=f+\frac{1}{n-1}rf^{\prime }.  \label{solenoidal-f-g}
\end{equation}
Note that $f$ is typically accessible through experiments or computer
simulations so one can then determine $g$.

The above is the paradigm for treatments of other TRFs in more complex
situations, which can be summarized as:
\begin{itemize}
  \item find the explicit form of the correlation function;
  \item impose a restriction dictated by the relevant physics;
  \item support the results by experiments and/or computational modeling.
\end{itemize}

\paragraph{Restriction imposed by a curl-free property} Other situations in continuum physics require a vector field to be irrotational. Thus, consider a TRF $\boldsymbol{v}$ over $\mathbb{R}^{3}$ to
satisfy $\curl\boldsymbol{v}=\mathbf{0}$. This implies
\[
0=\left\langle \epsilon _{ijk}v_{k},_{j}\left( \boldsymbol{0}\right),
v_{p}\left( \boldsymbol{r}\right) \right\rangle =\epsilon
_{ijk}R_{kp},_{j}\left( \boldsymbol{r}\right) ,
\]
where $\epsilon _{ijk}$\ is the Levi-Civita permutation tensor in three dimensions, see Example~\ref{ex:Levi-Civita}. Given \eqref{eq:6}, we identify two distinct restrictions on
\[
R_{ij}:=\langle v_{i}(\boldsymbol{0}),v_{j}(\boldsymbol{r})\rangle:
\]
\begin{description}
  \item[Case~1]: $\left( i,p\right) =\left( 1,1\right)$:
\[
R_{31},_{2}\left( \boldsymbol{r}\right) =A^{\prime }\left( r\right) \frac{1}{%
r}r_{1}r_{2}r_{3}=R_{21},_{3}\left( \boldsymbol{r}\right).
\]
  \item[Case~2]: $\left( i,p\right) =\left( 1,2\right)$:
\[
\begin{aligned}
R_{32},_{2}\left( \boldsymbol{r}\right) &=\displaystyle A^{\prime }\left(
r\right) \frac{1}{r}r_{2}^{2}r_{3}+B\left( r\right) r_{3}, \\
R_{22},_{3}\left( \boldsymbol{r}\right) &=\displaystyle A^{\prime }\left(
r\right) \frac{1}{z}r_{2}^{2}r_{3}+B^{\prime }\left( r\right) \frac{1}{r}%
r_{3}.
\end{aligned}
\]
\end{description}
Case~1 is satisfied identically, while Case~2 implies the restriction
\[
rA=B^{\prime }\left( r\right) .
\]
In terms of the $f$ and $g$ functions \eqref{f and g}, we find
\begin{equation}
f=g+rg^{\prime }.  \label{irrotational-f-g}
\end{equation}%
Interestingly, this mirrors the relation between $f$ and $g$ in \eqref{solenoidal-f-g} for $n=2$.

Also, \eqref{solenoidal-f-g} and \eqref{irrotational-f-g} provide equations
for a new correlation function from a known one.

\paragraph{Velocity and stress field correlations in fluid mechanics} Returning back to the incompressible velocity field $\boldsymbol{v}$, the
Reynolds stress $R_{kl}:=-\rho \left\langle v_{k},v_{l}\right\rangle $
defines a symmetric rank 2 TRF. Next, consider the spatial average of the
turbulence energy defined from $R_{kl}$ as $\psi =\frac{1}{2}R_{kk}=-\frac{1%
}{2}\rho \left\langle v_{k},v_{l}\right\rangle $. While this defines a scalar
RF, its correlation follows from \eqref{eq:31}
\[
\left\langle \psi (\boldsymbol{0}),\psi (\boldsymbol{r})\right\rangle =\frac{1%
}{4}\left\langle R_{ii}(\boldsymbol{0}),R_{kk}(\boldsymbol{r})\right\rangle =%
\frac{1}{4}\sum_{m=1}^{5}S_{m}(z)J_{iikk}^{(m)}(\boldsymbol{r}),
\]
implying an explicit link between the correlation function of energy and the
five $S_{m}(r)$ functions of the Reynolds stress:
\[
\left\langle \psi \left( \boldsymbol{0}\right), \psi \left( \boldsymbol{r}%
\right) \right\rangle =\frac{9}{4}S_{1}\left( r\right) +\frac{3}{2}%
S_{2}\left( r\right) +\frac{3}{2}S_{3}\left( r\right) +S_{4}\left( r\right) +%
\frac{1}{4}S_{5}\left( r\right) .
\]

\subsubsection{Rank 2 TRFs}

\paragraph{Classical continuum mechanics} Rank 2 TRFs play a very important r\^{o}le in statistical continuum physics. In the case of small deformation gradients, the state of the medium isdescribed by three dependent TRFs of: Cauchy stress $\sigma _{ij}$, displacement $u_{i}$, and strain tensor $\varepsilon_{ij}$. The latter is defined from the former by
\begin{equation}
\varepsilon _{ij}=u_{(i},_{j)},  \label{strain-displacement}
\end{equation}%
where $u_{(i},_{j)}=(u_{i,j}+u_{j,i})$ is the \emph{symmetrised gradient},\index{symmetrised gradient} while the stress field is subject to the balance of linear and angular
momenta

\begin{subequations}\label{Cauchy stress balances}
\begin{align}
\sigma _{ij},_{j}+\rho f_{i}&=\rho \ddot{u}_{i},\\
\sigma_{ij}&=\sigma_{ji}.
\end{align}
\end{subequations}
In \eqref{Cauchy stress balances} $f_{i}$\ is the body force field (resolved
per unit volume) and $\rho$ is the mass density of the body.

Given that any 2nd-rank tensor field $T$ can be decomposed into
potential ($T_{1}$) and birotational ($T_{2}$) parts \cite{MR0095615}
\[
T=T_{1}+T_{2}, \qquad\curl T_{1}=\mathbf{0},\quad\diverge T_{2}=\mathbf{0},
\]
where $T_{1}$\ is described by the vector potential and $T_{2}$ by the tensor potential, one can conclude that $\sigma _{ij}$ is birotational, while $\varepsilon _{ij}$\ is a potential tensor field, with $u_{i}$ being its potential. It follows that \eqref{strain-displacement} and
\eqref{Cauchy stress balances} provide restrictions on the admissible forms
of the correlation structure of these three fields, prior to assuming any
constitutive behavior. Proceeding in the same manner as in Subsubsection~\ref{subsub:rank1}, while assuming absence of any body force fields, such restrictions have been worked out in thermal conductivity, classical elasticity, and micropolar elasticity \cite{MR0163467,MR3327730,SHERMERGOR1971392}.  Extending these results from the quasi-static to dynamic settings is an open research topic.

Interpretations of specific correlations:\ Given that the rank 2 tensor $T$ has diagonal and off-diagonal components, there are five special cases of $B_{ij}^{kl}$ which shed light on the physical meaning of $K_{\alpha }$'s:

\begin{enumerate}
\item $\langle T_{ij}( \boldsymbol{0}), T_{kl}(
\boldsymbol{r}) \rangle |_{i=j=k=l}$; i.e. auto-correlations of
diagonal terms:
\begin{equation*}
\langle T_{11}( \boldsymbol{0}), T_{11}( \boldsymbol{r}%
) \rangle
=K_{0}+2K_{1}+2r_{1}^{2}K_{2}+4r_{1}^{2}K_{3}+r_{1}^{4}K_{4}
\end{equation*}%
and then $\langle T_{22}( \boldsymbol{0}), T_{22}(
\boldsymbol{r}) \rangle $ and $\langle T_{33}(
\boldsymbol{0}), T_{33}( \boldsymbol{r}) \rangle $ by
cyclic permutations $1\rightarrow 2\rightarrow 3$.

\item $\langle T_{ij}( \boldsymbol{0}), T_{kl}(
\boldsymbol{r}) \rangle |_{i=j\neq k=l}$; i.e.
cross-correlations of diagonal terms:
\begin{equation*}
\langle T_{11}( \boldsymbol{0}), T_{22}( \boldsymbol{r}%
) \rangle =K_{0}+( r_{2}^{2}+r_{1}^{2})
K_{2}+r_{2}^{2}r_{1}^{2}K_{4}
\end{equation*}%
and then $\langle T_{22}( \boldsymbol{0}), T_{33}(
\boldsymbol{z}) \rangle $ and $\langle T_{33}(
\boldsymbol{0}), T_{11}( \boldsymbol{r}) \rangle $ by
cyclic permutations $1\rightarrow 2\rightarrow 3$.

\item $\langle T_{ij}( \boldsymbol{0}), T_{kl}(
\boldsymbol{r}) \rangle |_{i=k\neq j=l}$; i.e. auto-correlations
of off-diagonal terms:
\begin{equation*}
\langle T_{12}( \boldsymbol{0}), T_{12}( \boldsymbol{r}%
) \rangle =K_{1}+( r_{1}^{2}+r_{2}^{2})
K_{3}+r_{1}^{2}r_{2}^{2}K_{4}
\end{equation*}%
and then $\langle T_{23}( \boldsymbol{0}), T_{23}(
\boldsymbol{r}) \rangle $ and $\langle T_{31}(
\boldsymbol{0}), T_{31}( \boldsymbol{r}) \rangle $ by
cyclic permutations $1\rightarrow 2\rightarrow 3$.

\item $\langle T_{ij}( \boldsymbol{0}), T_{kl}(
\boldsymbol{r}) \rangle |_{j\neq i=k\neq l\neq j}$; that is
cross-correlations of off-diagonal terms:
\[
\langle T_{12}(\boldsymbol{0}), T_{13}( \boldsymbol{r}) \rangle
=r_{2}r_{3}K_{3}+r_{1}^{2}r_{2}r_{3}K_{4},
\]
then $\langle T_{13}(\boldsymbol{0}), T_{32}( \boldsymbol{r}) \rangle $ and
$\langle T_{32}( \boldsymbol{0}), T_{12}( \boldsymbol{r}) \rangle $ by cyclic permutations $1\rightarrow 2\rightarrow 3$.

\item $\langle T_{ij}( \boldsymbol{0}), T_{kl}(
\boldsymbol{r}) \rangle |_{i=j=k\neq l\neq j}$; i.e.
cross-correlations of diagonal with off-di\-a\-go\-nal terms: such as
\[
\begin{aligned}
\langle T_{11}( \boldsymbol{0}), T_{12}( \boldsymbol{r}) \rangle &=r_{1}r_{2}( K_{2}+2K_{3})+r_{1}r_{2}^{3}K_{4},\\
\langle T_{12}( \boldsymbol{0}),T_{13}( \boldsymbol{r}) \rangle
&=r_{2}r_{3}K_{2}+r_{1}^{2}r_{2}r_{3}K_{4}
\end{aligned}
\]
and the other ones by cyclic permutations $1\rightarrow 2\rightarrow 3$.
\end{enumerate}

In principle, we can determine these five correlations for a specific
physical situation. For example, when $T$ is the anti-plane elasticity tensor for a given resolution (or \emph{mesoscale}, which will be defined in \eqref{mesoscale}), we can use micromechanics or experiments,
and then determine the best fits of $K_{\alpha }$ ($\alpha =1,\dots,5$)
coefficients.

One may determine $C_{ij}^{kl}(\boldsymbol{r})$\ through experimental
measurements or by computational mechanics/physics on diverse material
microstructures, in both cases following a strategy for 4th-rank TRF in 2d
or 3d, see \cite{Sena2013}.

\textit{Special case}: the TRF is locally isotropic $T_{ik}\left(
\boldsymbol{r}\right) =T(r)\delta _{ik}$ with (necessarily) $T\left(
r\right) >0$, so that we simply have a scalar random field. Then, the
auto-correlation $B_{11}^{11}$ of $T$\ is a single scalar function $C(r)$.
With the variance $\Var(T)=\left\langle T(\boldsymbol{0}),T(\boldsymbol{r}%
)\right\rangle $, the correlation coefficient $\rho (\boldsymbol{r}%
):=C(z)/\Var(T)$ is constrained by a standard condition of scalar RFs $%
-1/n\leq \rho (\boldsymbol{r})\leq 1$, if the model is set in $\mathbb{R}%
^{n} $. Basically, this is the correlation function of the conductivity (or
diffusion) in conventional SPDEs of elliptic type, conventionally set up on
scalar random fields.

\paragraph{In-plane case} With the group-theoretical considerations determined in detail in \cite[Section~4.5]{MR3930601}, the interpretations of specific correlations rely on this form of the correlation function for a rank tensor $T$:%
\[
T_{ijkl}(\boldsymbol{r})=H_{1}\delta _{ij}\delta _{kl}+H_{2}\left(
\delta _{ik}\delta _{jl}+\delta _{il}\delta _{jk}\right) +H_{4}\left( \delta
_{ij}r_{k}r_{l}+\delta _{kl}r_{i}r_{j}\right) +H_{5}r_{i}r_{j}r_{k}r_{l}.
\]
Given that $T$ has diagonal and off-diagonal components, there are four special cases of $T_{ijkl}$ which shed light on the physical meaning of $H_{n}$'s:
\begin{enumerate}
\item auto-correlations of diagonal terms: $%
T_{1111}=H_{1}+2H_{2}+2r_{1}^{2}H_{4}+r_{1}^{4}H_{5}$ and $%
T_{2222}=H_{1}+2H_{2}+2r_{2}^{2}H_{4}+r_{2}^{4}H_{5}$.

\item cross-correlation of diagonal terms: $T_{1122}=H_{1}+\left(
r_{1}^{2}+r_{2}^{2}\right) H_{4}+r_{1}^{2}r_{2}^{2}H_{5}$.

\item auto-correlation of an off-diagonal term: $%
T_{1212}=H_{2}+r_{1}^{2}r_{2}^{2}H_{5}$.

\item cross-correlation of a diagonal with an off-diagonal term:
\[
\left\langle T_{11}(\boldsymbol{0}),T_{12}(\boldsymbol{r})\right\rangle
=T_{1112}=r_{1}r_{2}H_{4}+r_{1}^{3}r_{2}H_{5}.
\]
\end{enumerate}

Just as in the case of TRF of rank 1, we can determine these four
correlations for a specific physical situation. Without loss of generality
(due to wide-sense isotropy), when $\boldsymbol{r}\equiv \left(
r_{1},r_{2}\right) $ is chosen equal to $\left( r,0\right) $,
\begin{equation}
\begin{aligned}
H_{1}&=T_{2222}-2T_{1212}, \\
H_{2}&=T_{1212}, \\
H_{4}&=r^{-2}\left(T_{1122}-T_{2222}-2T_{1212}\right), \\
H_{5}&=r^{-4}\left(T_{1111}+T_{2222}-2T_{1122}\right).
\end{aligned}
\label{H-T-connection}
\end{equation}

For example, when $T$ is the anti-plane elasticity tensor for a given resolution (or mesoscale), one can use computational micromechanics or
experiments, and then determine the best fits of $H_{\alpha }$ ($\alpha
=1,2,4,5$) coefficients, providing the positive-definiteness of $T$
is imposed. However, when $T$ represents a dependent field quantity, then a restriction dictated by the field equation needs to be imposed. In the following, by analogy to what was reported in \cite{MR3336288}, we
consider $T$ being either the in-plane stress or the in-plane
strain.

TRF with a local isotropic property:\ Take $T_{ij}=T\delta _{ij}$, where the
axial component $T$ is the scalar random field describing the randomness of
such a medium. Since $T_{11}=T_{22}$ and $T_{12}=0$\ must hold everywhere, $T_{1111}=T_{2222}=T_{1122}$ and $T_{1212}=0$. Hence, $H_{2}=H_{4}=H_{5}=0$ and only $H_{1}\neq 0$ is retained, and that is the function modeling the correlations in $T$. One example is the constitutive response (e.g., conductivity $k_{ij}=k\delta _{ij}$) in conventional models of SPDEs; see \cite[Subsection~4.8.1]{MR3930601} for a model with anisotropy.

When $T_{ij}$\ represents a dependent field quantity, then a restriction
dictated by an appropriate field equation needs to be imposed. One may then
consider $T$ to be either the in-plane stress or the in-plane
strain under the restriction imposed on the correlation tensor when it is
divergence-free or potential tensor field.

\paragraph{Micropolar continuum mechanics} The conservation equations of linear and angular momenta in a micropolar continuum have been given in \cite[Chap\-ter~1]{MR3930601}. Henceforth, focusing on a static problem in the absence of body forces and moments, (\ref{Cauchy stress
balances}a) remains unchanged ($\tau _{ij},_{j}+\rho f_{i}=\rho \ddot{u%
}_{i}$), whereas (\ref{Cauchy stress balances}b) is replaced by
\begin{equation}
\mu _{ki},_{k}+\epsilon _{ijk}\tau _{jk}+\rho g_{i}=\rho J\ddot{\varphi}_{i}.
\label{micropolar balance laws}
\end{equation}%
Here $\mu _{ki}$ is the couple-stress, $\tau _{ji}$\ is the Cauchy stress, $g_{i}$ is the body torque per unit mass, $J$ is the (micro)inertia tensor of
a particle, and $\varphi _{i}$\ is the microrotation. We write $\tau _{jk}$
for the Cauchy stress to point out that this tensor generally lacks symmetry
of $\sigma _{ij}$ in \eqref{Cauchy stress balances}.

For a statistically homogeneous case, from (\ref{micropolar balance laws}) it follows that%
\begin{equation}
\left\langle \epsilon _{ijk}\sigma _{jk}\left( \boldsymbol{r}_{1}\right),
\epsilon _{prs}\sigma _{rs}\left( \boldsymbol{r}_{1}+\boldsymbol{r}\right)
\right\rangle =\left\langle \mu _{ji},_{j}\left( \boldsymbol{r}_{1}\right),
\mu _{rp},_{r}\left( \boldsymbol{r}_{1}+\boldsymbol{r}\right) \right\rangle .
\label{epsilon-mu connection}
\end{equation}%
The left hand side may be written as $LHS=\epsilon _{ijk}\epsilon
_{prs}Q_{jk}^{rs}$, where $Q_{jk}^{rs}\left( \boldsymbol{r}\right)
:=\left\langle \sigma _{jk}\left( \boldsymbol{r}_{1}\right), \sigma
_{rs}\left( \boldsymbol{r}_{1}+\boldsymbol{r}\right) \right\rangle $ is the
correlation function of the\ stress field. Since $\sigma _{jk}$\ is
generally asymmetric but, by assumption, statistically isotropic, we have $%
Q_{jk}^{rs}\left( \boldsymbol{r}\right) =Q_{rs}^{jk}\left( \boldsymbol{r}%
\right) $.

\subsection{TRFs of constitutive responses}

\subsubsection{From a random microstructure to mesoscale response}

Take a sheet of paper and hold it against light. You will see a grayscale,
i.e. a spatially non-uniform opacity of paper. This inhomogeneity is due to
formation of paper which involves a random placement of fibers during
manufacture, a high-speed process involving, for example, $10^{10}$ fibers
per second in a newsprint made continuously at 20 $m/s$ in 10 $m$ width, see \cite{Deng1994}. Due to the van der Waals forces, cellulose fibers cluster in flocs comprising thousands of them. But, even if one were able to remove flocculation, with random placement of fibers, one would not be able to obtain a homogeneous medium. One could achieve a homogeneous medium if the fibers had very well controlled lengths and were placed in a perfectly periodic manner as if it were a textile. Thus, paper is a random
(quasi-)two-dimensional (2d) medium, which one may be tempted to describe by
a random field of mass density%
\begin{equation}
\left\{ \rho \left( \mathbf{x},\omega \right) ;\mathbf{x}\in \mathbb{R}%
^{2},\omega \in \Omega \right\} .  \label{mass density RF}
\end{equation}%
Here $\mathbf{x}$\ is the location in the plane of paper and $\omega$
indicates one realization of the mass density field.

By analogy, assuming a linear elastic response, the mechanical properties
are described by a RF of the in-plane stiffness tensor
\begin{equation}
\left\{C\left( \mathbf{x},\omega \right) ;\mathbf{x}\in \mathbb{R}%
^{2},\omega \in \Omega \right\} .  \label{2d stiffness RF}
\end{equation}%
The next tempting assumption is to introduce a local isotropy, i.e. assume
paper to be described not by the 4th-rank tensor random field (TRF)\ but by
a \textquotedblleft vector\textquotedblright\ RF of the Young modulus and
Poisson ratio $\left\{ E\left( \mathbf{x},\omega \right) ,\nu \left( \mathbf{%
x},\omega \right) ;\mathbf{x}\in \mathbb{R}^{2},\omega \in \Omega \right\} $
or, even more simply, by a single scalar RF $E\left( \mathbf{x},\omega
\right) $. Such models have been used widely. However, as discussed in \cite{MR2341287,OSTOJASTARZEWSKI2016111}, anisotropy cannot be disregarded.

Now, as we move the paper sheet away from our eyes, we note that it is
becoming more and more uniform, i.e. its spatial randomness diminishes and
tends to zero as the length scale of observation $L$ becomes very large
(\dots and then infinite) relative to the floc size $d$. Clearly,
introducing the concept of a \textit{mesoscale}\index{mesoscale}
\begin{equation}
\delta :=\frac{L}{d},  \label{mesoscale}
\end{equation}%
we recognize that the RFs above must be mesoscale dependent. Hence, in place
of (\ref{mass density RF})--(\ref{2d stiffness RF}) we can write more
compactly%
\begin{equation}
\left\{ \rho \left( \mathbf{x},\omega ,\delta \right) ,C\left(
\mathbf{x},\omega ,\delta \right) ;\mathbf{x}\in \mathbb{R}^{2},\omega \in
\Omega \right\} .  \label{mass+stiffness-RF}
\end{equation}%
This set of all the deterministic realizations $\rho \left( \mathbf{x}%
,\omega ,\delta \right) $ and $\mathsf{C}\left( \mathbf{x},\omega ,\delta
\right) $ defines a random material $\mathcal{B}_{\delta }$ parameterised by
a mesoscale $\delta $, and occupying a domain $B_{\delta }\subset\mathbb{R}^{2}$:
\[
\mathcal{B}_{\delta }=\left\{ B_{\delta }(\omega );\omega \in \Omega
\right\}.
\]

As we pull the sheet of paper away from our eyes, the mesoscale increases
and we observe that\ the randomness vanishes. Hence, we have the limits:%
\begin{equation*}
\lim_{\delta \rightarrow \infty }\rho \left( \mathbf{x},\omega ,\delta
\right) =\rho ^{\mathrm{eff}},\text{ \ \ \ }\lim_{\delta \rightarrow \infty }C\left( \mathbf{x},\omega ,\delta \right) =C^{\mathrm{eff}},
\end{equation*}
corresponding to a conventional, deterministic continuum. In the above, we
have introduced the effective material properties (such as stiffness,
modulus, and Poisson ratio) which are typically employed in deterministic
models of continuum mechanics.

\begin{itemize}
\item How can one proceed when an RVE cannot be identified?

\item How can one determine the SVE (mesoscale)\ properties?

\item How can one solve a macroscopic boundary value problem?
\end{itemize}

How can one proceed in the case of media with fractal structures and (also)
long range effects?

\begin{figure}
  \centering
  \includegraphics[width=\columnwidth]{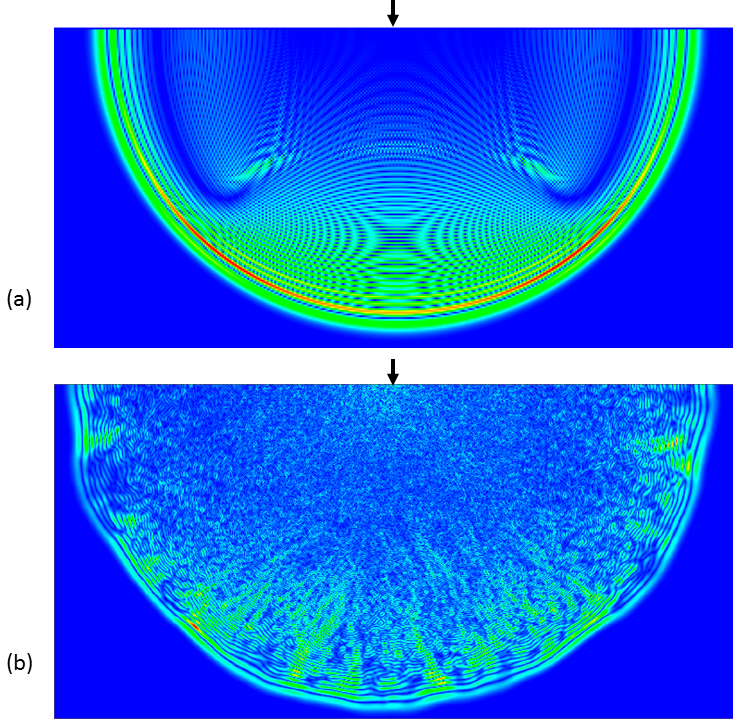}
  \caption{(a) Lamb's problem for a homogeneous half-plane with a concentrated tangential load $P\left( t\right) $ having a
triangular-in-time history. (b) The same problem on a random field of mass
density, with: the coefficient of variation $=0.124$, Cauchy correlation
function \cite{RUIZMEDINA2011259}, fractal dimension $D=2.1$, and Hurst coefficient $H=0.5$. Solutions in both cases have been obtained via a cellular automata method which (i) allows assignment of heterogeneous material properties from cell to cell and (ii) in the limit of infinitesimal cells, is analytically equivalent to the continuum elastodynamics equations \cite{MR3597929}. Reprinted from \cite{MR4031379} by permission of Elsevier.}\label{fig:1}
\end{figure}

As an example, Fig.~\ref{fig:1}(a) shows the response of such an idealized linear elastic half-plane to an impact-type normal force (i.e., a half-plane
subjected to a vertical point force). On the other hand, as Fig.~\ref{fig:1}(b) shows, in terms of one realization of the random field of mass density \eqref{mass density RF}, this Lamb's problem responds differently when the half-plane has a random mass density: the wavefront gets progressively distorted and attenuated through scattering on all the heterogeneities. Depending on the kind of material under consideration, various types of random fields can be considered and, as discussed in \cite[Section~4]{MR4031379}, in order to model a geophysical-type randomness found in nature, the particular random field (RF) taken here has a fractal-and-Hurst character. Clearly, Fig.~\ref{fig:1}(b) represents
one realization from a space of solutions of SBVP on a random field; here
SVBP stands for a \textit{stochastic} BVP. Some basic questions arising here
are:

What random field models of material properties should drive such SBVPs?

What methods exist for solving these SBVPs?

What constraints on RFs of field quantities are dictated by the governing
equations?

The scale-dependent homogenization~--- discussed in detail in \cite{MR2341287,OSTOJASTARZEWSKI2016111}, ~--- shows that the anisotropy accompanies randomness: anisotropy goes to zero as the mesoscale increases. In other words, assuming isotropy of a smooth continuum on any mesoscale forces one to assume homogeneity, i.e. lack of spatial randomness. [This, of course, does not contradict anisotropic piece-wise homogeneous medium models, such as polycrystals.] Thus, to truly allow for anisotropy of the stiffness tensor field, the field equation written for the anti-plane displacement $u\equiv u_{3}$\ is
\begin{equation}
(C\delta _{ij}u,_{j}),_{i}+\rho f=\rho \ddot{u},
\label{isotropic-anti-plane}
\end{equation}%
should be replaced by
\begin{equation}
(C_{ij}u,_{j}),_{i}+\rho f=\rho \ddot{u}.  \label{anisotropic-anti-plane}
\end{equation}%
Here $\mathbf{C}$ is a $2$nd-rank TRF set up on a mesoscale $\delta $
corresponding to the resolution desired in a particular physical problem,
Fig.~\ref{fig:2}(c); both equations are defined for $\mathbf{x}\in\mathbb{R}^{2}$ and $\omega \in \Omega $. This type of a model, instead of $C_{ij}(\boldsymbol{\cdot },\omega )=C(\boldsymbol{\cdot },\omega )\delta _{ij}$, is sorely needed in SPDE and stochastic finite element (SFE) methods.

Moving to the 2d (in-plane) or 3d elasticity, if one assumes local material
isotropy, a simple way to introduce material spatial randomness is to take
the pair of Lam\'{e} constants $(\lambda ,\mu )$ as a \textquotedblleft
vector\textquotedblright\ random field, resulting in a generalization of the
classical Navier equation for the displacement field $\mathbf{u}$:%
\begin{equation}
(\lambda u_{kk}\delta _{ij}+2\mu u_{(i},_{j)}),_{j}+\rho f_{i}=\rho \ddot{u}%
_{i}.  \label{pre-Navier}
\end{equation}%
However, just like in (\ref{isotropic-anti-plane}), the local isotropy is a
crude approximation in view of the micromechanics upscaling arguments. Note
that, due to spatial gradients, any realization of a RF in Fig.~5.2(c)
involves some degree of anisotropy. Thus, \eqref{pre-Navier} should be replaced by
\begin{equation}
(C_{ijkl}u_{(k},_{l)}),_{j}+\rho f_{i}=\rho \ddot{u}_{i},
\label{anisotropic-elasticity}
\end{equation}%
where the stiffness $C$\ ($=C_{ijkl}\mathbf{e}_{i}\otimes \mathbf{e}%
_{j}\otimes \mathbf{e}_{k}\otimes \mathbf{e}_{l}$) is a 4th rank TRF. At any scale finitely larger than the microstructural scale, it is almost surely
anisotropic, typically triclinic.

The same arguments as above apply in case one prefers to work with field
equations directly in the language of stresses. Then the \textit{Ignaczak
equation of elastodynamics} \cite{MR4000178} is applicable, see \cite{MR170515}:
\begin{equation}
\left( \rho ^{-1}\sigma _{(ik},_{k}\right) ,_{j)}-S_{ijkl}\ddot{\sigma}%
_{kl}=0,  \label{Ignaczak-eqn}
\end{equation}%
where $S_{ijkl}$\ are the components of the local compliance tensor $S=C^{-1}$. Note that this formulation avoids the gradients of
compliance but introduces gradients of mass density, which is the reverse of
what the displacement language formulation (\ref{anisotropic-elasticity})
requires.

Summarizing, with the advent of \textquotedblleft multiscale
methods\textquotedblright , the contemporary solid mechanics recognizes the
hierarchical structure of materials, but hardly accounts for their
statistical nature.\ In effect, the multiscale methods remain deterministic,
while the SPDE\ and SFE are rather oblivious to micromechanics,
homogenization, and TRFs of properties.

With reference to \cite[Example~2.1]{MR3930601}, three cases are relevant
here: rank 2 TRF with\ $U=\mathsf{S}^{2}\left( \mathbb{R}^{3}\right) $, rank
3 TRF with $U=\mathsf{S}^{2}\left( \mathbb{R}^{3}\right) \otimes \mathbb{R}%
^{3}$, and rank 4 TRF with $U=\mathsf{S}^{2}\left( \mathsf{S}^{2}\left(
\mathbb{R}^{3}\right) \right) $. The correlation function in the first case
(such as the 3d conductivity) has already been given above in \eqref{eq:31}.

The correlation function in the case of rank 3 TRF(such as piezoelectricity)
has the general form \cite{MR3930601}
\[
R_{ijkprs}(\mathbf{r})=\sum_{n=1}^{21}L_{ijkprs}^{n}(\mathbf{r}%
)K_{n}\left( r\right),
\]
The correlation function in the case of rank 4 TRF (such as elasticity) has
the general form \cite{MR3930601}:
\[
R_{ijk\ell prst}(\mathbf{r})=\sum_{n=1}^{29}L_{ijk\ell prst}^{n}(\mathbf{r}%
)K_{n}\left( r\right).
\]

See \cite{preprint2016} and \cite{MR3621303} for explicit forms of all $L_{ijk\ell prst}^{n}$.

\begin{remark}
The correlation function of a rank 1 (resp., 2, 3, 4) TRF is
a sum of 2 (resp., 5, 21, 29) addends, each being a product of a tensor
function of twice higher rank with a scalar function of the norm $\left\Vert
\mathbf{r}\right\Vert $.
\end{remark}

\subsubsection{TRFs in damage phenomena}

\begin{figure}
  \centering
  \includegraphics[width=\columnwidth]{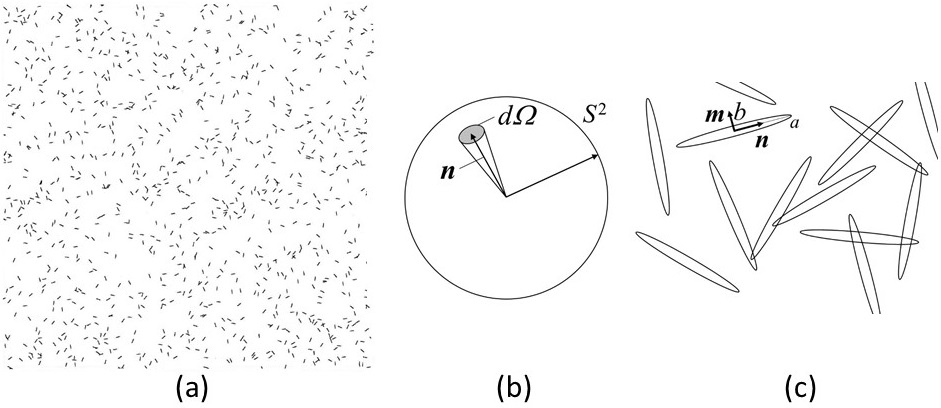}
  \caption{(a) Random crack field; (b) distribution of $\mathbf{n}$ on the unit sphere; (c) random ellipse field with a single ellipse (with semi-axes $b$, $a$) showing the $\left( \mathbf{m},\mathbf{n}\right) $ pair for determination of the fabric tensor. Reprinted from \cite{MR3930601} by permission of Cambridge University Press.}\label{fig:2}
\end{figure}

Many natural and man-made
materials work in presence of fields of many little cracks and defects.
Progressive failure occurs through evolution and catastrophic coalescence of
these microscale events. Descriptions of damage phenomena fields which enter
predictive models of solid mechanics typically involve continuum models.
Consider a disordered field of cracks (Fig.~\ref{fig:2}(a)), each specified by an
orientation vector $\mathbf{n}$ (i.e., a unit vector normal to the crack).
With reference to Fig.~\ref{fig:2}(b), for a system in 3d (respectively, 2d), one
introduces a probability density function of $\mathbf{n}$ on a unit sphere $%
S^{2}$ (respectively, on a circle)%
\begin{equation*}
p\left( \mathbf{n}\right) =D_{0}+D_{ij}f_{ij}\left( \mathbf{n}\right)
+D_{ijkl}f_{ijkl}\left( \mathbf{n}\right) +\dots
\end{equation*}%
This is a generalised Fourier series with respect to the irreducible tensor
bases
\[
\begin{aligned}
f_{ij}\left( \mathbf{n}\right) &=n_{i}n_{j}-\frac{1}{3}\delta _{ij}, \\
f_{ijkl}\left( \mathbf{n}\right) &=n_{i}n_{j}n_{k}n_{l}-\frac{1}{7}(\delta
_{ij}n_{k}n_{l}+\delta _{ik}n_{j}n_{l}+\delta _{il}n_{j}n_{k}+\delta
_{jk}n_{i}n_{l} \\
&\quad+\delta _{jl}n_{i}n_{k}+\delta _{kl}n_{i}n_{j})+\frac{1}{5\times 7}\left(
\delta _{ij}n_{k}n_{l}+\delta _{ik}n_{j}n_{l}+\delta _{il}n_{j}n_{k}\right) ,%
\end{aligned}%
\]
while
\begin{subequations}\label{fabric-tensors}
\begin{align}
D_{0}&=\displaystyle\frac{1}{4\pi }\int_{S^{2}}p\left( \mathbf{n}\right)
d\Omega , \\
D_{ij}&=\displaystyle\frac{1}{4\pi }\frac{3\times 5}{2}\int_{S^{2}}p\left(
\mathbf{n}\right) f_{ij}\left( \mathbf{n}\right) d\Omega , \\
D_{ijkl}&=\displaystyle\frac{1}{4\pi }\frac{3\times 5\times 7\times 9}{%
2\times 3\times 4}\int_{S^{2}}p\left( \mathbf{n}\right) f_{ijkl}\left(
\mathbf{n}\right) d\Omega ,\text{ }\dots
\end{align}
\end{subequations}%
are the scalar, second-order, fourth-order, damage tensors (respectively,
fabric tensors), describing the directional distribution of damage state
(grain-grain contacts). Thus, the overall directional distribution of damage
state of a material is described through the tensors $D_{0}$, $D_{ij}$, $%
D_{ijkl}$, \dots. Depending on the specific choice of the constitutive model,
either just one ($D_{0}$) or two, or more damage tensors are included in the
analysis.

Classical (deterministic) continuum damage mechanics (CDM) makes use of
fabric tensors in a manner borrowed from the mechanics of granular media, see \cite{Ganczarski2015,LUBARDA19932859,Murakamy2012}. Even if such models are introduced on scales much larger than a
single crack or defect, there is a a need to grasp spatially random
fluctuations and anisotropy. In the following, we outline a TRF model of
the rank 2 (symmetric) tensor $D_{ij}$ in (\ref{fabric-tensors}b). Its
correlation function%
\[
\mathcal{D}_{ij}^{kl}=\mathsf{E}\left[ D_{ij}(\boldsymbol{r+r}_{1})D_{kl}(%
\boldsymbol{r}_{1})\right],
\]
in the case of statistically isotropic damage, has the representation \eqref{eq:31}
To determine all the $A_{m}$'s, without loss of generality, we may take the
unit vector $\boldsymbol{n}=\left( n_{1}=1,n_{2}=0,n_{3}=0\right) $
co-aligned with $\boldsymbol{r}$, so that the following auto- and
cross-correlations (consecutively named $M_{i}$, $i=1,\dots,7$) result%
\[
\begin{aligned}
M_{1}&=\left\langle T_{11}(\boldsymbol{0}),T_{11}(\boldsymbol{r})\right\rangle
=S_{1}(r)+2S_{2}(r)+2S_{3}(r)+4S_{4}(r)+S_{5}(r) \\
M_{2}&=\left\langle T_{22}(\boldsymbol{0}),T_{22}(\boldsymbol{r})\right\rangle
=S_{1}(r)+2S_{2}(r) \\
M_{3}&=\left\langle T_{11}(\boldsymbol{0}),T_{22}(\boldsymbol{r})\right\rangle
=S_{1}(r)+S_{3}(r) \\
M_{4}&=\left\langle T_{22}(\boldsymbol{0}),T_{33}(\boldsymbol{r})\right\rangle
=S_{1}(r) \\
M_{5}&=\left\langle T_{12}(\boldsymbol{0}),T_{12}(\boldsymbol{r})\right\rangle
=S_{2}(r)+S_{4}(r) \\
M_{6}&=\left\langle T_{23}(\boldsymbol{0}),T_{23}(\boldsymbol{r})\right\rangle
=S_{2}(r) \\
M_{7}&=\left\langle T_{11}(\boldsymbol{0}),T_{12}(\boldsymbol{r})\right\rangle
=0.%
\end{aligned}%
\]
While the last result shows that the cross-correlation between the 11- and
12-components is always zero, we note that $M_{2}=M_{4}+2M_{6}$\ must hold,
so that only five $M_{i}$'s\ are independent, just as we have five functions
$A_{m}$, $m=1,\dots,5$, see \cite[Section~4.9]{MR3930601}. In principle, we
can determine these five correlations for a specific physical situation.
Thus, when $T_{ij}$\ is the damage tensor for a given resolution (on a given
mesoscale) in a coordinate system defined by $\boldsymbol{n}$, we can use
micromechanics or experiments to determine the best fits of $M_{i}$'s. Thus,
we have a following strategy for determination of the correlation structure $\mathcal{D}_{ij}^{kl}$:

\begin{enumerate}
\item Measure $M_{i}$, $i=1,\dots,6$.

\item Determine $A_{1}=M_{4}$\ and $A_{2}=M_{6}$.

\item Determine $A_{3}=M_{3}-M_{4}$ and $A_{4}=M_{5}-M_{6}$.

\item Determine $A_{5}=M_{1}-M_{3}-4M_{5}+2M_{6}$.
\end{enumerate}

\subsubsection{TRFs as dyadics}

With the components of the covariance tensor of a given TRF determined as
outlined above, the next task is the simulation of TRF. This will be
accomplished through a polyadic representation of tensors. Consider the
rank 2 TRF $C$, over a 2d domain $B$, taking values in the
space $\mathsf{S}^{2}\left( \mathbb{R}^{2}\right) $ which, in turn, provides
input into a stochastic partial differential equation (SPDE) of anti-plane
elasticity with local anisotropy \eqref{stoch-eq}.

If we were to assume $C\left( \omega ,\mathbf{x}\right) =C\left(
\omega ,\mathbf{x}\right)I$, the TRF $C$ would become
a random scalar field, whereupon we would simply have a stochastic
generalization of the anti-plane Navier equation of elastodynamics with a RF
shear coefficient $\mu $. As already argued in Subsection~\ref{sub:motivation} above, micromechanics studies imply that the local anisotropy \textit{must} be present in spatially inhomogeneous random media on the level of Fig.~\ref{fig:2}(c). To account for it, the rank 2 TRF is modeled as a superposition of the isotropic mean with a dyadic product of two vector RFs $\mathbf{a}=a_{i}\mathbf{e}_{i}$\ and $\mathbf{b}=b_{i}\mathbf{e}_{i}$\ ($i=1,2$), written in tensor and matrix notations:%
\begin{equation}
\begin{aligned}
C_{ij}&=\big\langle C_{ij}\big\rangle+sC_{ij}^{\prime },\qquad
\big\langle C_{ij}\big\rangle=\mu \delta _{ij},\qquad s=\mathrm{const}, \\
C_{ij}^{\prime }&=a_{i}b_{j},\qquad\left[ C_{ij}^{\prime }\right] =
\begin{pmatrix}
a_{1}b_{1} & a_{1}b_{2} \\
a_{2}b_{1} & a_{2}b_{2}%
\end{pmatrix}%
.
\end{aligned}
\label{dyadic}
\end{equation}%
Each $a_{i}$\ and $b_{i}$\ ($i=1,2$) is a Gaussian RF: $N\left( \mu
_{d},\sigma _{d}^{2}\right) =N\left( 0,1\right) $. Already the special case with $\mathbf{a}$ and $\mathbf{b}$ fully correlated indicates that taking a random medium with a stiffness TRF $C$ with full anisotropy leads to stronger fluctuations in responses to BVPs than a random medium with a scalar RF $C$, see \cite{Zhang2021}. This is the stepping-stone to studies without any restriction on  $\mathbf{a}$ and $\mathbf{b}$ in hyperbolic as well as parabolic and elliptic SPDEs. In the planned project, we will adopt the dyadic (\ref{dyadic}) for the rank 2 TRF without the restriction on $\mathbf{a}$ and $\mathbf{b}$. One needs codes to examine the spatial randomness effects of dyadic RFs with various
covariances adopted for the individual components $a_{i}$\ and $b_{i}$ in
this and other BVPs.

\paragraph{Polyadic TRFs} Random media occupy domains in $\mathbb{R}^{d}$. With the summation
convention, $\mathbf{x}=x_{i}\mathbf{e}_{i}$ is the vector in $\mathbb{R}%
^{d} $ where the set $\{\,\mathbf{e}_{i}\colon 1\leq i\leq d\,\}$ provides
an orthonormal basis. The symbol $(\mathbb{R}^{d})^{r}$, where $r$ is a
nonnegative integer, denotes the linear space $\mathbb{R}^{1}$ if $r=0$ and
the Cartesian product of $r$ copies of the set $\mathbb{R}^{d}$ otherwise. A
function $f\colon (\mathbb{R}^{d})^{r}\rightarrow \mathbb{R}$ is called the
\emph{polyadic}\index{polyadic function} (\emph{monadic}\index{monadic function} if $r=1$, \emph{dyadic}\index{dyadic function} if $r=2$, \emph{triadic}\index{triadic function} if $r=3$, and so on), if either $r=0$ and $f$ is linear or $r\geq 1$ and $f$ is $r$-linear, that is, for all integers $i$ with $1\leq i\leq d$, for all $\mathbf{x}_{1}$, $\mathbf{x}_{2}$, \dots , $\mathbf{x}_{i-1}$, $\mathbf{x}$, $\mathbf{y}$, $\mathbf{x}_{i+1}$, \dots , $\mathbf{x}_{r}$ in $\mathbb{R}^{d}$, and for all $\alpha $, $\beta \in \mathbb{R}$ we have
\[
\begin{aligned}
&f(\mathbf{x}_1,\mathbf{x}_2,\dots,\mathbf{x}_{i-1},
\alpha\mathbf{x}+\beta\mathbf{y},\mathbf{x}_{i+1},\dots,\mathbf{x}_r)\\
&\quad=\alpha f(\mathbf{x}_1,\mathbf{x}_2,\dots,\mathbf{x}_{i-1},
\mathbf{x},\mathbf{x}_{i+1},\dots,\mathbf{x}_r)\\
&\quad+\beta f(\mathbf{x}_1,\mathbf{x}_2,\dots,\mathbf{x}_{i-1},
\mathbf{y},\mathbf{x}_{i+1},\dots,\mathbf{x}_r).
\end{aligned}
\]

The number $r$ is called the \emph{rank}\index{rank of a polyadic} of a polyadic. Denote a polyadic
with a capital sans serif letter with a superscript preceding it to indicate
its rank, for example, ${}^{2}\mathsf{T}$ is a dyadic. The set of all
polyadics is a real linear space of dimension $d^{r}$. Indeed, if $r=0$,
then $(\mathbb{R}^{d})^{r}=\mathbb{R}^{1}$, $d^{r}=1$, and a linear
functional $f\colon \mathbb{R}^{1}\rightarrow \mathbb{R}$ has the form $%
f(x_{1})=y_{1}x_{1}$ with $y_{1}\in \mathbb{R}^{1}$. Otherwise, if $r\geq 1$%
, we define a special polyadic as follows.

A \emph{polyad}\index{polyad} (\emph{monad},\index{monad} \emph{dyad},\index{dyad} \emph{triad},\index{triad} \dots ) is the
polyadic $\mathbf{x}_{1}\mathbf{x}_{2}\cdots \mathbf{x}_{r}$ given by
\[
\mathbf{x}_{1}\mathbf{x_{2}}\cdots \mathbf{x}_{r}(\mathbf{y}_{1},\mathbf{%
y_{2}},\dots ,\mathbf{y}_{r})=(\mathbf{x}_{1}\cdot \mathbf{y}_{1})(\mathbf{x}%
_{2}\cdot \mathbf{y}_{2})\cdots (\mathbf{x}_{r}\cdot \mathbf{y}_{r})
\]
for all $(\mathbf{y}_{1},\mathbf{y}_{2},\dots ,\mathbf{y}_{r})\in (\mathbb{R}%
^{d})^{r}$, where $\mathbf{x}_{1}$, $\mathbf{x_{2}}$, \dots , $\mathbf{x}%
_{r} $ are arbitrary vectors in $\mathbb{R}^{d}$.

The $d^{r}$ polyads $\mathbf{e}_{i_{1}}\mathbf{e}_{i_{2}}\cdots \mathbf{e}%
_{i_{r}}$ with $1\leq i_{j}\leq d$ for all $j$ with $1\leq j\leq r$, form a
basis in the linear space of all polyadics. With the \emph{$r$-dot product}
the above basis becomes orthonormal. In particular, if $r=1$, then the $1$%
-dot product of two monads $\mathbf{x}$ and $\mathbf{y}$ coincides with the
scalar product $\mathbf{x\cdot y}$, the $2$-dot product of two dyads is $%
\mathbf{x}_{1}\mathbf{x_{2}}\colon \mathbf{y}_{1}\mathbf{y_{2}}=(\mathbf{x}%
_{1}\cdot \mathbf{y}_{1})(\mathbf{x}_{2}\cdot \mathbf{y}_{2})$, and the $3$%
-dot product of two triads by $\mathbf{x}_{1}\mathbf{x}_{2}\mathbf{x}%
_{3}\,\vdots \,\mathbf{y}_{1}\mathbf{y}_{2}\mathbf{y}_{3}=(\mathbf{x}%
_{1}\cdot \mathbf{y}_{1})(\mathbf{x}_{2}\cdot \mathbf{y}_{2})(\mathbf{x}%
_{3}\cdot \mathbf{y}_{3})$. A polyad $\mathbf{x}_{1}\mathbf{x_{2}}\cdots
\mathbf{x}_{r}$ is identified with the tensor product $\mathbf{x}_{1}\otimes
\mathbf{x_{2}}\otimes \cdots \otimes \mathbf{x}_{r}$ and the linear space of
all polyadics with the tensor product $(\mathbb{R}^{d})^{\otimes r}$ of $r$%
~copies of the space $\mathbb{R}^{d}$. We can prove that, for $r\geq 1$, any
$r$-adic can be represented as a sum of not more than $d^{r-1}$ $r$-ads.

If, under a shift map $\mathbb{R}^{3}\rightarrow \mathbb{R}^{3}$, $\mathbf{x}%
\mapsto \mathbf{x}+\mathbf{y}$, the \emph{one-point \emph{correlation}
polyadic} $\langle {}^{r}\mathsf{T}(\mathbf{x})\rangle $ and the \emph{%
two-point covariance polyadic}
\[
\langle {}^{r}\mathsf{T}(\mathbf{x}),{}^{r}\mathsf{T}(\mathbf{y})\rangle
=\left\langle ({}^{r}\mathsf{T}(\mathbf{x})-\langle {}^{r}\mathsf{T}(\mathbf{%
x})\rangle ),({}^{r}\mathsf{T}(\mathbf{y})-\langle {}^{r}\mathsf{T}(\mathbf{y}%
)\rangle )\right\rangle  \label{correlation polyadics}
\]
do not change, a second-order TRF is WSS. It is WSSI if its correlation
polyadics satisfy
\[
\langle {}^{r}\mathsf{T}(g\mathbf{x})\rangle =g\cdot \langle {}^{r}\mathsf{T}%
(\mathbf{x})\rangle ,\qquad \langle ^{r}\mathsf{T}(g\mathbf{x}),{}^{r}%
\mathsf{T}(g\mathbf{y})\rangle =(\rho \otimes \rho )(g)\langle {}^{r}\mathsf{%
T}(g\mathbf{x}),{}^{r}\mathsf{T}(g\mathbf{y})\rangle ,
\label{WSSI correlation polyadics}
\]
where $g$ $\mapsto \mathsf{S}^{2}(g)$ is an orthogonal representation of the
group $\mathrm{O}(3)$. Some open issues are:

While it is clear how to construct the polyadic TRFs from scalar random
fields, the resulting statistics and correlation structures are open
research topics.

\section{Random cross-sections of homogeneous vector bundles}\label{sec:sections}

This section is motivated by the following observation: the Cosmic Microwave Background can be described mathematically as a random cross-section of a homogeneous vector bundle over the sky sphere, see \cite{MR3170229,MR2737761,MR2884225,MR2840154,preprint2021}.

A \emph{random cross-section}\index{cross-section!random} of a homogeneous vector bundle $(E,B,\pi)$ is a collection of random vectors $\{\,s(x,\omega)\colon x\in B\,\}$ such that $s(x,\omega)$ is a random vector in $\pi^{-1}(x)$. In what follows, we use the terms ``random cross-section''  and ``random field'' for $s(x)$ interchangeably.

\begin{definition}
A random cross-section $s(x)$ of a homogeneous vector bundle $(E,B,\pi)$ is called \emph{strictly isotropic}\index{cross-section!random!strictly isotropic} if for any positive integer $n$, for any $n$ distinct points $x_1$, \dots, $x_n$ in $B$, and for arbitrary $g\in G$, the finite-dimensional distributions $(s(x_1),\dots,s(x_n))$ and $(s(gx_1),\dots,s(gx_n))$ are identical.
\end{definition}

To simplify the subsequent notation, introduce the following agreements. Let $tL=L$ and $t\mathbf{l}=\mathbf{l}$ whenever $\mathbb{K}\in\{\mathbb{R},\mathbb{H}\}$. For a \emph{complex} Hilbert space $V$ with inner product $(\cdot,\cdot)$, and for $\mathbf{v}\in V$, let $t\mathbf{v}$ be the element of the space $tV$ satisfying $(t\mathbf{v},\mathbf{v}_1)=(\mathbf{v}_1,\mathbf{v})$ for each $\mathbf{v}_1\in V$.

Assume that the random cross-section $s(x)$ of a homogeneous vector bundle $(E,B,\pi)$ with fibre $L_0$ is \emph{second-order},\index{cross-section!random!second-order} that is, $\mathsf{E}[\|s(x)\|^2_{L_0}]<\infty$, $x\in B$. If, in addition, such a field is strictly isotropic, then it is wide-sense isotropic.

\begin{definition}
A second-order random cross-section $s(x)$ is called \emph{wide-sense isotropic}\index{cross-section!random!wide-sense isotropic} if and only if its one-point correlation tensor\index{correlation tensor!one-point} $\langle s(x)\rangle=\mathsf{E}[s(x)]$ and the two-point correlation tensor\index{correlation tensor!two-point}
$\langle s(x),s(y)\rangle=\mathsf{E}[t(s(x)-\langle s(x)\rangle)\otimes(s(y)-\langle s(y)\rangle)]$ are $G$-invariant, that is, for any $g\in G$ we have
\begin{equation}\label{eq:41}
\langle s(gx)\rangle=\langle s(x)\rangle,\qquad
\langle s(gx),s(gy)\rangle=\langle s(x),s(y)\rangle.
\end{equation}
\end{definition}

Note that the one-point correlation tensor of a  random field $s(x)$ is $\mathbb{K}$-valued, while the two-point one is $\mathbb{K}'$-valued. See Section~\ref{sec:language} for explanation of the above symbols.

In what follows, we consider only wide-sense isotropic random fields in homogeneous vector bundles and call them just isotropic.

The next step is to define a mean-square continuous random field in a homogeneous vector bundle. The standard definition, Equation~\eqref{eq:33}, does not work, because if $x$, $y\in B$ and $x\neq y$, then the random vectors $s(x)$ and $s(y)$ belong to different linear spaces and cannot be subtracted one from another. See a discussion of this problem in \cite{MR2884225,MR2977490}. Here we propose an alternative definition.

Let $L^2(\Omega,L_0)$ be the Hilbert space of $L_0$-valued random vectors $\mathbf{X}$ with $\mathsf{E}[\|\mathbf{X}\|^2_{L_0}]<\infty$ and inner product
\begin{equation}\label{eq:34}
(\mathbf{X},\mathbf{Y})=\mathsf{E}[(\mathbf{X},\mathbf{Y})_{L_0}].
\end{equation}
Let $H$ be the closed linear span of the set of random vectors $\{\,s(x)\colon x\in B\,\}$ in $L^2(\Omega,L_0)$. For each $g\in G$, there is a unique bounded linear operator $A(g)$ in $H$ with $A(g)T(x)=s(g^{-1}x)$. Moreover, the correspondence $g\mapsto A(g)$ is a representation of $G$ in $H$.

\begin{definition}
A random cross-section $s(x)$ is called \emph{mean-square continuous},\index{cross-section!random!mean-square continuous} if the above representation is continuous.
\end{definition}

Moreover, it is not difficult to prove that a random field $s(x)$ with constant one-point correlation tensor is isotropic if and only if the inner product \eqref{eq:34} is $G$-invariant.

How to find a spectral expansion of an isotropic random cross-section? First, exactly as in the case of ordinary random fields, we conclude that the one-point correlation tensor $\langle s(x)\rangle$ is an arbitrary element of the isotypic subspace of the space $L_0$ that corresponds to the trivial representation of the group $H$.

Let ${}_{L^n,p}Y_{ijkLm}(x)$ be the cross-sections of the orthonormal basis constructed in Subsection~\ref{sub:basis}. In this notation, the index $L^n$ runs over the set of all irreducible representations of the group $H$, for which $\Hom_{\mathbb{K}H}(L^n,L_0)\neq\{\mathbf{0}\}$, the index $p$ enumerates the copies of the representation $L^n$ inside $L_0$, the index $L$ runs over the set of all irreducible representations of the group $G$, for which $\Hom_{\mathbb{K}H}(\res^G_HL,L^n)\neq\{\mathbf{0}\}$, the index $i$ enumerates the copies of the representation $L^n$ inside the representation $\Hom_{\mathbb{K}H}(\res^G_HL,L^n)$, the index $j$ enumerates identical blocks inside a certain matrix, the index $k$ runs from $1$ to the dimension of a certain division algebra over $\mathbb{K}'$, and the index $m$ enumerates the basis vectors in $L$.

Let $d_L$ be the multiplicity of the representation $L$ in the induced representation $L^2(E)$, and let $\{\,{}_{L^n,p}\mathbf{a}_{Lm}\colon L\in\hat{G}_{\mathbb{K}},1\leq m\leq\dim L\,\}$ be a sequence of centred uncorrelated $\mathbb{K}^{d_L}$-valued random vectors such that the covariance matrix of the vector ${}_{L^n,p}\mathbf{a}_{Lm}$ does not depend on $m$ and for fixed $L^n$ and $p$ we have
\[
\sum_{L\in\hat{G}_{\mathbb{K}}}\dim L\mathsf{E}[\|{}_{L^n,p}\mathbf{a}_{L1}\|^2]<\infty.
\]

\begin{theorem}\label{th:section}
A random field $s(x)$ is isotropic if and only if it has the form
\[
s(x)=s+\sum{}_{L^n,p}Y_{ijkLm}(x)({}_{L^n,p}\mathbf{a}_{Lm})_{ijk}.
\]
\end{theorem}

\begin{proof}
Let $(E,\pi,B\times B)$ be the homogeneous vector bundle where the representation of the group $G\times G$ induced by the representation $tL_0\otimes L_0$ of the subgroup $H\times H$ acts. Observe that the two-point correlation tensor
$\langle s(x),s(y)\rangle$ of the random field $s(x)$ is a continuous cross-section of the above bundle. The cross-sections $t{}_{L^n,p}Y_{ijkLm}(x)\otimes{}_{L'^n,p'}Y_{i'j'k'L'm'}(y)$
form an orthonormal basis in the Hilbert space $L^2(E)$.

We note that $t{}_{L^n,p}Y_{ijkLm}(x)={}_{L^n,p}Y_{ijk(tL)m}(x)$, and, for a fixed $L\in\hat{G}_{\mathbb{K}}$, the functions ${}_{L^n,p}Y_{ijkLm}(x)$ form an orthonormal basis in the isotypic subspace of the space $L^2(E)$ that correspond to $L$, call it $\mathcal{H}_L$. Moreover, by the construction of this basis, for fixed $L^n$, $p$, $i$, $j$, and $k$, we have
\[
{}_{L^n,p}Y_{ijkLm}(x)=\sum_{q=1}^{\dim L}\theta^L_{mq}(g){}_{L^n,p}Y_{ijkLq}(x).
\]
Let $C_{LL'}$ be the matrix of the Fourier coefficients of the cross-section $\langle s(x),s(y)\rangle$ in the above basis. The second equation in condition~\eqref{eq:41} becomes
\[
C_{LL'}\theta^L=\theta^{L'}C_{LL'},
\]
that is, $C_{LL'}\in\Hom_{\mathbb{K}G}(\mathcal{H}_L,\mathcal{H}_{L'})$. Schur's Lemma immediately shows that $C_{LL'}$ is the zero matrix whenever $tL\neq L'$. Otherwise, the matrix $C^{tLL}$ is symmetric when $\mathbb{K}'=\mathbb{R}$, Hermitian when $\mathbb{K}'=\mathbb{C}$, consists of square sub-matrices with $\dim V$ rows, and each sub-matrix is block-diagonal, with blocks described in Subsection~\ref{sub:basis}. Because the blocks of sizes $2\times 2$ and $4\times 4$ are skew-symmetric, they must be diagonal and multiple to the identity matrix. In the space $\mathcal{H}_L$, the Fourier coefficients form a square matrix with $d^2_L$ blocks, where each block is a multiple of the $\dim L\times \dim L$ identity matrix. The elements of each block which are located at the intersection of the $m$th row and the $m$th column, form the covariance matrix of the random vector ${}_{L^n,p}\mathbf{a}_{Lm}$. The rest is obvious.
\end{proof}

\section{Applications to cosmology}\label{sec:cosmology}

We make a simplification in Example~\ref{ex:standard}. Suppose that the spacetime~$M$ is the Minkowski one, that is, the nonzero components of the bilinear form $g(x)$ are $g_{00}(x)=-c^2$, $g_{11}(x)=g_{22}(x)=g_{33}(x)=1$, where $c$ is the speed of light. Physically, this means that we neglect possible secondary effects of the CMB (the effects that appear after the last scattering of the CMB photons).

Consider an observer at the event $\mathbf{0}\in M$. Light rays through her eye correspond to null straight lines through $\mathbf{0}$. Their past directions with $t<0$ constitute her field of vision. She can imagine herself at the centre of her sphere of vision, a unit sphere $S^2$, the celestial sphere, see Example~\ref{ex:standard}. It can be regarded in different ways.

\begin{example}[The $(\Theta,\mathcal{Q}\pm\mathrm{i}\mathcal{U},\mathcal{V})$ model]
Regard the celestial sphere as the complex projective line $\mathbb{C}P^1$. Let $T_{\zeta}\mathbb{C}P^1$ be the tangent space at a point $\zeta\in\mathbb{C}P^1$. According to Example~\ref{ex:standard}, the electric field of the CMB at the point $\zeta$ belongs to the real linear space  $rT_{\zeta}\mathbb{C}P^1$. The space $crT_{\zeta}\mathbb{C}P^1$ is the \emph{complexified tangent space}\index{tangent space!complexified} at $\zeta\in\mathbb{C}P^1$, see \cite{MR2102340}. We have $crT_{\zeta}\mathbb{C}P^1=T^{1,0}_{\zeta}\mathbb{C}P^1\oplus T^{0,1}_{\zeta}\mathbb{C}P^1$, where the first (resp. the second) term is the \emph{holomorphic}\index{tangent space!holomorphic} (resp. \emph{anti-holomorphic})\index{tangent space!anti-holomorphic} \emph{tangent space} at $\zeta\in\mathbb{C}P^1$. The projection $\pi(crT_{\zeta}\mathbb{C}P^1)=\zeta$ determines a vector bundle, which is the vector bundle $(\Sigma\mathbb{C}P^1,\mathbb{C}P^1,\pi)$ with fibre $V_1\oplus tV_1$ of Example~\ref{ex:spinweight}. That is, the electric field of the CMB is a spinor field.

Identify a point $\zeta\in\mathbb{C}P^1$ with the point $\mathbf{x}\in S^2\subset\mathbb{R}^3$ with coordinates given by \cite[Equation~(1.2.8)]{MR917488}:
\[
x=\frac{\zeta+\zeta^*}{\zeta\zeta^*+1},\qquad y=\frac{\zeta-\zeta^*}{\mathrm{i}(\zeta\zeta^*+1)},\qquad z=\frac{\zeta\zeta^*-1}{\zeta\zeta^*+1},
\]
where the point at infinity $\zeta=\infty$ is identified with the north pole of the sphere. If we observe an electromagnetic wave of the CMB at a point $\mathbf{x}\in S^2$, it propagates in the direction $\mathbf{n}=-\mathbf{x}$.
Choose two vectors $\bm{\varepsilon}^{(1)}$ and $\bm{\varepsilon}^{(2)}$ in the space $crT_{\zeta}\mathbb{C}P^1$ in such a way that the vectors $(\bm{\varepsilon}^{(1)},\bm{\varepsilon}^{(2)},\mathbf{n})$ form a right-handed orthonormal basis. The pair $(\bm{\varepsilon}^{(1)},\bm{\varepsilon}^{(2)})$ is called the \emph{linear polarisation basis}.\index{polarisation basis!linear} It is customary to represent the electric field of the wave in the \emph{complex form}
\[
\mathbf{E}(t)=E_1(t)\bm{\varepsilon}^{(1)}+E_2(t)\bm{\varepsilon}^{(2)},\qquad E_1(t),E_2(t)\in\mathbb{C}.
\]
According to \cite[p.~305]{MR1651651}:
\begin{quote}
\dots only the real part of the fields has a physical meaning. We make use of the complex notation because many formulas are thereby simplified\dots
\end{quote}

A CMB detector observes the time-averaged electric field $\mathbf{E}=E_1\bm{\varepsilon}^{(1)}+E_2\bm{\varepsilon}^{(2)}$. Following \cite{durrer_2020}, introduce the Hermitian $2\times 2$ matrix $\tilde{\mathcal{P}}=\left(
\begin{smallmatrix}
  |E_1|^2 & E_1^*E_2 \\
  E_2^*E_1 & |E_2|^2
\end{smallmatrix}
\right)$. The standard orthonormal basis in the real $4$-dimensional linear space of Hermitian $2\times 2$ matrices with inner product $(\tilde{\mathcal{P}}_1,\tilde{\mathcal{P}}_2)=\tr(\tilde{\mathcal{P}}_1\tilde{\mathcal{P}}_2)$ consists of the matrices $\frac{1}{\sqrt{2}}\sigma^{(i)}$, $0\leq i\leq 3$, where $\sigma^{(i)}$ are the \emph{Pauli matrices}\index{Pauli matrices} $\sigma^{(0)}=\left(
\begin{smallmatrix}
  1 & 0 \\
  0 & 1
\end{smallmatrix}
\right)$, $\sigma^{(1)}=\left(
\begin{smallmatrix}
  0 & 1 \\
  1 & 0
\end{smallmatrix}
\right)$, $\sigma^{(2)}=\left(
\begin{smallmatrix}
  0 & -\mathrm{i} \\
  \mathrm{i} & 0
\end{smallmatrix}
\right)$, $\sigma^{(3)}=\left(
\begin{smallmatrix}
  1 & 0 \\
  0 & -1
\end{smallmatrix}
\right)$. The expansion coefficients are called the \emph{Stokes parameters}\index{Stokes parameters} and defined as follows:
\begin{equation}\label{eq:45}
\begin{aligned}
I(\zeta)&=\tr(\tilde{\mathcal{P}}(\zeta)\sigma^{(0)}), & U(\zeta)&=\tr(\tilde{\mathcal{P}}(\zeta)\sigma^{(1)}),\\ V(\zeta)&=\tr(\tilde{\mathcal{P}}(\zeta)\sigma^{(2)}), & Q(\zeta)&=\tr(\tilde{\mathcal{P}}(\zeta)\sigma^{(3)}).
\end{aligned}
\end{equation}
The matrix $\tilde{\mathcal{P}}$ takes the form
\[
\tilde{\mathcal{P}}(\zeta)=\frac{1}{2}(I(\zeta)\sigma^{(0)}
+U(\zeta)\sigma^{(1)}+V(\zeta)\sigma^{(2)}+Q(\zeta)\sigma^{(3)}).
\]

The matrix $\tilde{\mathcal{P}}$ lies in the space
\begin{equation}\label{eq:47}
L_0=t(V_1\oplus tV_1)\otimes(V_1\oplus tV_1)=(tV_1\oplus V_1)\otimes(V_1\oplus tV_1)=2cU_0\oplus V_2\oplus tV_2,
\end{equation}
and the group $H$, the subgroup of diagonal matrices in $\mathrm{SU}(2)$, acts on $\tilde{\mathcal{P}}$ by $h\cdot\tilde{\mathcal{P}}=h\tilde{\mathcal{P}}h^{-1}$.

The theory of irreducible unitary representations of the group $\mathrm{SU}(2)$ is developed in numerous sources, we mention \cite{MR889252,MR1153249,MR0047664,MR3969956,MR793377} among others. The set $\hat{G}_{\mathbb{C}}$ contains the unitary irreducible representations of real type\index{representation!irreducible complex of $\mathrm{SU}(2)$!$cU_{\ell}$} $\{\,cU_{\ell}\colon\ell\geq 0\,\}$ and of quaternionic type\index{representation!irreducible complex of $\mathrm{SU}(2)$!$c'W_{\ell}$} $\{\,c'W_{\ell}\colon\ell=\frac{1}{2},\frac{3}{2},\dots\,\}$. To determine the structure of the restrictions $\res^G_HcU_{\ell}$ and $\res^G_Hc'W_{\ell}$, we use the \emph{characters}.

By definition, the \emph{character}\index{character} of a complex finite-dimensional representation $V$ is a function $G\to\mathbb{C}$, given by $\chi_V(g)=\tr\theta^V(g)$. The character of a real finite-dimensional representation $U$ (resp. quaternionic finite-dimensional representation $W$) is the character of the complex representation $cU$ (resp. $c'W$). We have $\chi_{V_1\oplus V_2}(g)=\chi_{V_1}(g)+\chi_{V_2}(g)$, and similarly for real and quaternionic representations.

According to \cite{MR1143783}, the restriction of the character of an irreducible representations with index $\ell$ to the subgroup $H$ is
\[
\chi_{cU_{\ell}}\left(\left(
\begin{smallmatrix}
  \alpha & 0 \\
  0 & \alpha^*
\end{smallmatrix}
\right)\right)=\chi_{c'W_{\ell}}\left(\left(
\begin{smallmatrix}
  \alpha & 0 \\
  0 & \alpha^*
\end{smallmatrix}
\right)\right)=\alpha^{2\ell}+\alpha^{-2\ell-2}
+\cdots+\alpha^{-2\ell}.
\]
In other words, the restriction $\res^G_HcU_{\ell}$ has the form
\begin{equation}\label{eq:35}
\res^G_HcU_{\ell}=tV_{2\ell}\oplus tV_{2\ell-2}\oplus\cdots\oplus tV_2\oplus cU_0\oplus V_2\oplus\cdots\oplus V_{2\ell},
\end{equation}
while the restriction $\res^G_Hc'W_{\ell}$ has the form
\begin{equation}\label{eq:36}
\res^G_Hc'W_{\ell}=tV_{2\ell}\oplus tV_{2\ell-2}\oplus\cdots\oplus tV_1\oplus V_1\oplus\cdots\oplus V_{2\ell}.
\end{equation}

The first copy of the irreducible representation $cU_0$ of the group $H$ acts in the complex linear space $L_{01}$ generated by the matrix $\frac{1}{\sqrt{2}}\sigma^{(0)}$. Schur's Lemma gives $D=\Hom_{\mathbb{C}H}(cU_0,cU_0)=\mathbb{C}$, and we have $\dim_{\mathbb{C}}D=1$. By Equations \eqref{eq:35} and \eqref{eq:36}, $\dim_{\mathbb{C}}\Hom_{\mathbb{C}H}(\res^G_HcU_{\ell},cU_0)=1$, while $\dim_{\mathbb{C}}\Hom_{\mathbb{C}H}(\res^G_Hc'W_{\ell},cU_0)=0$. The space $\Hom_{\mathbb{C}H}(\res^G_HcU_{\ell},cU_0)$ contains $n_{cU_{\ell}}=1$ copy of the representation $cU_0$, while the space $\Hom_{\mathbb{C}H}(\res^G_Hc'W_{\ell},cU_0)$ contains no copies of the above representation. The matrix $f^{111}$ has $1$ row and $2\ell+1$ columns. We have $f^{111}_{1\,\ell+1}=1$, the remaining entries are zeroes.

Choose a model for $cU_{\ell}$,\index{representation!irreducible complex of $\mathrm{SU}(2)$!$cU_{\ell}$} described in \cite{MR1143783} and in many other sources. The linear space $cU_{\ell}$ is the space of all polynomials in a real variable $x$ of degree $2\ell$. The group $G$ acts by
\[
\left(
\begin{smallmatrix}
  \alpha & \beta \\
  -\beta^* & \alpha^*
\end{smallmatrix}
\right)\cdot p(x)=(\beta x+\alpha^*)^{2\ell}p\left(\frac{\alpha x-\beta^*}{\beta x+\alpha^*}\right),\qquad p\in cU_{\ell}.
\]
The orthonormal basis in $cU_{\ell}$ is formed by the polynomials
\[
p_m(x)=\frac{x^{\ell-m}}{\sqrt{(\ell-m)!(\ell+m)!}},\qquad-\ell\leq m\leq\ell.
\]
The matrix $\left(
\begin{smallmatrix}
  \alpha & 0 \\
  0 & \alpha^*
\end{smallmatrix}
\right)\in H$ acts on $p_m(x)$ by multiplication by $\alpha^{-2m}$, that is, by one of the representations in the right hand side of \eqref{eq:35}. The Frobenius reciprocity isomorphism maps the matrix $f^{111}$ to the function  $F^{111}\colon cU_{\ell}\to C(G,L_{01})$. This function maps an element $g\in G$ to the zeroth component of the vector $g^{-1}\cdot p\in cU_{\ell}$ for $p\in U_{\ell}$. In particular, $F^{111}(p_m)$ is the $L_{01}$-valued function on $G$ given by $D^{*(\ell)}_{m0}(g)\frac{1}{\sqrt{2}}\sigma^{(0)}$, where $D^{(\ell)}_{m0}(g)$ are the \emph{Wigner $D$-functions},\index{Wigner $D$-functions} the matrix entries of the representations $cU_{\ell}$. After normalisation, the cross-sections ${}_{L_{01}}Y_{ijkLm}(gH)$ become
\[
{}_{L_{01}}Y_{ijkLm}(gH)=\sqrt{\frac{2\ell+1}{4\pi}}D^{*(\ell)}_{m0}(gH)\frac{1}{\sqrt{2}}\sigma^{(0)},
\]
and we recognise the \emph{spherical harmonics}\index{spherical harmonics!complex-valued} $Y_{\ell,m}(\zeta)$, see \cite[Equation~(A.4.43)]{durrer_2020}. Theorem~\ref{th:section} gives
\begin{equation}\label{eq:42}
I(\zeta)=\sum_{\ell=0}^{\infty}\sum_{m=-\ell}^{\ell}a_{\ell m}Y_{\ell,m}(\zeta),
\end{equation}
where $a_{\ell m}$ are complex-valued uncorrelated random variables with $\mathsf{E}[a_{\ell m}]=0$ for $\ell\geq 1$, $\mathsf{E}[|a_{\ell m}|^2]=C_{\ell}$, and
\[
\sum_{\ell=0}^{\infty}(2\ell+1)C_{\ell}<\infty.
\]

Note that we absorbed the constant cross-section $s$ into the expected value of $a_{00}$. The expansion \eqref{eq:42} was proved in \cite{Obukhov1947} long before the discovery of the relic radiation. The spherical harmonics $Y_{\ell,m}(\zeta)$ have the form
\[
\begin{aligned}
Y_{\ell,m}(\zeta)&=\frac{(-1)^{\ell-m}}{\sqrt{4\pi}\ell!}\sqrt{(2\ell+1)(\ell+m)!(\ell-m)!}
(1+\zeta\zeta^*)\\
&\quad\times\sum_{p=\max\{0,m\}}^{\ell}\binom{\ell}{p}\binom{\ell}{p-m}\zeta^p
(\zeta^*)^{p-m}.
\end{aligned}
\]
It follows from this equation that $Y_{\ell,-m}(\zeta)=(-1)^mY_{\ell,m}^*(\zeta)$. We see that the random field $I(\zeta)$ is real-valued if and only if $a_{\ell\,-m}=(-1)^ma^*_{\ell m}$.

The Stokes parameters, as defined above, have units of intensity $I$. It is conventional to express the CMB fluctuations in terms of the \emph{brightness temperature},\index{brightness temperature} $T$. Denote the Stokes parameters after re-scaling by $T$, $U_T$, $V_T$, and $Q_T$. Denote $T_0=\mathsf{E}[a_{00}]$, the averaged temperature of the CMB. Finally, the \emph{dimensionless Stokes parameters}\index{Stokes parameters!dimensionless} are
\[
\Theta=\frac{T-T_0}{T_0},\qquad\mathcal{U}=\frac{U_T}{T_0},\qquad\mathcal{V}
=\frac{V_T}{T_0},\qquad\mathcal{Q}=\frac{Q_T}{T_0}.
\]
Note that our definition of $\mathcal{U}$ and $\mathcal{Q}$ differs from \cite[Equation~(5.8)]{durrer_2020} by a factor of $\frac{1}{4}$.

The number $\ell$ is called the \emph{multipole number}\index{multipole number} (the \emph{monopole}\index{monopole} if $\ell=0$, the \emph{dipole}\index{dipole} if $\ell=1$). The number $m$ is called the \emph{azimuthal number},\index{azimuthal number} the number $C_{\ell}$ the \emph{multipole moment}.\index{multipole moment} The \emph{multipole expansion}\index{multipole expansion} of the dimensionless temperature deviation $\Theta(\zeta)$ generally starts at $\ell=2$:
\[
\Theta(\zeta)=\sum_{\ell=2}^{\infty}\sum_{m=-\ell}^{\ell}a^{\Theta}_{\ell m}Y_{\ell,m}(\zeta).
\]
This is because with our definition of $\Theta$, the monopole term vanishes, while the dipole term is affected by our own motion across space and dominates over the intrinsic cosmological dipole term.

The analysis for $\mathcal{V}$ is similar and gives
\[
\mathcal{V}(\zeta)=\sum_{\ell=0}^{\infty}\sum_{m=-\ell}^{\ell}a^{\mathcal{V}}_{\ell m}Y_{\ell,m}(\zeta).
\]

Denote
\[
\mathcal{P}(\zeta)=\frac{1}{2}(\mathcal{U}(\zeta)\sigma^{(1)}+\mathcal{Q}(\zeta)\sigma^{(3)})
=\frac{1}{2}
\begin{pmatrix}
  \mathcal{Q}(\zeta) & \mathcal{U}(\zeta) \\
  \mathcal{U}(\zeta) & -\mathcal{Q}(\zeta)
\end{pmatrix}
.
\]
Rotate the basis $\bm{\varepsilon}^{(1)}$ and $\bm{\varepsilon}^{(2)}$ of the space $crT_{\zeta}\mathbb{C}P^1$ by an angle $\psi$ around the direction $\mathbf{n}$. The new basis is
\[
\bm{\varepsilon}^{(1)'}=\cos\psi\bm{\varepsilon}^{(1)}+\sin\psi\bm{\varepsilon}^{(2)},
\qquad\bm{\varepsilon}^{(2)'}=-\sin\psi\bm{\varepsilon}^{(1)}+\cos\psi\bm{\varepsilon}^{(2)}.
\]
In the rotated basis, the coefficients of the electric field become
\[
E_1=E_1\cos+E_2\sin\psi,\qquad E_2'=-E_1\sin\psi+E_2\cos\psi.
\]
Equation~\eqref{eq:45} gives
\[
I=|E_1|^2+|E_2|^2,\quad Q=|E_1|^2-|E_2|^2,\quad U=2\RE(E^*_1E_2),\quad V=2\IM(E^*_1E_2),
\]
and we see that under the above rotation
\[
I'=I,\quad V'=V,\quad Q'=Q\cos(2\psi)-U\sin(2\psi),\quad U'=U\cos(2\psi)+Q\sin(2\psi).
\]
It follows that $Q'\pm\mathrm{i}U'=\mathrm{e}^{\pm 2\mathrm{i}\psi}(Q+\mathrm{i}U)$. In other words, $O+\mathrm{i}U\in V_2$, $Q-\mathrm{i}U\in tV_2$.

The irreducible representation $V_2$ of the group $H$ acts in the complex linear space $L_2$ generated by the matrix $\frac{1}{2}(\sigma^{(3)}+\mathrm{i}\sigma^{(1)})$. Schur's Lemma gives $D=\Hom_{\mathbb{C}H}(V_2,V_2)=\mathbb{C}$, and we have $\dim_{\mathbb{C}}D=1$. By Equations \eqref{eq:35} and \eqref{eq:36}, $\dim_{\mathbb{C}}\Hom_{\mathbb{C}H}(\res^G_HcU_{\ell},V_2)=1$ for $\ell\geq 2$, while $\dim_{\mathbb{C}}\Hom_{\mathbb{C}H}(\res^G_Hc'W_{\ell},V_2)=0$. The space $\Hom_{\mathbb{C}H}(\res^G_HcU_{\ell},V_2)$ contains $n_{cU_{\ell}}=1$ copy of the representation $V_2$, while the space $\Hom_{\mathbb{C}H}(\res^G_Hc'W_{\ell},cU_0)$ contains no copies of the above representation. The matrix $f^{111}$ has $1$ row and $2\ell+1$ columns. We have $f^{111}_{1\,\ell+1}=1$, the remaining entries are zeroes.

This time, the Frobenius reciprocity isomorphism maps the matrix $f^{111}$ to the function $F^{111}\colon V_2\to C(G,L_2)$. This function maps an element $g\in G$ to the zeroth component of the vector $g^{-1}\cdot p\in V_2$ for $p\in U_{\ell}$. In particular, $F^{111}(p_m)$ is the $L_2$-valued function on $G$ given by $D^{*(\ell)}_{m2}(g)\frac{1}{2}(\sigma^{(3)}+\mathrm{i}\sigma^{(1)})$. After normalisation, the cross-sections ${}_{L_2}Y_{ijkLm}(gH)$ become
\[
{}_{L_2}Y_{ijkLm}(gH)=\sqrt{\frac{2\ell+1}{4\pi}}D^{*(\ell)}_{m2}(gH)
\frac{1}{2}(\sigma^{(3)}+\mathrm{i}\sigma^{(1)}).
\]
with the help of \cite[Equation~(A.4.95)]{durrer_2020}, we recognise the \emph{spin 2 spherical harmonics}\index{spherical harmonics!spin 2!complex-valued} ${}_2Y_{\ell,m}(\zeta)$. Theorem~\ref{th:section} gives
\[
(\mathcal{Q}+\mathrm{i}\mathcal{U})(\zeta)=\sum_{\ell=2}^{\infty}\sum_{m=-\ell}^{\ell}a_{\ell m}^{(2)}\,{}_2Y_{\ell,m}(\zeta),
\]
where $a_{\ell m}^{(2)}$ are centred complex-valued uncorrelated random variables with
\[
\mathsf{E}[|a_{\ell m}^{(2)}|^2]=C^+_{\ell},\qquad\sum_{\ell=0}^{\infty}(2\ell+1)C^+_{\ell}<\infty.
\]
We say that the random fields $(\mathcal{Q}\pm\mathrm{i}\mathcal{U})(\zeta)$ have spin $\pm 2$ in contrast to the fields $\Theta(\zeta)$ and $\mathcal{V}(\zeta)$ which have spin $0$.

Similar analysis for the representation $tV_2$ gives
\[
(\mathcal{Q}\pm\mathrm{i}\mathcal{U})(\zeta)=\sum_{\ell=2}^{\infty}\sum_{m=-\ell}^{\ell}a_{\ell m}^{(\pm 2)}\,{}_{\pm 2}Y_{\ell,m}(\zeta).
\]
This expansion was proved in \cite{Zaldarriaga:1996xe}. The spin $s$ spherical harmonics have the form
\[
\begin{aligned}
{}_sY_{\ell,m}(\zeta)&=\frac{(-1)^{\ell-m}}{\sqrt{4\pi}(\ell-s)!(\ell+s)!}\sqrt{(2\ell+1)(\ell+m)!(\ell-m)!}
(1+\zeta\zeta^*)\\
&\quad\times\sum_{p=\max\{0,m-s\}}^{\min\{\ell-s,\ell+s\}}\binom{\ell-s}{p}\binom{\ell+s}{p+s-m}\zeta^p
(\zeta^*)^{p+s-m}.
\end{aligned}
\]
It follows from this equation that ${}_{-s}Y_{\ell,-m}(\zeta)=(-1)^{m+s}{}_sY_{\ell,m}^*(\zeta)$. We see that
\begin{equation}\label{eq:52}
a^{(-2)*}_{\ell m}=a^{(2)}_{\ell\,-m}.
\end{equation}

Combining everything together, we obtain
\begin{equation}\label{eq:48}
\begin{aligned}
\tilde{\mathcal{P}}(\zeta)&=\frac{1}{2}\sum_{\ell=2}^{\infty}\sum_{m=-\ell}^{\ell}a^{\Theta}_{\ell m}\sigma^{(0)}Y_{\ell,m}(\zeta)+\frac{1}{2}\sum_{\ell=0}^{\infty}\sum_{m=-\ell}^{\ell}a^{\mathcal{V}}_{\ell m}\sigma^{(2)}Y_{\ell,m}(\zeta)\\
&\quad+\frac{1}{2}\sum_{\ell=2}^{\infty}\sum_{m=-\ell}^{\ell}\sigma^{(3)}\RE(a_{\ell m}^{(2)}\,{}_2Y_{\ell,m}(\zeta))+\frac{1}{2}\sum_{\ell=2}^{\infty}\sum_{m=-\ell}^{\ell}\sigma^{(1)}\IM(a_{\ell m}^{(2)}\,{}_2Y_{\ell,m}(\zeta)).
\end{aligned}
\end{equation}

\end{example}

\begin{example}[The $(\Theta,\mathcal{Q},\mathcal{U},\mathcal{V})$ model]
Introduce the \emph{circular polarisation}\index{polarisation basis!circular} or \emph{helicity basis}\index{helicity basis} $\bm{\varepsilon}^{(\pm)}=\frac{1}{\sqrt{2}}(\bm{\varepsilon}^{(1)}\pm\mathrm{i}\bm{\varepsilon}^{(2)})$ and the \emph{symmetric trace free tensor spherical harmonics}\index{spherical harmonics!symmetric trace free tensor}
\begin{equation}\label{eq:46}
\begin{aligned}
Y^E_{\ell m}(\zeta)&=-\frac{1}{\sqrt{2}}[{}_{-2}Y_{\ell,m}(\zeta)\bm{\varepsilon}^{(+)}
\otimes\bm{\varepsilon}^{(+)}+{}_2Y_{\ell,m}(\zeta)\bm{\varepsilon}^{(-)}\otimes\bm{\varepsilon}^{(-)}],\\
Y^B_{\ell,m}(\zeta)&=\frac{\mathrm{i}}{\sqrt{2}}[{}_{-2}Y_{\ell,m}(\zeta)
\bm{\varepsilon}^{(+)}\otimes\bm{\varepsilon}^{(+)}-{}_2Y_{\ell,m}(\zeta)
\bm{\varepsilon}^{(-)}\otimes\bm{\varepsilon}^{(-)}].
\end{aligned}
\end{equation}
Then we obtain
\[
\mathcal{P}(\zeta)=\frac{1}{2}\sum_{\ell=2}^{\infty}[e_{\ell m}Y^E_{\ell,m}(\zeta)+b_{\ell m}Y^B_{\ell,m}(\zeta)],
\]
where
\[
e_{\ell m}=-\frac{1}{\sqrt{2}}(a^{(2)}_{\ell m}+a^{(-2)}_{\ell m}),\qquad b_{\ell m}=\frac{\mathrm{i}}{\sqrt{2}}(a^{(2)}_{\ell m}-a^{(-2)}_{\ell m}).
\]
This expansion was proved in \cite{PhysRevD.55.7368}, but we need to prove that their definition of symmetric trace free tensor spherical harmonics is equivalent to our definition.

To perform this, we use the results of \cite{PhysRevD.86.125013}. In our notation, a part of their Equation~(92) takes the form
\[
\begin{aligned}
Y^E_{\ell,m}(\zeta)&=-\left(\frac{(\ell+1)(\ell+2)}{2(2\ell-1)(2\ell+1)}\right)^{\frac{1}{2}}
\mathsf{T}^{2\,\ell-2,\ell m}(\zeta)\\
&\quad-\left(\frac{3(\ell-1)(\ell+2)}{(2\ell-1)(2\ell+3)}\right)^{\frac{1}{2}}
\mathsf{T}^{2\,\ell,\ell m}(\zeta)\\
&\quad-\left(\frac{\ell(\ell-1}{2(2\ell+1)(2\ell+3)}\right)^{\frac{1}{2}}
\mathsf{T}^{2\,\ell+2,\ell m}(\zeta),\\
Y^B_{\ell,m}(\zeta)&=\mathrm{i}\left(\frac{\ell+2}{2\ell+1}\right)^{\frac{1}{2}}
\mathsf{T}^{2\,\ell-1,\ell m}(\zeta)+\mathrm{i}\left(\frac{\ell-2}{2\ell+1}\right)^{\frac{1}{2}}
\mathsf{T}^{2\,\ell+1,\ell m}(\zeta),
\end{aligned}
\]
where $\mathsf{T}^{2\,\ell',\ell m}(\zeta)$ are the \emph{pure-orbital tensor spherical harmonics},\index{spherical harmonics!pure-orbital tensor} see \cite{MR191580,MR569166}. Compare these equations with \cite[Equations~(2.30d), (2.30f)]{MR569166}. We obtain
\[
Y^E_{\ell,m}(\zeta)=-\mathsf{T}^{E2,\ell m}(\zeta),\qquad Y^B_{\ell,m}(\zeta)=-\mathsf{T}^{B2,\ell m}(\zeta),
\]
where $\mathsf{T}^{E2,\ell m}(\zeta)$ and $\mathsf{T}^{E2,\ell m}(\zeta)$ are the \emph{pure-spin tensor spherical harmonics},\index{spherical harmonics!pure-spin tensor} see \cite{MR569166,MR270692}. \cite[Equations~(2.38e), (2.38f)]{MR569166} become equivalent to our Equation~\eqref{eq:46}.

The complete expansion takes the form
\begin{equation}\label{eq:49}
\begin{aligned}
\tilde{\mathcal{P}}(\zeta)&=\frac{1}{2}\sum_{\ell=2}^{\infty}\sum_{m=-\ell}^{\ell}a^{\Theta}_{\ell m}\sigma^{(0)}Y_{\ell,m}(\zeta)+\frac{1}{2}\sum_{\ell=0}^{\infty}\sum_{m=-\ell}^{\ell}a^{\mathcal{V}}_{\ell m}\sigma^{(2)}Y_{\ell,m}(\zeta)\\
&\quad+\frac{1}{2}\sum_{\ell=2}^{\infty}[e_{\ell m}Y^E_{\ell,m}(\zeta)+b_{\ell m}Y^B_{\ell,m}(\zeta)].
\end{aligned}
\end{equation}

\end{example}

\begin{example}[The $(\Theta,\mathcal{E},\mathcal{B},\mathcal{V})$ model]
Is is well-known that
\[
\eth^2({}_{-2}Y_{\ell,m}(\zeta))=\overline{\eth}^2({}_2Y_{\ell,m}(\zeta))
=\sqrt{\frac{(\ell+2)!}{(\ell-2)!}}Y_{\ell,m}(\zeta).
\]
If we assume that
\[
\sum_{\ell=2}^{\infty}\frac{(2\ell+1)\mathsf{E}[|a^{(2)}_{\ell 0}|^2](\ell+2)!}{(\ell-2)!}<\infty,
\]
then we may differentiate the random fields $(\mathcal{Q}\pm\mathrm{i}\mathcal{U})(\zeta)$ term by term:
\[
\begin{aligned}
\overline{\eth}^2((\mathcal{Q}+\mathrm{i}\mathcal{U})(\zeta))&=\sum_{\ell=2}^{\infty}
\sum_{m=-\ell}^{\ell}a^{(2)}_{\ell m}\sqrt{\frac{(\ell+2)!}{(\ell-2)!}}Y_{\ell,m}(\zeta),\\
\eth^2((\mathcal{Q}-\mathrm{i}\mathcal{U})(\zeta))&=\sum_{\ell=2}^{\infty}
\sum_{m=-\ell}^{\ell}a^{(-2)}_{\ell m}\sqrt{\frac{(\ell+2)!}{(\ell-2)!}}Y_{\ell,m}(\zeta).
\end{aligned}
\]
Introduce the random fields
\[
\begin{aligned}
\mathcal{E}(\zeta)&=\frac{1}{2}[\overline{\eth}^2((\mathcal{Q}+\mathrm{i}\mathcal{U})(\zeta))
+\eth^2((\mathcal{Q}-\mathrm{i}\mathcal{U})(\zeta))],\\
\mathcal{B}(\zeta)&=-\frac{\mathrm{i}}{2}[\overline{\eth}^2((\mathcal{Q}+\mathrm{i}\mathcal{U})(\zeta))
-\eth^2((\mathcal{Q}-\mathrm{i}\mathcal{U})(\zeta))].
\end{aligned}
\]
Their spectral expansion takes the form
\[
\begin{aligned}
\mathcal{E}(\zeta)&=\sum_{\ell=2}^{\infty}\sum_{m=-\ell}^{\ell}a^{\mathcal{E}}_{\ell m}\sqrt{\frac{(\ell+2)!}{(\ell-2)!}}Y_{\ell,m}(\zeta),\\
\mathcal{B}(\zeta)&=\sum_{\ell=2}^{\infty}\sum_{m=-\ell}^{\ell}a^{\mathcal{B}}_{\ell m}\sqrt{\frac{(\ell+2)!}{(\ell-2)!}}Y_{\ell,m}(\zeta),
\end{aligned}
\]
where
\begin{equation}\label{eq:53}
a^{\mathcal{E}}_{\ell m}=\frac{1}{2}(a^{(2)}_{\ell m}+a^{(-2)}_{\ell m}),\qquad a^{\mathcal{B}}_{\ell m}=-\frac{\mathrm{i}}{2}(a^{(2)}_{\ell m}-a^{(-2)}_{\ell m}).
\end{equation}
All random fields $\Theta(\zeta)$, $\mathcal{E}(\zeta)$, $\mathcal{B}(\zeta)$, and $\mathcal{V}(\zeta)$ are spin $0$ fields. The complete expansion takes the form
\begin{equation}\label{eq:50}
\begin{aligned}
(\Theta,\mathcal{E},\mathcal{B},\mathcal{V})^{\top}(\zeta)
&=\sum_{\ell=2}^{\infty}\sum_{m=-\ell}^{\ell}a^{\Theta}_{\ell m}Y_{\ell,m}(\zeta)\mathbf{e}_{\Theta}+\sum_{\ell=2}^{\infty}\sum_{m=-\ell}^{\ell}a^{\mathcal{E}}_{\ell m}Y_{\ell,m}(\zeta)\mathbf{e}_{\mathcal{E}}\\
&\quad+\sum_{\ell=2}^{\infty}\sum_{m=-\ell}^{\ell}a^{\mathcal{B}}_{\ell m}Y_{\ell,m}(\zeta)\mathbf{e}_{\mathcal{B}}+\sum_{\ell=0}^{\infty}\sum_{m=-\ell}^{\ell}a^{\mathcal{V}}_{\ell m}Y_{\ell,m}(\zeta)\mathbf{e}_{\mathcal{V}},
\end{aligned}
\end{equation}
where $\mathbf{e}_{\Theta}=(1,0,0,0)^{\top}$, and so on.
\end{example}

\begin{example}
Regard the celestial sphere as the base space of the principal bundle $(\mathrm{SO}(3),\mathrm{SO}(2),S^2,\tilde{\pi})$ with $\tilde{\pi}(g)=gH$. Let the group $\mathrm{SO}(2)$ acts in the space $L_0$ given by Equation~\eqref{eq:47} by $h\cdot\tilde{\mathcal{P}}=h\mathcal{P}h^{-1}$. Let $(\theta,\varphi)$ be the spherical coordinates on $S^2$. The polarisation tensor of the CMB can be described as a random cross-section of a vector bundle associated to the principal bundle $(\mathrm{SO}(3),\mathrm{SO}(2),S^2,\tilde{\pi})$ by the above representation. The spectral expansions of these random sections have the form of Equations~\eqref{eq:48}, \eqref{eq:49}, and \eqref{eq:50}, where $\zeta$ is replaced with the spherical coordinates $(\theta,\varphi)$ on the sphere $S^2$.

However, in this model the electric field of the CMB is \emph{not} a spinor field, because $\mathrm{SO}(2)\neq\mathrm{Spin}(2)$. The group $\mathrm{Spin}(2)=\mathrm{U}(2)$ covers the group $\mathrm{SO}(2)$ by the ``square'' map \eqref{eq:51}. Nevertheless, the term ``spin 0 field'' is traditionally used for $\Theta(\theta,\varphi)$ and similar random fields, while the term ``spin 2 field'' is used for $\mathcal{Q}(\theta,\varphi)$ and similar random fields.

The transformation $\mathbf{n}\mapsto -\mathbf{n}$ is called the \emph{parity transformation}.\index{parity transformation} Under parity, the spherical harmonics $Y_{\ell,m}$ transform according to \cite[Equation~(A.4.39)]{durrer_2020}: $Y_{\ell,m}(-\mathbf{n})=(-1)^{\ell}Y_{\ell,m}(\mathbf{n})$. The naming convention for harmonic types is called the \emph{even-odd convention},\index{even-odd convention} see \cite[Table~II]{MR321491}. We say that the spherical harmonics $Y_{\ell,m}$ have ``electric-type'' parity.\index{parity!electric-type} It follows that the random fields $\Theta(\theta,\varphi)$ and $\mathcal{V}(\theta,\varphi)$ have electric type. Using Equations~\eqref{eq:52} and \eqref{eq:53}, we see that $a^{\mathcal{E}}_{\ell m}\mapsto(-1)^{\ell}a^{\mathcal{E}}_{\ell\,-m}$ (electric type), but $a^{\mathcal{B}}_{\ell m}\mapsto(-1)^{\ell+1}a^{\mathcal{B}}_{\ell\,-m}$. This is called ``magnetic-type'' parity.\index{parity!magnetic-type}

In Equation~\eqref{eq:50}, consider the matrix of cross-correlations $C^{ij}_{\ell}=\mathsf{E}[a^{(i)*}_{\ell m}a^{(j)}_{\ell m}]$ with $i$, $j\in\{\Theta,\mathcal{E},\mathcal{B},\mathcal{V}\}$. The current standard cosmological model assumes that the random field that generates the initial fluctuations, is invariant under parity. In this case, the cross-correlations between the components with different parity vanish. That is, $C^{\Theta\mathcal{B}}_{\ell}=C^{\mathcal{E}\mathcal{B}}_{\ell}=C^{\mathcal{B}\mathcal{V}}_{\ell}=0$.

Again, according to the standard cosmological model, in the early universe, photons and baryons are coupled by the physical process called \emph{Thomson scattering}.\index{Thomson scattering} This process does introduce linear polarisation, but no circular one. Only a small amount of circular polarisation may be introduced by other physical processes, see \cite{MR3985285}. By this reason, we may put $a^{\mathcal{V}}_{\ell m}=0$. The expansion takes the form
\[
\begin{aligned}
(\Theta,\mathcal{E},\mathcal{B})^{\top}(\theta,\varphi)
&=\sum_{\ell=2}^{\infty}\sum_{m=-\ell}^{\ell}a^{\Theta}_{\ell m}Y_{\ell,m}(\theta,\varphi)\mathbf{e}_{\Theta}
+\sum_{\ell=2}^{\infty}\sum_{m=-\ell}^{\ell}a^{\mathcal{E}}_{\ell m}Y_{\ell,m}(\theta,\varphi)\mathbf{e}_{\mathcal{E}}\\
&\quad+\sum_{\ell=2}^{\infty}\sum_{m=-\ell}^{\ell}a^{\mathcal{B}}_{\ell m}Y_{\ell,m}(\theta,\varphi)\mathbf{e}_{\mathcal{B}}
\end{aligned}
\]
with
\[
\mathsf{E}[(a^{\Theta*}_{\ell m},a^{\mathcal{E}*}_{\ell m},a^{\mathcal{B}*}_{\ell m})^{\top}(a^{\Theta}_{\ell' m'},a^{\mathcal{E}}_{\ell'm'},a^{\mathcal{B}}_{\ell'm'})]
=\delta_{\ell\ell'}\delta_{mm'}\left(
\begin{smallmatrix}
  C^{\Theta\Theta}_{\ell} & C^{\Theta\mathcal{E}}_{\ell} & 0 \\
  C^{\mathcal{E}\Theta}_{\ell} & C^{\mathcal{E}\mathcal{E}}_{\ell} & 0 \\
  0 & 0 & C^{\mathcal{B}\mathcal{B}}_{\ell}
\end{smallmatrix}
\right).
\]

Finally, the standard cosmological model predicts arising primordial $\mathcal{B}$-po\-la\-ri\-sa\-tion due to a background of gravitational waves created either just after the Big Bang during a physical process called \emph{inflation}\index{inflation} or between the end of inflation and Big Bang nucleosynthesis. Currently, this mode is not detected, the best current upper estimates for its amplitude may be found in \cite{MR4233941}.
\end{example}

\begin{example}\label{ex:CMBreal}

Recall that the spinor representation $V_1\oplus tV_1$ of the group $H=\mathrm{U}(1)$ has a real structure $j(z_1,z_2)=(z^*_2,z^*_1)$, see Example~\ref{ex:spinweight}. Consider the corresponding real spinor representation $rV_1$: $\mathrm{e}^{\mathrm{i}\theta}\mapsto\left(
\begin{smallmatrix}
  \cos\theta & -\sin\theta \\
  \sin\theta & \cos\theta
\end{smallmatrix}
\right)$. Its tensor square has the form $rV_1\otimes rV_1=2U_0\oplus rV_2$, which easily follows from the corresponding relation for characters: $4\cos^2\theta=2+2\cos(2\theta)$.

Introduce the matrices $\sigma^{(i)}_{\mathbb{R}}$ by $\sigma^{(i)}_{\mathbb{R}}=\sigma^{(i)}$ for $i=0$, $1$, $3$, and $\sigma^{(2)}_{\mathbb{R}}=-\mathrm{i}\sigma^{(2)}$. The first (resp. second) copy of $U_0$ acts in the one-dimensional real space generated by  $\sigma^{(i)}_{\mathbb{R}}$ (resp.  $\sigma^{(2)}_{\mathbb{R}}$). The representation $rV_2$ acts in the two-dimensional space generated by  $\sigma^{(1)}_{\mathbb{R}}$ and  $\sigma^{(3)}_{\mathbb{R}}$.

On the other hand, the real irreducible representations of the group $G=\mathrm{SU}(2)$ are $\{\,U_{\ell}\colon\ell=0,1,\dots\,\}$\index{representation!irreducible real of $\mathrm{SU}(2)$!$U_{\ell}$} and $\{\,rc'W_{\ell}\colon\ell=\frac{1}{2},\frac{3}{2},\dots\,\}$.\index{representation!irreducible real of $\mathrm{SU}(2)$!$rc'W_{\ell}$} The character of $U_{\ell}$ is equal to that of $cU_{\ell}$, while the character of $rV_m$, $m>0$, is
\[
\chi_{rV_m}\left(\left(
\begin{smallmatrix}
  \alpha & 0 \\
  0 & \alpha^*
\end{smallmatrix}
\right)\right)=\chi_{crV_m}\left(\left(
\begin{smallmatrix}
  \alpha & 0 \\
  0 & \alpha^*
\end{smallmatrix}
\right)\right)=\chi_{V_m\oplus tV_m}\left(\left(
\begin{smallmatrix}
  \alpha & 0 \\
  0 & \alpha^*
\end{smallmatrix}
\right)\right)=\alpha^{2m}+\alpha^{-2m},
\]
and $\chi_{U_0}\left(\left(
\begin{smallmatrix}
  \alpha & 0 \\
  0 & \alpha^*
\end{smallmatrix}
\right)\right)=1$, where we used the relation $cr=1+t$ (\cite[Proposition~3.6]{MR0252560}). It follows that
\[
\res^G_HU_{\ell}=U_0\oplus rV_1\oplus\cdots\oplus rV_{\ell}.
\]
For the Hilbert space $L^2(E_0)$ of square-integrable section of the line bundle, where the real representation of $\mathrm{SU}(2)$ induced by the representation $U_0$ of the subgroup $\mathrm{U}(1)$ acts, the Frobenius reciprocity gives that $L^2(E_0)$ contains
\[
\frac{\dim_{\mathbb{R}}\Hom_{\mathbb{R}H}(\res^G_HU_{\ell},U_0)}
{\dim_{\mathbb{R}}\Hom_{\mathbb{R}G}(U_{\ell},U_{\ell})}=\frac{1}{1}=1
\]
copy of the representation $U_{\ell}$ for $\ell\geq 0$. The Hilbert space $L^2(E_2)$ contains
\[
\frac{\dim_{\mathbb{R}}\Hom_{\mathbb{R}H}(\res^G_HU_{\ell},rV_2)}
{\dim_{\mathbb{R}}\Hom_{\mathbb{R}G}(U_{\ell},U_{\ell})}=\frac{2}{1}=2
\]
copies of the representation $U_{\ell}$ for $\ell\geq 2$.

Similarly, the character of $rc'W_{\ell}$ is $\chi_{rc'W_{\ell}}=\chi_{crc'W_{\ell}}=2\chi_{c'W_{\ell}}$, and we obtain
\[
\res^G_Hrc'W_{\ell}=2rV_1\oplus 2rV_3\oplus\cdots\oplus 2rV_{2\ell}.
\]
Neither $L^2(E_0)$ nor $L^2(E_2)$ contain copies of the representations $rc'W_{\ell}$.

Instead of constructing the bases of the above real Hilbert spaces from scratch, we use the existing bases of the corresponding complex spaces. Indeed, the basis in $cL^2(E_0)$ is given by the complex-valued spherical harmonics $Y_{\ell,m}(\zeta)$. The complex conjugation is a real structure in $cL^2(E_0)$. Therefore, the real-valued spherical harmonics\index{spherical harmonics!real-valued}
\[
Y^m_{\ell}(\zeta)=
\begin{cases}
  \frac{1}{\sqrt{2}}[Y_{\ell,m}(\zeta)+(-1)^mY_{\ell,-m}(\zeta)], & \mbox{if } m>0, \\
  Y_{\ell,0}(\zeta), & \mbox{if } m=0, \\
  \frac{1}{\sqrt{2}\mathrm{i}}[Y_{\ell,m}(\zeta)-(-1)^mY_{\ell,-m}(\zeta)], & \mbox{if } m<0
\end{cases}
\]
form the basis in $L^2(E_0)$. Likewise, the basis in $cL^2(E_2)\oplus cL^2(E_{-2})$ is given by the $\mathbb{C}^2$-valued \emph{spin 2 real-valued spherical harmonics}\index{spherical harmonics!spin 2!real-valued}
\[
\left(
\begin{smallmatrix}
  {}_2Y_{\ell,m}(\zeta) \\
  0
\end{smallmatrix}
\right),\quad\left(
\begin{smallmatrix}
  0 \\
  {}_{-2}Y_{\ell,m}(\zeta)
\end{smallmatrix}
\right),\qquad \ell=2, 3,\dots,\quad -\ell\leq m\leq\ell.
\]
The real structure of Example~\ref{ex:spinweight} maps $\left(
\begin{smallmatrix}
  {}_2Y_{\ell,m}(\zeta) \\
  0
\end{smallmatrix}
\right)$ to $\left(
\begin{smallmatrix}
  0 \\
  (-1)^m{}_{-2}Y_{\ell,-m}(\zeta)
\end{smallmatrix}
\right)$ and $\left(
\begin{smallmatrix}
  0 \\
  {}_{-2}Y_{\ell,m}(\zeta)
\end{smallmatrix}
\right)$ to $\left(
\begin{smallmatrix}
  (-1)^m{}_2Y_{\ell,m}(\zeta) \\
  0
\end{smallmatrix}
\right)$. Therefore, the $\mathbb{R}^2$-valued spin-weighted spherical harmonics take the form
\[
\begin{aligned}
{}_2\mathbf{Y}^m_{\ell}(\zeta)&=\left(
\begin{smallmatrix}
  {}_2Y_{\ell,m}(\zeta)+(-1)^m{}_{-2}Y_{\ell,-m}(\zeta) \\
  -\mathrm{i}\,{}_2Y_{\ell,m}(\zeta)+\mathrm{i}(-1)^m\,{}_{-2}Y_{\ell,-m}(\zeta)
\end{smallmatrix}
\right),\\
{}_{-2}\mathbf{Y}^m_{\ell}(\zeta)&=\left(
\begin{smallmatrix}
  \mathrm{i}\,{}_2Y_{\ell,m}(\zeta)-\mathrm{i}(-1)^m\,{}_{-2}Y_{\ell,-m}(\zeta) \\
  {}_2Y_{\ell,m}(\zeta)+(-1)^m{}_{-2}Y_{\ell,-m}(\zeta)
\end{smallmatrix}
\right).
\end{aligned}
\]

It is easy to check that the real-valued spherical harmonics have ``electric-type'' parity. We obtain
\[
\begin{aligned}
\Theta(\zeta)&=\sum_{\ell=2}^{\infty}\sum_{m=-\ell}^{\ell}a^{\Theta}_{\ell m}Y^m_{\ell}(\zeta),\\
\left(\begin{smallmatrix}
        \mathcal{Q} \\
        \mathcal{U}
      \end{smallmatrix}
\right)(\zeta)&=\sum_{\ell=2}^{\infty}\sum_{m=-\ell}^{\ell}\left({}_2a^{\mathcal{Q},\mathcal{U}}_{\ell m}{}_2\mathbf{Y}^m_{\ell}(\zeta)+{}_{-2}a^{\mathcal{Q},\mathcal{U}}_{\ell m}{}_{-2}\mathbf{Y}^m_{\ell}(\zeta)\right),\\
\mathcal{V}(\zeta)&=\sum_{\ell=2}^{\infty}\sum_{m=-\ell}^{\ell}a^{\mathcal{V}}_{\ell m}Y^m_{\ell}(\zeta),
\end{aligned}
\]
where the coefficients are centred random variables with cross-correlations given by
\[
\mathsf{E}[(a^{\Theta}_{\ell m},{}_2a^{\mathcal{Q},\mathcal{U}}_{\ell m},{}_{-2}a^{\mathcal{Q},\mathcal{U}}_{\ell m},a^{\mathcal{V}}_{\ell m})(a^{\Theta}_{\ell'm'},{}_2a^{\mathcal{Q},\mathcal{U}}_{\ell' m'},{}_{-2}a^{\mathcal{Q},\mathcal{U}}_{\ell'm'},a^{\mathcal{V}}_{\ell' m'})^{\top}]=\delta_{\ell\ell'}\delta_{mm'}C_{\ell},
\]
and the symmetric positive-definite matrices $C_{\ell}$ satisfy
\[
\sum_{\ell=2}^{\infty}(2\ell+1)\tr C_{\ell}<\infty.
\]
\end{example}

\begin{example}
Observe that the spinor representation $V_1\oplus tV_1$ of the group $H=\mathrm{U}(1)$ has a quaternionic structure $j(z_1,z_2)=(-z^*_2,z^*_1)$. The corresponding quaternionic irreducible representation $qV_1$ of the group $\mathrm{U}(1)$ acts in $\mathbb{H}^1$ by $\mathrm{e}^{\mathrm{i}\theta}\mapsto\mathrm{e}^{\mathrm{i}\theta}$. By \cite[Definition~3.31]{MR0252560}, its character is the character of $c'qV_1=V_1\oplus tV_1$. Again by \cite{MR0252560}, the tensor square of $c'qV_1$ is a real representation, and its character is the same as that of the representation $2U_0\oplus rV_2$ of Example~\ref{ex:CMBreal}. No new expansions appears here.
\end{example}

\section{Conclusions}

The theory of random sections of various fibre bundles is now a mature part of Probability with various interesting links to Special functions, Invariant Theory, Group Representations, Convex Analysis, etc. The case of trivial tensor bundles over Euclidean spaces found interesting and important applications in Continuum Physics, while that of nontrivial spinor bundles over the two-dimensional sphere $S^2$ in Cosmology.

This survey reflects the scientific interests of the authors and is concentrated on spectral expansions of random functions in question and their applications to Physics. On the one hand, many research questions are missing; for example, tensor-valued random fields as solutions of boundary value problems for stochastic partial differential equations similar to Equation~\eqref{stoch-eq}, statistics of random sections of spinor bundles, properties of sample paths of random sections, and many others. On the other hand, we formulated several unsolved research questions above.

The authors hope that this survey can be useful both for mathematicians who are working in different physical applications and for applied specialists who need to use advanced mathematical tools.

\section{Acknowledgements}

We are grateful to Victor Abramov for useful discussion on principal bundles.

\appendix

\section{The mathematical language of the theory}\label{sec:language}

We give a short description of mathematical terms used in the paper. For a more detailed account, see \cite[Appendix~A]{MR2977490}, \cite[Chapter~2]{MR3930601}, and the literature cited there.

\subsection{Tensors}

The symbol $\mathbb{R}$ (resp. $\mathbb{C}$, resp. $\mathbb{H}$) denotes the field of real numbers (resp. the field of complex numbers, resp. the skew field of quaternions). The symbol $U$ (resp. $V$, resp. $W$) denotes a finite-dimensional linear space over $\mathbb{R}$ (resp. over $\mathbb{C}$, resp. a right finite-dimensional linear space over $\mathbb{H}$). The round brackets $(\cdot,\cdot)$ denote an inner product in either $U$, $V$, or $W$. It is $\mathbb{K}$-linear in the second argument, has the property $\langle \mathbf{x}_1,\mathbf{x}_2\rangle=\langle \mathbf{x}_2,\mathbf{x}_1\rangle^*$ for all $\mathbf{x}_1$, $\mathbf{x}_2\in U$ (resp. in $V$, resp. in $W$), and is positive-definite: $\langle\mathbf{x},\mathbf{x}\rangle\geq 0$ with equality only if $\mathbf{x}=\mathbf{0}$, where a $\mathbb{R}$-linear map $\mathbb{K}\to\mathbb{K}$, $x\mapsto x^*$ is given by $a^*=a$ for $a\in\mathbb{R}$, $(a+b\mathrm{i})^*=a-b\mathrm{i}$ for $(a+b\mathrm{i})\in\mathbb{C}$, and
\[
(a+b\mathrm{i}+c\mathrm{j}+d\mathrm{k})=a-b\mathrm{i}-c\mathrm{j}-d\mathrm{k}
\]
for $(a+b\mathrm{i}+c\mathrm{j}+d\mathrm{k})\in\mathbb{H}$.
By choosing a basis $\{\,\mathbf{e}^i\colon 1\leq i\leq d\,\}$ which is orthonormal with respect to the introduced inner product, the linear space $U$ (resp. $V$, resp. $W$) can be identified with the coordinate space $\mathbb{R}^d$ (resp. $\mathbb{C}^d$, resp. $\mathbb{H}^d$).

The symbol $U^*$ denotes the set of all linear forms on $U$. A real \emph{rank $r$ tensor}\index{tensor} is an element of a tensor product of $r\geq 0$ finite-dimensional real linear spaces. By definition, the \emph{tensor product}\index{tensor product!of real linear spaces} of an empty family of real linear spaces ($r=0$) is the coordinate space $\mathbb{R}^1$. The tensor product of the family $\{U_1,\dots,U_r\}$ with $r\geq 1$ is the set of all $r$-linear forms defined on the Cartesian product $U^*_1\times\cdots\times U^*_r$. The inner product on each space $U_i$, $1\leq i\leq r$, enables to identify $U_i$ with $U^*_i$ with the help of the map $U_i\to U^*_i$, $\mathbf{x}\mapsto(\mathbf{x},\cdot)$ and do not distinguish between covariant and contravariant tensors. The tensor product\index{tensor product!of vectors} of the vectors $\mathbf{x}_1\in U_1$, \dots, $\mathbf{x}_r\in U_r$ is the $r$-linear form
\[
\mathbf{x}_1\otimes\cdots\otimes\mathbf{x}_r(\mathbf{y}_1,\dots,\mathbf{y}_r)
=(\mathbf{y}_1,\mathbf{x}_1)\cdots(\mathbf{y}_r,\mathbf{x}_r).
\]

Similarly, the inner product on each \emph{complex} linear space $V_i$, $1\leq i\leq r$, enables to identify $V_i$ with $V^*_i$. Moreover, let the space $tV$ has the same underlying set as $V$, but let $z\in\mathbb{C}$ acts on $tV$ as $z^*$ used to act on $V$. We can identify $V^*$ and $tV$ with the help of the map $tV\to V^*$, $\mathbf{x}\mapsto(\mathbf{x},\cdot)$ and do not distinguish between four types of complex tensors, described, e.g., in \cite[Section~13.1]{MR757180}.

The inner product on each \emph{quaternionic} linear space $W_i$, $1\leq i\leq r$, enables to identify $W_i$ with $W^*_i$. However, assume that $W_1$ and $W_2$ are two right quaternionic linear spaces, and $q_1$, $q_2$ are two quaternions. An $\mathbb{H}$-bilinear form $\mu\colon W_1\times W_2\to\mathbb{H}^1$ must satisfy $\mu(\mathbf{x}_1q_1,\mathbf{x}_2q_2)=\mu(\mathbf{x}_1,\mathbf{x}_2q_2)q_1=\mu(\mathbf{x}_1,\mathbf{x}_2)q_1q_2$ and $\mu(\mathbf{x}_1q_1,\mathbf{x}_2q_2)=\mu(\mathbf{x}_1q_1,\mathbf{x}_2)q_2=\mu(\mathbf{x}_1,\mathbf{x}_2)q_2q_1$. In general, $q_1q_2\neq q_2q_1$, and this does not work. On the other hand, a bilinear form is well-defined if $W_2$ is a \emph{left} quaternionic space. The set of such forms is \emph{not} a right quaternionic linear space. Indeed, the form $(\mathbf{x}_1,\mathbf{x}_2)\mapsto\mu(\mathbf{x}_1,\mathbf{x}_2)q$ is not bilinear: $\mu(\mathbf{x}_1q_1,\mathbf{x}_2)q$ must be equal to $\mu(\mathbf{x}_1,\mathbf{x}_2)qq_1$, but it is equal to $\mu(\mathbf{x}_1,\mathbf{x}_2)q_1q$. By similar reasons, it is not a left quaternionic linear space. However, the set $W_1\otimes W_2$ is a real linear space.

By this reasons, we put $\mathbb{R}'=\mathbb{H}'=\mathbb{R}$ and $\mathbb{C}'=\mathbb{C}$. In this notation, the tensor product of two $\mathbb{K}$-linear spaces is a $\mathbb{K}'$-linear space.

\begin{example}[The Levi--Civita permutation tensor in three dimensions]\label{ex:Levi-Civita}
Let $U_1$, $U_2$, and $U_3$ be three copies of a real $3$-dimensional linear space with inner product. define the \emph{Levi--Civita permutation tensor}\index{tensor!Levi--Civita} in a coordinate-free form by
\[
\varepsilon(\mathbf{x}_1,\mathbf{x}_2,\mathbf{x}_3)=(\mathbf{x}_1,\mathbf{x}_2\times\mathbf{x}_3),
\]
that is, the scalar triple product of vectors in a right-handed coordinate system.

In coordinates, let $\{\mathbf{e}^i\colon 1\leq i\leq 3\}$ be the standard basis of the space $\mathbb{R}^3$. Denote
\[
\varepsilon_{ijk}=\varepsilon(\mathbf{e}^i,\mathbf{e}^j,\mathbf{e}^k).
\]
It is easy to check that the introduced symbol is \emph{totally antisymmetric}. That is, when any two indices are interchanged, the symbol is negated, if any two indices are equal, the symbol is zero, and when all indices are unequal, the symbol $\varepsilon_{ijk}$ is equal to the signature of the permutation $(i,j,k)$.
\end{example}

\subsection{Group representations}

A \emph{left} action\index{group action!left} of a group $G$ with identity element $e$ on a non-empty set $L$ is such a map $G\times L \to L$, $(g,x)\mapsto g\cdot x$, that $e\cdot x=x$ and $g\cdot(h\cdot x)=(gh)\cdot x$ for all $g$, $h\in G$ and for all $x\in L$. Similarly, a \emph{right} action\index{group action!right} of a group $G$ with identity element $e$ on a non-empty set $L$ is such a map $L\times G \to L$, $(x,g)\mapsto x\cdot g$, that $x\cdot e=x$ and $(x\cdot h)\cdot g=x\cdot(hg)$ for all $g$, $h\in G$ and for all $x\in L$.

As an example of a left action, think about the matrix-vector multiplication. A more interesting example is as follows. Consider a real finite-dimensional space $U$ as a group with respect to the addition. Assume that the group $U$ acts on a set $E$ \emph{transitively},\index{group action!transitive} that is, for each pair $A$, $B\in E$ there exists a $x\in U$ such that $x\cdot A=B$, and \emph{faithfully},\index{group action!faithful} that is, for any $x\in U\setminus\{0\}$ there is an $A\in E$ with $x\cdot A\neq A$. We change notation and denote the result of the action of a vector $x\in U$ on a point $A\in E$ by $A+x$. Denote by $B-A$ such a vector in $U$ that $A+(B-A)=B$. The inner product on $V$ enables to define the distance in $E$:
\[
\rho(A,B)=\sqrt{(A-B,A-B)}.
\]
An affine space with this distance is called a \emph{Euclidean space}.\index{Euclidean space} If the dimension of $U$ is equal to $d$, then we denote the pair $(E,\rho)$ by $E^d$. According to \cite{MR1345386}:
\begin{quote}
Formerly, the universe was provided not with an affine, but with a linear structure (the geocentric system of the universe).
\end{quote}

We omit the dot in the notation $g\cdot x$ in the case when $G$ is a topological group, $L$ is a Hilbert space over $\mathbb{K}\in\{\mathbb{R},\mathbb{C},\mathbb{H}\}$, $g\cdot x$ is a $\mathbb{K}$-linear function of $x$ and a continuous function of $g$ and $x$ (in the strong operator topology, when $X$ is infinite-dimensional). Such an action is called a \emph{representation}\index{representation} of $G$. Recall that a topological group is a set $G$ which is a group and a topological space such that the map
\begin{equation}\label{eq:1}
G\times G\to G,\qquad (g,h)\mapsto g^{-1}h
\end{equation}
is continuous.

The \emph{translation}\index{translation} $\theta^L(g)\colon x\to gx$ is a bounded linear operator in $L$. For two Hilbert spaces $L_1$ and $L_2$ over the same (skew) field $\mathbb{K}$, a bounded linear operator $f\colon X_1\to X_2$ is called \emph{intertwining}\index{intertwining operator} if $f(gx)=g(fx)$, $g\in G$, $x\in L_1$. Two representations are equivalent if the $\mathbb{K}'$-linear space $\Hom_G(L_1,L_2)$ of intertwining operators contains an invertible operator with bounded inverse.

If $G$ is a \emph{compact} topological group, that we can and will give $L$ an inner product which is invariant under $G$. We then speak of an \emph{orthogonal}\index{representation!orthogonal} (resp. \emph{unitary},\index{representation!unitary} resp. \emph{symplectic})\index{representation!symplectic} representation in the case of $\mathbb{K}=\mathbb{R}$ (resp. $\mathbb{K}=\mathbb{C}$, resp. $\mathbb{K}=\mathbb{H}$). In what follows the symbol $G$ always denotes a compact group.

A non-zero linear space $L$ is \emph{reducible}\index{representation!reducible} if some proper non-zero closed subspace $M$ of $X$ is invariant: $gy\in M$ for all $y\in M$, otherwise $L$ is \emph{irreducible}.\index{representation!irreducible} For a compact group, every representation $L$ is a Hilbert direct sum of irreducible components. Moreover, let $L_i$ runs over the inequivalent irreducible representations, as $i$ runs over some set $I$. For a finite-dimensional space $L$, there exists a unique set $\{\,m_i\colon i\in I\,\}$ of nonnegative integers such that all of them except finitely many are zeroes and the space $L$ is uniquely decomposed into the Hilbert direct sum of \emph{isotypical subspaces}\index{isotypic subspace} in which the direct sum of $m_i>0$ copies of $L_i$ acts.

Let $U_1$ and $U_2$ be two real finite-dimensional representations of a topological group $G$ with the translations $\theta^{U_1}$ and $\theta^{U_2}$. A map $f\colon U_2\to U_1$ is called \emph{form-invariant}\index{form-invariant map} for the pair $(U_1,U_2)$ if
\[
f(\theta^{U_2}(g)x)=\theta^{U_1}(g)f(x),\qquad g\in G,\quad x\in U_2.
\]
If $U_1=\mathbb{R}^1$ and $\theta^{U_1}(g)=1$, then the function $f\colon X_2\to\mathbb{R}^1$ is called just invariant of the representation $U_2$, hence the name \emph{Invariant Theory}.

We formulate an important result of Invariant Theory. Let $U$, $U_1$, \dots, $U_N$ be finitely many finite-dimensional orthogonal representations of a closed subgroup $G$ of the group $\mathrm{O}(3)$ of $3\times 3$ orthogonal matrices. It turns out that there exists finitely many, say $K$, polynomials $\{\,I_k(x_1,\dots,x_N)\colon 1\leq k\leq K\,\}$ with $x_n\in U_n$ for $1\leq n\leq N$ that constitute an \emph{integrity basis}\index{integrity basis} for polynomial invariants of the representation $U_1\oplus\cdots\oplus U_N$. That is, every polynomial invariant of the above representation is a polynomial in $I_1$, \dots, $I_K$. Similarly, there exists an integrity basis $\{\,T_l\colon 1\leq l\leq L\,\}$ for form-invariant polynomials of the pair $(U,U_1\oplus\cdots\oplus U_N)$. Following \cite{MR171421}, call $T_l$ the \emph{basis form-invariant tensors}.\index{tensor!basis form-invariant}

\begin{theorem}[Wineman--Pipkin, \cite{MR171421}]\label{th:WP}
A function $T\colon U_1\oplus\cdots\oplus U_N\to U$ is a measurable form-invariant map of the pair $(U,U_1\oplus\cdots\oplus U_N)$ if and only if it has the form
\[
T(x_1,\dots,x_N)=\sum_{l=1}^{L}\varphi_l(I_1,\dots,I_K)T_l(x_1,\dots,x_N),
\]
where $\varphi_l$ are real-valued measurable functions.
\end{theorem}

\subsection{The Adams construction}

In this subsection, we introduce a convenient notation system for irreducible representations of compact topological groups due to Adams \cite{MR0252560}.

Let $V$ be a finite-dimensional complex representation of a compact topological group $G$. Let the representation $tV$ has the same underlying set as $V$ and the same action of $G$, but let $z\in\mathbb{C}$ acts on $tV$ as $z^*$ used to act on $V$. For example, let $G=\mathrm{U}(1)$, the set of all complex numbers $z$ with $|z|=1$, and let an element $g=\mathrm{e}^{\mathrm{i}\varphi}\in\mathrm{U}(1)$ acts on $V_n=\mathbb{C}^1$\index{representation!irreducible complex of $\mathrm{U}(1)$!$V_n$} by $g\cdot z=\mathrm{e}^{\mathrm{i}n\varphi}z$ for $n\geq 1$ and $z\in\mathbb{C}^1$. We have $tV_n=\mathbb{C}^1$\index{representation!irreducible complex of $\mathrm{U}(1)$!$tV_n$} with the action $g\cdot z=\mathrm{e}^{-\mathrm{i}n\varphi}z$, and $tV_n$ is not equivalent to $V_n$. An irreducible complex representation with this property is called a \emph{complex representation of complex type}.\index{representation!complex!of complex type}

Let the representation $rV_n$\index{representation!irreducible real of $\mathrm{U}(1)$!$rV_n$} has the same underlying set as $V_n$ and the same action of $G$, but regard it as a $\mathbb{R}$-linear space. The matrix of the representation $rV_n$ is $\left(
\begin{smallmatrix}
  \cos(n\varphi) & -\sin(n\varphi) \\
  \sin(n\varphi) & \cos(n\varphi)
\end{smallmatrix}
\right)$. An irreducible real representation of the form $rV$, where $V$ is a complex representation of complex type, is called a \emph{real representation of complex type}.\index{representation!real!of complex type} The representation $rtV$ is equivalent to $rV$ by \cite[Proposition~3.6]{MR0252560}. In particular, when $n=1$, we conclude, that the group $\mathrm{U}(1)$ is isomorphic to the group $\mathrm{SO}(2)$ of orthogonal $2\times 2$ matrices with unit determinant.

Define $qV_n=V_n\otimes_{\mathbb{C}}\mathbb{H}$ and regard it as a right quaternionic linear space with the vector-scalar multiplication $(\mathbf{v}\otimes_{\mathbb{C}}q')q=\mathbf{v}\otimes_{\mathbb{C}}q'q$ and a quaternionic representation of $G$ with the action $g\cdot(\mathbf{v}\otimes_{\mathbb{C}}q)=(g\cdot\mathbf{v})\otimes_{\mathbb{C}}q$. The quaternionic representation $qV_n$\index{representation!irreducible quaternionic of $\mathrm{U}(1)$!$qV_n$} is irreducible and is called a \emph{quaternionic representation of complex type}.\index{representation!quaternionic!of complex type}

Assume that an irreducible complex representation $V$ is not of complex type. This may happen by one of two mutually exclusive reasons.
\begin{enumerate}
  \item There exists a conjugate-linear map $j\colon V\to V$ that commutes with the representation and satisfies the condition $j^2=\id$.
  \item There exists a conjugate-linear map $j\colon V\to V$ that commutes with the representation and satisfies the condition $j^2=-\id$.
\end{enumerate}

As an example, consider the trivial representation $V_0=\mathbb{C}^1$ of the group $\mathrm{U}(1)$ with the action $g\cdot z=z$. This representation clearly commutes with the map $j(z)=z^*$. For such a representation $V$, there always exists a real irreducible representation $U$ such that $cU=V$. Here $cU=U\otimes_{\mathbb{R}}\mathbb{C}$ with the vector-scalar multiplication $(\mathbf{u}\otimes_{\mathbb{R}}z')z=\mathbf{u}\otimes_{\mathbb{R}}z'z$ and the action of $G$ by $g\cdot(\mathbf{u}\otimes_{\mathbb{R}}z)=(g\cdot\mathbf{u})\otimes_{\mathbb{R}}z$. In this case the representation $V$ (resp. $U$) is called a \emph{complex}\index{representation!complex!of real type} (\emph{resp. real})\index{representation!real!of real type} \emph{representation of real type}, while the irreducible representation $W=qcU$ is called a \emph{quaternionic representation of real type}.\index{representation!quaternionic!of real type}

As an another example, consider the group $G=\mathrm{SU}(2)$ of unitary $2\times 2$ matrices with unit determinant and elements $g=\left(
\begin{smallmatrix}
  \alpha & \beta \\
  -\beta^* & \alpha^*
\end{smallmatrix}
\right)$. Put $V_{1/2}=\mathbb{C}^2$, where the group $G$ acts by matrix-vector multiplication, and $j(z_1,z_2)=(z^*_2,-z^*_1)$. Then we have
\[
\begin{aligned}
gj(z_1,z_2)&=
\begin{pmatrix}
  \alpha & \beta \\
  -\beta^* & \alpha^*
\end{pmatrix}
\begin{pmatrix}
  z^*_2 \\
  -z^*_1
\end{pmatrix}
=
\begin{pmatrix}
\alpha z^*_2-\beta z^*_1 \\
-\beta^*z^*_2-\alpha^*z^*_1
\end{pmatrix}
,\\
jg(z_1,z_2)&=j
\begin{pmatrix}
  \alpha & \beta \\
  -\beta^* & \alpha^*
\end{pmatrix}
\begin{pmatrix}
  z_1 \\
  z_2
\end{pmatrix}
=j
\begin{pmatrix}
\alpha z_1+\beta z_2 \\
-\beta^*z_1+\alpha^*z_2
\end{pmatrix}
=
\begin{pmatrix}
-\beta z^*_1+\alpha z^*_2 \\
-\alpha^*z^*_1-\beta^*z^*_2
\end{pmatrix}
,
\end{aligned}
\]
see also \cite[Proposition~3]{MR1090745}. In this case, there exists an irreducible quaternionic representation, call it $W_{1/2}$, such that $c'W_{1/2}=V_{1/2}$, where the representation $c'W_{1/2}$ has the same underlying set as $W_{1/2}$ and the same action of $G$, but regard it as a $\mathbb{C}$-linear space. Any irreducible complex representation $V$ with the above property is called a \emph{complex representation of quaternionic type},\index{representation!complex!of quaternionic type} the corresponding representation $W$ a \emph{quaternionic representation of quaternionic type},\index{representation!quaternionic!of quaternionic type} and the real irreducible representation $U=rV=rc'W$ a \emph{real representation of quaternionic type}.\index{representation!real!of quaternionic type} It is easy to see that $W_{1/2}=\mathbb{H}^1$, where a matrix $g$ acts by $g\cdot q=(\alpha+\beta\mathrm{j})q$. The representation $W_{1/2}$ establishes an isomorphism between the group $\mathrm{SU}(2)$ and the group $\mathrm{Sp}(1)$ of all $q\in\mathbb{H}$ with $qq^*=1$. The irreducible real representation $U_{1/2}=rV_{1/2}=rc'W_{1/2}$ acts in $\mathbb{R}^4$ by
\[
g\cdot\mathbf{u}=
\begin{pmatrix}
  \RE\alpha & -\IM\alpha & \RE\beta & -\IM\beta \\
  \IM\alpha & \RE\alpha & \IM\beta & \RE\beta \\
  -\RE\beta & -\IM\beta & \RE\alpha & \IM\alpha \\
  \IM\beta & -\RE\beta & -\IM\alpha & \RE\alpha
\end{pmatrix}
\mathbf{u},\qquad\mathbf{u}\in\mathbb{R}^4,
\]
see \cite[Example~15]{MR1090745}.

It turns out that for any compact topological group $G$ and for any skew field $\mathbb{K}$, the partition of the set $\hat{G}_{\mathbb{K}}$ of equivalence classes of irreducible representations of $G$ over $\mathbb{K}$ into representations of real, complex, and quaternionic types is exhaustive, see \cite[Theorem~3.57]{MR0252560} and \cite[Section~II.5]{MR1410059}.

\subsection{Manifolds}

A \emph{chart}\index{chart} in a set $M$ is a pair $(\mathcal{U},\varphi)$, where $\mathcal{U}$ is a subset of $M$, and $\varphi$ is a one-to-one map from $\mathcal{U}$ to an open subset $\varphi(\mathcal{U})$ of the coordinate space $\mathbb{R}^d$ or $\mathbb{C}^d$.

The \emph{overlap map}\index{overlap map} is the map
\[
\varphi_{\beta}\circ\varphi^{-1}_{\alpha}\colon\varphi_{\alpha}
(\mathcal{U}_{\alpha}\cap\mathcal{U}_{\beta})\to\varphi_{\beta}
(\mathcal{U}_{\alpha}\cap \mathcal{U}_{\beta})
\]
provided that the intersection $\mathcal{U}_{\alpha}\cap\mathcal{U}_{\beta}$ of the domains of the charts $(\mathcal{U}_{\alpha},\varphi_{\alpha})$ and $(\mathcal{U}_{\beta},\varphi_{\beta})$ is not empty.

A family $\mathcal{A}=\{(\mathcal{U}_{\alpha},\varphi_{\alpha})\colon\alpha\in A\}$ of charts on $M$ is called an \emph{differentiable}\index{atlas!differentiable} (resp. \emph{real-analytic},\index{atlas!real-analytic} resp. \emph{holomorphic})\index{atlas!holomorphic} \emph{atlas} if the sets $\mathcal{U}_{\alpha}$ cover all of $M$, and for any two charts $(\mathcal{U}_{\alpha},\varphi_{\alpha})$ and $(\mathcal{U}_{\beta},\varphi_{\beta})$ with $\mathcal{U}_{\alpha}\cap\mathcal{U}_{\beta}\neq\varnothing$, both the domain and range of the corresponding overlap map are open subsets of $\mathbb{R}^d$ (resp. $\mathbb{R}^d$, resp. $\mathbb{C}^d$), and the overlap map is infinitely differentiable (resp. real-analytic, resp. holomorphic).

Two atlases are \emph{equivalent} if their union is also an atlas. A differentiable (resp. real-analytic, resp. holomorphic) \emph{structure}\index{structure!differentiable}\index{structure!real-analytic}\index{structure!holomorphic} on $M$ is an equivalence class of differentiable (resp. real-analytic, resp. holomorphic) atlases. A set $M$ together with a differentiable (resp. real-analytic, resp. holomorphic) structure is called a differentiable (resp. real-analytic, resp. holomorphic) \emph{manifold}.\index{manifold!differentiable}\index{manifold!real-analytic}\index{manifold!holomorphic}  A chart is called \emph{admissible}\index{chart!admissible} if it belongs to an atlas from the structure. The domains of the admissible charts form a base for a topology on $M$. To avoid pathologies, we assume that this topology is Hausdorff and second-countable.

Let $M_1$ and $M_2$ be two manifolds of the same type, and let $\Phi\colon M_1\to M_2$ be a continuous map. Let $x\in M$ and let $(\mathcal{V},\psi)$ be a chart for $M_2$ with $\Phi(x)\in\mathcal{V}$. It is possible to choose a chart $(\mathcal{U},\varphi)$ for $M_1$ with $\Phi(\mathcal{U})\subset\mathcal{V}$. The map
\[
\Phi_{\varphi\psi}\colon\varphi(\mathcal{U})\to\psi(\mathcal{V}),\qquad x\mapsto\psi\circ\Phi\circ\varphi^{-1}(x)
\]
is called the \emph{the local representative\index{local representative} of $\Phi$ with respect to the charts $(\mathcal{U},\varphi)$ and $(\mathcal{V},\psi)$}. A map $\Phi$ is called differentiable\index{map!differentiable} (resp. real-analytic,\index{map!real-analytic} resp. holomorphic)\index{map!holomorphic} if for every $x\in M$ and every chart $(\mathcal{V},\psi)$ for $M_2$ with $\Phi(x)\in\mathcal{V}$ there exists a chart $(\mathcal{U},\varphi)$ for $M_1$ such that $\Phi(\mathcal{U})\subset\mathcal{V}$ and the local representative $\Phi(\mathcal{U})\subset\mathcal{V}$ is differentiable (resp. real-analytic, resp. holomorphic). If $\Phi$ is invertible and both $\Phi$ and $\Phi^{-1}$ are differentiable (resp. real-analytic, resp. holomorphic), then $\Phi$ is called a \emph{diffeomorphism}.\index{diffeomorphism}

A group $G$ is called a \emph{Lie group}\index{Lie group} if it has a structure of either real-analytic or holomorphic manifold, and the map $G\times G\to G$, $(g_1,g_2)\mapsto g^{-1}_1g_2$ is either real-analytic or holomorphic.

\begin{example}
Let $S^d\subset\mathbb{R}^{d+1}$ be the unit sphere.\index{sphere} Define two chart domains $\mathcal{U}_{\pm}\subset S^d$ by
$\mathcal{U}_{\pm}=S^d\setminus\{(0,\dots,0,\pm 1)^{\top}\}$ and tho chart maps $\varphi_{\pm}\colon\mathcal{U}_{\pm}\to\mathbb{R}^d$ by
\[
\varphi_{\pm}(x_1,\dots,x_{d+1})=\frac{1}{1\mp x_{d+1}}(x_1,\dots,x_d)^{\top}.
\]
It is easy to check that $\varphi_{\pm}(\mathcal{U}_{\pm})=\mathbb{R}^d\setminus\{\mathbf{0}\}$ and
\[
\varphi^{-1}_+(\mathbf{x})
=\frac{1}{\|\mathbf{x}\|^2+1}(2x_1,\dots,2x_d,\|\mathbf{x}\|^2-1)^{\top}.
\]
The overlap map $(\varphi_-\circ\varphi^{-1}_+)(\mathbf{x})=\|\mathbf{x}\|^{-2}\mathbf{x}$ for $\mathbf{x}\in\mathbb{R}^d\setminus\{\mathbf{0}\}$ is real-analytic and makes $S^d$ into a real-analytic manifold.

In the case of $d=2$, we will also use an admissible chart
\[
\left(
\begin{smallmatrix}
  \cos\varphi\sin\theta \\
  \sin\varphi\sin\theta \\
  \cos\theta
\end{smallmatrix}
\right)\mapsto(\theta,\varphi)\in(0,\pi)\times(0,2\pi)\subset\mathbb{R}^2,
\]
the angular part of the spherical coordinates. The domain of this chart is dense in $S^2$.

In the case of $d=3$, we will also use an admissible chart
\[
\left(
\begin{smallmatrix}
  \cos(\theta/2)\cos((\psi+\varphi)/2) \\
  -\cos(\theta/2)\sin((\psi+\varphi)/2) \\
  -\sin(\theta/2)\cos((\psi-\varphi)/2) \\
  \sin(\theta/2)\sin((\psi-\varphi)/2)
\end{smallmatrix}
\right)\mapsto(\theta,\varphi,\psi)\in(0,\pi)\times(0,2\pi)\times(0,4\pi)
\subset\mathbb{R}^3
\]
with dense domain.

When $d=2$, replace the maps $\varphi_{\pm}$ with the map $\tilde{\varphi}_{\pm}\colon\mathcal{U}_{\pm}\to\mathbb{C}^1$ given by
\[
\tilde{\varphi}_+(x_1,x_2,x_3)=\frac{x_1+\mathrm{i}x_2}{1-x_3},\qquad
\tilde{\varphi}_-(x_1,x_2,x_3)=\frac{x_1-\mathrm{i}x_2}{1+x_3}.
\]
For this choice, we obtain $(\tilde{\varphi}_-\circ\tilde{\varphi}^{-1}_+)(z)=z^{-1}$. By verifying the Cauchy--Riemann equations, we see that this map is a holomorphic on $\mathbb{C}^1\setminus\{0\}$. We made the real-analytic structure into the holomorphic one.

\end{example}

\begin{example}
Let $\mathbb{C}P^1$ be the \emph{complex projective line},\index{complex projective line} the set of $1$-dimensional subspaces of $\mathbb{C}^2$. If $(z_0,z_1)\in\mathbb{C}^2\setminus\{\mathbf{0}\}$, denote by $\pcoor{z_0:z_1}$ the line through this point. For $j\in\{0,1\}$, define the charts $(\mathcal{U}_j,\varphi_j)$ as follows:
\[
\mathcal{U}_j=\{\,\pcoor{z_0:z_1}\in\mathbb{C}P^1\colon z_j\neq 0\,\},\qquad\varphi_0(\pcoor{z_0:z_1})=\frac{z_1}{z_0},\quad
\varphi_1(\pcoor{z_0:z_1})=\frac{z_0}{z_1}.
\]
The holomorphic overlap map is $(\varphi_0\circ\varphi^{-1}_1)(z)=z^{-1}$. The maps $\Phi_{\pm}\colon\mathcal{U}_{\pm}\to\mathbb{C}P^1$ given by
\[
\Phi_+(x,y,z)=\pcoor{\tilde{\varphi}_+(x_1,x_2,x_3):1},\qquad
\Phi_-(x,y,z)=\pcoor{1:\tilde{\varphi}_-(x_1,x_2,x_3)}
\]
coincide on the sphere without the poles. Together, they give a holomorphic diffeomorphism between $S^2$ and $\mathbb{C}P^1$.
\end{example}

\begin{example}
For the group $\mathrm{U}(1)=\{\,z\in\mathbb{C}\colon|z|=1\,\}$, the map $\Phi\colon S^1\to\mathrm{U}(1)$, $(x_1,x_2)^{\top}\to x_1+\mathrm{i}x_2$ is one-to-one. For the group $\mathrm{SO}(2)$ of orthogonal $2\times 2$ matrices with unit determinant, the map $\Psi\colon\mathrm{SO}(2)\to S^1$, $\left(
\begin{smallmatrix}
  a & b \\
  -b & a \\
\end{smallmatrix}
\right)\mapsto(a,b)^{\top}$, is one-to-one. As $(\mathcal{U},\psi)$ runs over the set of all admissible atlases for the real-analytic structure on $S^1$, the charts $(\Phi(\mathcal{U}),\psi\circ\Phi^{-1})$ constitute the structure of a Lie group on $\mathrm{U}(1)$.\index{Lie group!$\mathrm{U}(1)$} We will use an admissible chart $\mathrm{e}^{\mathrm{i}\varphi}\mapsto\varphi\in(0,2\pi)\subset\mathbb{R}^1$ with dense domain.

The charts $(\Psi^{-1}(\mathcal{U}),\psi\circ\Psi^)$ constitute the structure of a Lie group on $\mathrm{SO}(2)$.\index{Lie group!$\mathrm{SO}(2)$} We will use an admissible chart $\left(
\begin{smallmatrix}
  \cos\varphi & -\sin\varphi \\
  \sin\varphi & \cos\varphi \\
\end{smallmatrix}
\right)\mapsto\varphi\in(0,2\pi)\subset\mathbb{R}^1$ with dense domain.

For the group $\mathrm{SU}(2)$ of unitary $2\times 2$ matrices with unit determinant, the map $\Phi\colon\mathrm{SU}(2)\to S^3$, $\left(
\begin{smallmatrix}
  a & b \\
  -b^* & a^* \\
\end{smallmatrix}
\right)\mapsto(\RE a,\IM a,-\RE b,\IM b)^{\top}$, is one-to-one. The charts $(\Phi(\mathcal{U}),\psi\circ\Phi^{-1})$ constitute the structure of a Lie group on $\mathrm{SU}(2)$.\index{Lie group!$\mathrm{SU}(2)$} We will use an admissible chart $\left(
\begin{smallmatrix}
  \cos(\beta/2)\mathrm{e}^{-\mathrm{i}(\gamma+\alpha)/2} & \sin(\beta/2)\mathrm{e}^{\mathrm{i}(\gamma-\alpha)/2} \\
  -\sin(\beta/2)\mathrm{e}^{-\mathrm{i}(\gamma-\alpha)/2} & \cos(\beta/2)\mathrm{e}^{\mathrm{i}(\gamma+\alpha)/2}
\end{smallmatrix}
\right)\mapsto(\alpha,\beta,\gamma)\in(0,2\pi)\times(0,\pi)\times(0,4\pi)\subset\mathbb{R}^3$ with dense domain and call it the \emph{Euler angles}.\index{Euler angles}
\end{example}

\subsection{Bundles}

\begin{definition}
A quadruple $\lambda=(P,H,B,\tilde{\pi})$ is called a \emph{principal bundle}\index{bundle!principal} if and only if
\begin{itemize}
  \item the \emph{total space}\index{total space!of a principal bundle} $P$ and the \emph{base space}\index{base space!of a principal bundle} $B$ are either real-analytic of holomorphic manifolds, the \emph{fibre}\index{fibre!of a principal bundle} $H$ is a Lie group;
  \item $H$ acts on $P$ from the right, and the map $P\times H\to P$, $(p,h)\mapsto p\cdot h$ is either real analytic or holomorphic;
  \item the \emph{bundle projection}\index{bundle projection!of a principal bundle} $\tilde{\pi}\colon P\to B$ is either a real-analytic or holomorphic map that satisfies to the following \emph{local triviality condition}:\index{local triviality condition} for every $b\in B$, there exists a neighbourhood $\mathcal{U}$ of $b$, and either a real-analytic or holomorphic one-to-one map $f\colon\mathcal{U}\times H\to\tilde{\pi}^{-1}(\mathcal{U})$ such that for all $u\in\mathcal{U}$ and for all $h$, $h'\in H$ we have
      \[
      \tilde{\pi}(f(u,h))=u,\qquad f(u,hh')=f(u,h)\cdot h'.
      \]
\end{itemize}
\end{definition}

\begin{example}[A cylinder and a Möbius bundle]
Let $n$ be an integer. Put
\[
P_n=\{\,(\cos\theta,\sin\theta,r\cos(n\theta/2),r\sin(n\theta/2))^{\top}\in\mathbb{R}^4\colon\theta,r\in\mathbb{R}\,\}.
\]
The bundle projection $\tilde{\pi}(p_1,p_2,p_3,p_4)=(p_1,p_2)$ maps $P_n$ onto the circle $B=S^1\subset\mathbb{R}^2$. The group $H$ is $\mathbb{R}^1$ acting by scalar-vector multiplication. The principal bundle $\lambda_n=(P_n,H,B,\tilde{\pi})$ is a \emph{cylinder}\index{cylinder} for $n=0$ and a Möbius bundle\index{Möbius bundle} if $n=1$.
\end{example}

\begin{example}[Spin groups]\label{ex:spin}
Let $n$ be a positive integer. There is a Lie group called the \emph{spin group}\index{spin group} and denoted by $\mathrm{Spin}(n)$, and a map $\tilde{\pi}_n\colon\mathrm{Spin}(n)\to\mathrm{SO}(n)$ such that the quadruple $\lambda_n=(\mathrm{Spin}(n),\mathrm{O}(1),\mathrm{SO}(n),\tilde{\pi}_n)$ is a principal bundle, see \cite[Subsection~1.2.1]{MR3410545}. In particular, $\mathrm{Spin}(1)=\mathrm{O}(1)$ with $\tilde{\pi}_1(t)=t^2$, and $\mathrm{Spin}(2)=\mathrm{U}(1)$ with
\begin{equation}\label{eq:51}
\tilde{\pi}_2(\mathrm{e}^{\mathrm{i}\theta})=\left(
\begin{smallmatrix}
  \cos(2\theta) & -\sin(2\theta) \\
  \sin(2\theta) & \cos(2\theta)
\end{smallmatrix}
\right).
\end{equation}
Under the identification of $\mathrm{SO}(2)$ with $\mathrm{U}(1)$, the map $\tilde{\pi}_2$ is again a ``square'' map $z\mapsto z^2$.

The map $\tilde{\pi}_3\colon\mathrm{Spin}(3)=\mathrm{SU}(2)\to\mathrm{SO}(3)$ has the form
\[
\tilde{\pi}_3\left(
\begin{smallmatrix}
  a & b \\
  -b^* & a^*
\end{smallmatrix}
\right)=\left(
\begin{smallmatrix}
  \RE(a^2-b^2) & \IM(a^2+b^2) & -2\RE(ab) \\
  -\IM(a^2-b^2) & \RE(a^2+b^2) & 2\IM(ab)  \\
  2\RE(ab^*) & 2\IM(ab^*) & |a|^2-|b|^2
\end{smallmatrix}
\right).
\]
\end{example}

\begin{example}[The Hopf bundle]\label{ex:cpp}
Let $P=\mathrm{SU}(2)$, $H$ be the subgroup of diagonal matrices in $P$, isomorphic to $\mathrm{U}(1)$, $B=\mathbb{C}P^1$, and let $\tilde{\pi}\colon P\to B$ acts by
\[
\tilde{\pi}\left(
\begin{smallmatrix}
  a & b \\
  -b^* & a^*
\end{smallmatrix}
\right)=\pcoor{a:b}.
\]
The quadruple $(\mathrm{SU}(2),\mathrm{U}(1),\mathbb{C}P^1,\tilde{\pi})$ is a principal bundle over $\mathbb{C}P^1$ called the \emph{Hopf bundle}.\index{Hopf bundle} The \emph{fibre} $\tilde{\pi}^{-1}(\pcoor{\alpha:\beta})$ of this bundle over a point $\pcoor{a:b}\in\mathbb{C}P^1$ is a copy of the \emph{torsor}\index{torsor} (the underlying set) of the group $\mathrm{U}(1)$, embedded into $\mathrm{SU}(2)$ as the subset $\left(
\begin{smallmatrix}
  \mathrm{e}^{\mathrm{i}\theta}a & \mathrm{e}^{\mathrm{i}\theta}b \\
  -\mathrm{e}^{-\mathrm{i}\theta}b^* & \mathrm{e}^{-\mathrm{i}\theta}a^*
\end{smallmatrix}
\right)$ with $\theta\in\mathbb{R}$.
\end{example}

\begin{example}[Principal bundles over a sphere]\label{ex:principals2}
Let $H$ be the group of orthogonal $3\times 3$ matrices $h$ with $\det h=h_{33}=1$. In fact, $H$ is isomorphic to the group $\mathrm{SO}(2)$. The quadruple $(\mathrm{SO}(3),\mathrm{SO}(2),S^2,\tilde{\pi})$, where $\tilde{\pi}(g)=(0,0,1)g$, is a principal bundle over $S^2$. The fibre of this bundle over a point
\begin{equation}\label{eq:44}
b=(2\RE(\alpha\beta^*),2\IM(\alpha\beta^*),|\alpha|^2-|\beta|^2)^{\top}\in S^2\subset\mathbb{R}^3
\end{equation}
is a copy of $\mathrm{SO}(2)$ embedded into $\mathrm{SO}(3)$ as the subset
\begin{equation}\label{eq:43}
\left(
\begin{smallmatrix}
  \RE(\mathrm{e}^{2\mathrm{i}\theta}(\alpha^2-\beta^2)) & \IM(\mathrm{e}^{2\mathrm{i}\theta}(\alpha^2+\beta^2)) & -2\RE(\mathrm{e}^{2\mathrm{i}\theta}\alpha\beta) \\
  -\IM(\mathrm{e}^{2\mathrm{i}\theta}(\alpha^2-\beta^2)) & \RE(\mathrm{e}^{2\mathrm{i}\theta}(\alpha^2+\beta^2)) & 2\IM(\mathrm{e}^{2\mathrm{i}\theta}\alpha\beta) \\
  2\RE(\alpha\beta^*) & 2\IM(\alpha\beta^*) & |\alpha|^2-|\beta|^2
\end{smallmatrix}
\right)
\end{equation}
with $\theta\in\mathbb{R}$.

The quadruple $(\mathrm{O}(3),\mathrm{O}(2),S^2,\tilde{\pi})$ is another principal bundle over $S^2$.
\end{example}

Let $\lambda=(P,H,B,\tilde{\pi})$ be a principal bundle. Assume that $H$ acts on either a real-analytic or a holomorphic manifold $F$ from the left, and the map $H\times F\to F$, $(h,f)\mapsto h\cdot f$ is either real-analytic or holomorphic. The group $H$ acts on the right on the Cartesian product $P\times F$ by $(p,f)\cdot h=(p\cdot h,h^{-1}\cdot f)$. Denote by $E$ the set of orbits for this action. Denote the orbit of a point $(p,f)\in P\times F$ by $[p,f]$. Define the map $\pi\colon E\to B$ by $\pi([p,f])=\tilde{\pi}(p)$. The map $\pi$ is continuous with respect to the quotient topology on $E$.

\begin{definition}
The triple $(E,B,\pi)$ is called the \emph{fibre bundle\index{bundle!fibre} associated with the principal bundle} $\lambda=(P,H,B,\tilde{\pi})$ by the action $(h,f)\mapsto h\cdot f$. The manifold $E$ is the \emph{total space}\index{total space!of a fibre bundle} of the bundle, the manifold $B$ is its \emph{base space},\index{base space!of a fibre bundle} the map $\pi$ is the \emph{bundle projection},\index{bundle projection!of a fibre bundle} while $F$ is the \emph{fibre}.\index{fibre!of a fibre bundle}
\end{definition}

The triple $(E,B,\pi)$ satisfies the local triviality condition: for each $b\in B$ there is an open set $\mathcal{U}$ containing $b$ and a homeomorphism $\Psi\colon\pi^{-1}(\mathcal{U})\to\mathcal{U}\times F$ such that for all $b\in\mathcal{U}$ and $f\in F$ we have $\pi(\Psi^{-1}(b,f))=b$. If $H$ were not a Lie group, then this statement would be wrong, see a counterexample in \cite[p.~219]{MR3444405}.

In what follows, we consider a particular case, when the fibre $F=L_0$ is a finite-dimensional linear space over a (skew) field $\mathbb{K}$, and $H$ is a Lie group that acts on $L_0$ by a representation $\theta$. The fibre bundle $(E,B,\pi)$ is called a \emph{vector bundle}\index{bundle!vector} (resp. a \emph{tensor bundle},\index{bundle!tensor} resp. a \emph{line bundle})\index{bundle!line} if the elements of $L_0$ are vectors (resp. tensors, resp. elements of $\mathbb{K}^1$).

\begin{definition}
A \emph{cross-section}\index{cross-section} of a fibre bundle $(E,B,\pi)$ is a map $f\colon B\to E$ satisfying $\pi(f(b))=b$ for all $b\in B$.
\end{definition}

Conversely, consider a vector bundle $(E,B,\pi)$, that is, for any $b\in B$, the set $\pi^{-1}(b)$ is an $n$-dimensional linear space over a (skew) field $\mathbb{F}$, and the local triviality condition is satisfied. In addition, assume that the set $\pi^{-1}(b)$ carries an inner product $(\cdot,\cdot)$ which is either a real-analytic or holomorphic function of $b$. Let $P$ be the set of points $\mathbf{p}=(p_1,\dots,p_n)\in E^n$ such that $\pi(p_1)=\pi(p_2)=\cdots=\pi(p_n)=b\in B$, and the vectors $p_1$, \dots, $p_n$ constitute an orthonormal basis in $\pi^{-1}(b)$.

\begin{definition}
The quadruple $(P,H,B,\tilde{\pi})$, where $\tilde{\pi}(\mathbf{p})=\pi(p_1)$ and $H$ acts on $P$ by a representation $(p_1,\dots,p_n)h=(p_1\cdot h,\dots,p_n\cdot h)$, is called the \emph{principal bundle of orthonormal frames}\index{bundle!principal!of orthonormal frames} on the vector bundle $(E,B,\pi)$.
\end{definition}

\begin{example}\label{ex:spinweight}
Consider the principal fibre bundle $(\mathrm{SU}(2),\mathrm{U}(1),\mathbb{C}P^1,\tilde{\pi})$ of Example~\ref{ex:cpp}. For any integer or half-integer $s$, consider the representation $z\mapsto z^{2s}$ of the group $\mathrm{U}(1)$ in the linear space $L_0=\mathbb{C}^1$. The cross-sections of the line bundle $(E,\mathbb{C}P^1,\pi_{2s})$ associated to the principal fibre bundle by the above representation, are called \emph{functions of spin weight~$s$},\index{functions of spin weight $s$} see \cite{MR484262,MR272918,MR194172}.

The line bundle $(E,\mathbb{C}P^1,\pi_2)$ is the tangent bundle $T\mathbb{C}P^1$, see proof in \cite[p.~92]{MR2093043}. Let $(\mathrm{Spin}(T\mathbb{C}P^1),\mathrm{U}(1),\mathbb{C}P^1,\tilde{\pi})$ be the principal bundle of orthonormal frames corresponding to $T\mathbb{C}P^1$. The fibres of this bundle are copies of $\mathrm{U}(1)$.

On the other hand, consider the principal fibre bundle $(\mathrm{SO}(3),\mathrm{SO}(2),S^2,\tilde{\pi}')$ of Example~\ref{ex:principals2}. The representation $\left(
\begin{smallmatrix}
  \cos\theta & -\sin\theta \\
  \sin\theta & \cos\theta
\end{smallmatrix}
\right)\mapsto\mathrm{e}^{2\mathrm{i}\theta}$ determines a line bundle over $S^2$. In fact, it is the tangent bundle $TS^2$, in which the tangent space at a point $\mathbf{u}\in S^2\subset\mathbb{R}^3$ is first identified with the two-dimensional real linear space $T_{\mathbf{u}}S^2=\{\,\mathbf{v}\in\mathbb{R}^3\colon(\mathbf{v},\mathbf{u})=0\,\}$, and then $T_{\mathbf{u}}S^2$ is converted into a complex line by defining $z\mathbf{v}=\RE z\mathbf{v}+\IM z(\mathbf{u}\times\mathbf{v})$. Denote by $(\mathrm{SO}(TS^2),\mathrm{SO}(2),S^2,\pi')$ the principal bundle of orthonormal frames corresponding to $TS^2$. The fibres of this bundle are copies of $\mathrm{SO}(2)$.

The map $S^2\to\mathbb{C}P^1$, $\mathbf{u}\mapsto\pcoor{u_1+u_2\mathrm{i}:1-u_3}$ is a homeomorphism. Consider the map $\Theta\colon\mathrm{Spin}(T\mathbb{C}P^1)\to\mathrm{SO}(TS^2)$ given by $\Theta(\pcoor{u_1+u_2\mathrm{i},1-u_3},z)=(\mathbf{u},z^2)$. By construction, this map is fibre-preserving. Moreover, for all $(\pcoor{u_1+u_2\mathrm{i}:1-u_3},z)\in\mathrm{Spin}(T\mathbb{C}P^1)$ and for all $w\in\mathrm{U}(1)=\mathrm{Spin}(2)$ we have
\[
\Theta((\pcoor{u_1+u_2\mathrm{i}:1-u_3},z)w)
=\Theta(\pcoor{u_1+u_2\mathrm{i}:1-u_3},z)w^2.
\]
By \cite[Definition~2.1]{MR3410545}, the map $\Theta$ defines a \emph{spin structure}\index{spin structure} on the manifold $S^2$.

Recall that the \emph{complex spin representation}\index{representation!complex spin} of the group $\mathrm{Spin}(2)=\mathrm{U}(1)$ is the complex representation $\rho_2(\mathrm{e}^{\mathrm{i}\theta})=\left(
\begin{smallmatrix}
  \mathrm{e}^{\mathrm{i}\theta} & 0 \\
  0 & \mathrm{e}^{-\mathrm{i}\theta}
\end{smallmatrix}
\right)$. Its irreducible components $\rho_2^+(\mathrm{e}^{\mathrm{i}\theta})=\mathrm{e}^{\mathrm{i}\theta}$ and $\rho_2^-(\mathrm{e}^{\mathrm{i}\theta})=\mathrm{e}^{-\mathrm{i}\theta}$ are called \emph{complex spin half-representations}.\index{half-representation!complex spin} The space of the representation $\rho_2$ is equal to the direct sum $V_1\oplus tV_1$ of linear spaces where the above components act. Note that these representations do not descend to the group $\mathrm{SO}(2)$ since $\rho_2(-1)=-\id$ and $\rho_2^{\pm}(-1)=-1$.

The complex vector bundle $(\Sigma\mathbb{C}P^1,\mathbb{C}P^1,\pi)$ with fibre $V_1\oplus tV_1$ is the bundle associated with the principal fibre bundle $(\mathrm{SU}(2),\mathrm{U}(1),\mathbb{C}P^1,\tilde{\pi})$ by the complex spin representation. We see that this bundle is the direct sum of the two bundles $(\Sigma^+\mathbb{C}P^1,\mathbb{C}P^1,\pi_1)$ and $(\Sigma^-\mathbb{C}P^1,\mathbb{C}P^1,\pi_{-1})$ with fibres $V_1$ and $tV_1$ that correspond to the cases of $s=\pm 1$ above. The cross-sections of the bundle $(\Sigma\mathbb{C}P^1,\mathbb{C}P^1,\pi)$ are called \emph{spinor fields}.\index{spinor field}

Note that the map $j\colon V_1\oplus tV_1\to V_1\oplus tV_1$, $j(z_1,z_2)=(z^*_2,z^*_1)$ is a real structure on the space $V_1\oplus tV_1$ that commutes with the complex spin representation. The corresponding \emph{real spin representation}\index{representation!real spin} $\mathrm{e}^{\mathrm{i}\theta}\mapsto\left(
\begin{smallmatrix}
  \cos\theta & -\sin\theta \\
  \sin\theta & \cos\theta
\end{smallmatrix}
\right)$ is $rV_1$. Similarly, the map $j:V_1\oplus tV_1\to V_1\oplus tV_1$, $j(z_1,z_2)=(z^*_2,-z^*_1)$ is a quaternionic structure on the space $V_1\oplus tV_1$ that commutes with the complex spin representation. The corresponding \emph{quaternionic spin representation}\index{representation!quaternionic spin} $\mathrm{e}^{\mathrm{i}\theta}\mapsto\mathrm{e}^{\mathrm{i}\theta}\in\mathrm{Sp}(1)$
is $qV_1$.

It turns out that not all manifolds may carry a spin structure, there exist topological obstructions for that, see \cite[Section~3.1]{MR3410545} and \cite[Section~II.3]{MR1031992}.

Finally, note that the maps $\tilde{\pi}_1(t)=t^2$ and $\tilde{\pi}_2(z)=z^2$ of Example~\ref{ex:spin} lead to the idea of a ``square root'' between vectors and spinors, see \cite{MR1443911}.
\end{example}

\begin{example}[The edth operator]
There exists a beautiful coordinate-free description of the operator $\eth$ in \cite{MR671187}, where it is associated with the $\partial$-operator and the Dolbeault resolution of complex analysis, see also \cite[Equation~(4.12.15)]{MR917488}. We give a coordinate description instead. Note that $\eth$ is the phonetic symbol for voiced ``th'', see \cite[Section~4.12]{MR917488}.

Introduce a chart $\mathbb{C}P^1\setminus\{\pcoor{1:0}\}\to\mathbb{C}$, $\pcoor{z_1:z_2}\mapsto\zeta=\frac{z_1}{z_2}$. This is a standard stereographic correspondence between the \emph{Argand plane}\index{Argand plane} of $\zeta$ and the complex projective line without the point at infinity. To include the above point, we regard $\zeta=\infty$ as a ``point'' added to the Argand plane, and map it to the point $\pcoor{1:0}\in\mathbb{C}P^1$. The union $\mathbb{C}\cup\{\infty\}$ is the Riemann sphere.\index{Riemann sphere}

Let $f$ be a smooth cross-section of the vector bundle $(\mathrm{SU}(2),\mathbb{C}P^1,\pi_s)$ of Example~\ref{ex:spinweight}. We define
\[
\begin{aligned}
(\eth f)(\zeta)&=(1+\zeta\zeta^*)^{1-s}\frac{\partial}{\partial\zeta}
((1+\zeta\zeta^*)^sf(\zeta)),\\
(\overline{\eth} f)(\zeta)&=(1+\zeta\zeta^*)^{1+s}\frac{\partial}{\partial\zeta^*}
((1+\zeta\zeta^*)^{-s}f(\zeta)),
\end{aligned}
\]
where we adapted definition \cite[Equation~(4.15.117)]{MR917488} in order to be consistent with cosmological literature. Then, the \emph{spin rising operator}\index{spin rising operator} $\eth$ maps smooth cross-sections of the bundle $(\mathrm{SU}(2),\mathbb{C}P^1,\pi_s)$ to smooth cross-sections of the bundle $(\mathrm{SU}(2),\mathbb{C}P^1,\pi_{s+1})$. Similarly, the \emph{spin lowering operator}\index{spin lowering operator} $\overline{\eth}$ maps smooth cross-sections of the bundle $(\mathrm{SU}(2),\mathbb{C}P^1,\pi_s)$ to smooth cross-sections of the bundle $(\mathrm{SU}(2),\mathbb{C}P^1,\pi_{s-1})$.
\end{example}

\subsection{Induced representations}

Let $H$ be a closed subgroup of a compact Lie group $G$, and let $B=G/H$. Consider the principal fibre bundle $(G,H,B,\tilde{\pi})$, where $\tilde{\pi}(g)=gH$. Let $(E,B,\pi)$ be the vector bundle associated to the above principal fibre bundle by a representation of $H$ ia a finite-dimensional linear space $L_0$ over a (skew) field $\mathbb{K}$. Define a left action of $G$ on $G\times L_0$ by $g_0\cdot(g,\mathbf{l})=(g_0g,\mathbf{l})$. The above action induces an action of $G$ on $E$. It is easy to check that with this action, the vector bundle $(E,B,\pi)$ becomes homogeneous.

\begin{definition}[\cite{MR0498996}]
A vector bundle  $(E,G/H,\pi)$ is called \emph{homogeneous}\index{bundle!vector!homogeneous} if $G$ acts on $E$ such that $g\cdot\pi^{-1}(x)=\pi^{-1}(g\cdot x)$ and the action of $G$ on $\pi^{-1}(x)$ is linear for all $x\in G/H$ and $g\in G$.
\end{definition}

The construction given above describes all homogeneous vector bundles, see \cite{MR0498996}.

Let $C(G,L_0)$ be the linear space of all continuous functions $f\colon G\to L_0$ satisfying the condition $f(gh)=h^{-1}\cdot f(g)$, $h\in H$, $g\in G$. Introduce the notation
\begin{equation}\label{eq:37}
\langle f_1,f_2\rangle=\int_G(f_1(g),f_2(g))\,\mathrm{d}g,
\end{equation}
where $f_1$, $f_2\in C(G,L_0)$, and $\mathrm{d}g$ is the probabilistic $G$-invariant measure on the Borel $\sigma$-field of subsets of $G$. We observe that the map $C(G,L_0)\times C(G,L_0)\to\mathbb{K}$, $(f_1,f_2)\mapsto\langle f_1,f_2\rangle$, is an inner product. The map $(\tilde{\theta}(g_0)f)(g)=f(g^{-1}_0g)$ extends to a representation of $G$ in the Hilbert space $L^2(G,L_0)$, which is the completion of the linear space $C(G,L_0)$ with respect to the norm induced by the inner product \eqref{eq:37}. Moreover, the above inner product is $G$-invariant. This fact is proved in \cite[Subsection~5.3.3]{MR0498996} for the case of $\mathbb{K}=\mathbb{C}$, and can be proved literally in the same way for the remaining (skew) fields. In what follows, in similar cases we refer to proofs of our statements for the complex case.

Denote by $\Gamma E$ the $\mathbb{K}$-linear space of continuous cross-sections of the homogeneous vector bundle $(E,B,\pi)$. Let $G$ act on $\Gamma E$ by $(\theta(g)s)(x)=g\cdot s(g^{-1}x)$ for $g\in G$, $s\in\Gamma E$, and $x\in B$. The representation $(\Gamma E,\theta)$ extends to a representation of $G$ in the Hilbert space $L^2(E)$ of the square-integrable cross-sections of the vector bundle $(E,B,\pi)$ with the $G$-invariant inner product
\[
\langle s_1,s_2\rangle=\int_B(s_1(x),s_2(x))_{L_0}\,\mathrm{d}x,
\]
where $\mathrm{d}x$ is the probabilistic $G$-invariant measure on $B$. The complex case is proved in \cite[Subsection~5.3.2]{MR0498996}. We call $\theta$ the representation of $G$ \emph{induced}\index{representation!induced} by the representation $L_0$ of the subgroup $H$.

Define the map $\Gamma E\to C(G,L_0)$, $s\mapsto\tilde{s}$ by $\tilde{s}(g)=g^{-1}\cdot s(gH)$ for $g\in G$. This map extends to an equivalence between the representations $(L^2(E),\theta)$ and $(L^2(G,L_0),\tilde{\theta})$, see proof for the complex case in \cite[Lemma~5.3.4]{MR0498996}. Act by $g$ to both hand sides of this equality. We obtain, that the inverse map has the form
\begin{equation}\label{eq:40}
s(gH)=g\cdot\tilde{s}(g).
\end{equation}

\subsection{The structure of the induced representation}\label{sub:structure}

The Hilbert space $L^2(E)$ is uniquely decomposed into the Hilbert direct sum of isotypical subspaces. Let $L$ be an irreducible representation of $G$, and let $\res^G_HL$ be the restriction of this representation to the group $H$. Let $\Hom_{\mathbb{K}}(\res^G_HL,L_0)$ be the $\mathbb{K}'$-linear space of $\mathbb{K}$-linear maps from $L$ to $L_0$. The group $H$ acts on this space by $(h\cdot f)\mathbf{l}=h\cdot(f(h^{-1}\cdot\mathbf{l}))$ for $h\in H$, $f\in\Hom_{\mathbb{K}}(\res^G_HL,L_0)$, and $\mathbf{l}\in L$. Let $\Hom_{\mathbb{K}H}(\res^G_HL,L_0)$ be the isotypical subspace of the space $\Hom_{\mathbb{K}}(\res^G_HL,L_0)$ that corresponds to the trivial representation of $H$. In other words, $\Hom_{\mathbb{K}H}(\res^G_HL,L_0)$ is the subspace of elements in $\Hom_{\mathbb{K}}(\res^G_HL,L_0)$ which are invariant under $H$, or the linear space of intertwining operators between $\res^G_HL$ and $L_0$.

Similarly, construct the $\mathbb{K}'$-linear space $\Hom_{\mathbb{K}G}(L,L^2(E))$ of intertwining operators between $L$ and $L^2(E)$. The celebrated \emph{Frobenius reciprocity}\index{Frobenius reciprocity} states that the two above constructed spaces are isomorphic, see \cite[Proposition~2.1]{MR0182022} and \cite[Chapter~III, Proposition~6.2]{MR1410059}. In particular, $\dim_{\mathbb{K}'}\Hom_{\mathbb{K}H}(\res^G_HL,L_0)=\dim_{\mathbb{K}'}\Hom_{\mathbb{K}G}(L,L^2(E))$ and the representation $(L^2(E),\theta)$ contains
\[
n=\frac{\dim_{\mathbb{K}'}\Hom_{\mathbb{K}H}(\res^G_HL,L_0)}
{\dim_{\mathbb{K}'}\Hom_{\mathbb{K}G}(L,L)}
\]
copies of the irreducible representation $L$.

The particular case of the above statement, when $\mathbb{K}=\mathbb{R}$ and $H$ contains only the identity element, is called the \emph{Fine Structure Theorem},\index{Fine Structure Theorem} see \cite{MR4201900}. We used this result in Example~\ref{ex:1}.

The next step is construct an orthonormal basis in the Hilbert space $L^2(E)$ in such a way that different kinds of ``spherical harmonics'' are particular cases of our general construction. See also alternative approaches to construction of various spherical harmonics in \cite{MR825165,MR569166}.

\subsection{The construction of a basis in the space of an induced representation}\label{sub:construction}

\subsubsection{Schur's Lemma}

This result describes the structure of the space $\Hom_{\mathbb{K}G}(L_1,L_2)$ for $L_1$, $L_2\in\hat{G}_{\mathbb{K}}$.

\begin{lemma}\label{lem:Schur}\index{Schur's Lemma}
Let $L$, $L_1$, $L_2\in\hat{G}_{\mathbb{K}}$.
\begin{itemize}
  \item Any element of the space $\Hom_{\mathbb{K}G}(L_1,L_2)$ is either zero or an isomorphism. Moreover, $\Hom_{\mathbb{K}G}(L_1,L_2)=\{0\}$ if and only if $L_1$ is not equivalent to $L_2$.
  \item As a linear space (resp. as a division algebra), the set $\Hom_{\mathbb{R}G}(U,U)$ is isomorphic to $\mathbb{R}^1$ (resp. to $\mathbb{R}$) if and only if $U$ is of real type; to $\mathbb{R}^2$ (resp. to $\mathbb{C}$) if and only if $U$ is of complex type; to $\mathbb{R}^4$ (resp. to $\mathbb{H}$) if and only if $U$ is of quaternionic type.
  \item As a linear space (resp. as a division algebra), the set $\Hom_{\mathbb{C}G}(V,V)$ is isomorphic to $\mathbb{C}^1$ (resp. to $\mathbb{C}$).
  \item As a linear space (resp. as a division algebra), the set $\Hom_{\mathbb{H}G}(W,W)$ is isomorphic to $\mathbb{R}^4$ (resp. to $\mathbb{H}$) if and only if $W$ is of real type; to $\mathbb{R}^2$ (resp. to $\mathbb{C}$) if and only if $W$ is of complex type; to $\mathbb{R}^1$ (resp. to $\mathbb{R}$) if and only if $W$ is of quaternionic type.
\end{itemize}
\end{lemma}

\begin{proof}
All items except the last one are proved in \cite[Lemma~3.22, Corollary~3.23]{MR0252560} and \cite[Section~II.6]{MR1410059}.

To prove the last item, note that $c\Hom_{\mathbb{H}G}(W,W)=\Hom_{\mathbb{C}G}(c'W,c'W)$ by \cite[Section~3.9 (iii)]{MR0252560}. If $W$ is of real type, then, by \cite[Proposition~II.6.6 (iv)]{MR1410059}, there exists an irreducible complex representation $V$ of real type such that $c'W=V\oplus V$. By \cite[Proposition~3.11]{MR0252560}, $\Hom$ is bilinear over the direct sum $\oplus$, therefore
\[
\Hom_{\mathbb{C}G}(c'W,c'W)=\Hom_{\mathbb{C}G}(V\oplus V,V\oplus V)=\mathbb{C}^4.
\]
If $W$ is of complex type, then, by \cite[Proposition~II.6.6 (v)]{MR1410059}, there exists an irreducible complex representation $V$ of complex type such that $c'W=V\oplus tV$. Therefore
\[
\Hom_{\mathbb{C}G}(c'W,c'W)=\Hom_{\mathbb{C}G}(V\oplus tV,V\oplus tV)=\mathbb{C}^2.
\]
If $W$ is of quaternionic type, then, by \cite[Proposition~II.6.6 (vi)]{MR1410059}, there exists an irreducible complex representation $V$ of quaternionic type such that $c'W=V$. Therefore
\[
\Hom_{\mathbb{C}G}(c'W,c'W)=\Hom_{\mathbb{C}G}(V,V)=\mathbb{C}^1.
\]
The statement about the isomorphism of linear spaces follows. The Frobenius classification of finite-dimensional real division algebras \cite{MR1581640} states that every such an algebra is isomorphic either to $\mathbb{R}$ or to $\mathbb{C}$ or to $\mathbb{H}$ and implies the statement about the isomorphism of division algebras.
\end{proof}

\subsubsection{The construction of a basis}\label{sub:basis}

Consider the $\mathbb{K}'$-linear space $\Hom_{\mathbb{K}H}(\res^G_HL,L_0)$ and assume that $L_0$ is irreducible. Denote by $D$ the division algebra $\Hom_{\mathbb{K}H}(L_0,L_0)$. The above space contains
\[
n_L=\frac{\dim_{\mathbb{K}'}\Hom_{\mathbb{K}H}(\res^G_HL,L_0)}{\dim_{\mathbb{K}'}D}
\]
copies of the representation $L_0$. Choose and fix the $\mathbb{K}$-linear subspaces $L_0^1$, \dots, $L^{n_L}_0$ of the linear space $\res^G_HL$ where the above copies act. This choice is not unique. However, we have
\[
\Hom_{\mathbb{K}H}(\res^G_HL,L_0)=\Hom_{\mathbb{K}H}(L^1_0,L_0)\oplus\cdots
\oplus\Hom_{\mathbb{K}H}(L^{n_L}_0,L_0).
\]
The isomorphism of division algebras between $D$ and $\Hom_{\mathbb{K}H}(L^i_0,L_0)$, $1\leq i\leq n_L$, described in Schur's Lemma, determines a representation of $D$ in $\Hom_{\mathbb{K}H}(L^i_0,L_0)$.

It is well-known that any finite-dimensional complex representation of the algebra $\mathbb{C}$ is a direct sum of finitely many copies of the unique irreducible representation $\mathbb{C}\to\Hom_{\mathbb{C}}(\mathbb{C}^1,\mathbb{C}^1)$. We choose it in the form $z\mapsto z$. A real representation of the algebra $\mathbb{R}$ (resp. $\mathbb{C}$, resp. $\mathbb{H}$) is a direct sum of finitely many copies of the unique irreducible representation $\mathbb{R}\to\Hom_{\mathbb{R}}(\mathbb{R}^1,\mathbb{R}^1)$ (resp. $\mathbb{C}\to\Hom_{\mathbb{R}}(\mathbb{R}^2,\mathbb{R}^2)$, resp. $\mathbb{H}\to\Hom_{\mathbb{R}}(\mathbb{R}^4,\mathbb{R}^4)$). We choose it in the form $a\mapsto a$ (resp.
\begin{equation}\label{eq:38}
z=a+b\mathrm{i}\mapsto\left(
\begin{smallmatrix}
  a & -b \\
  b & a
\end{smallmatrix}
\right)\in\Hom_{\mathbb{R}}(\mathbb{R}^2,\mathbb{R}^2),
\end{equation}
resp.
\begin{equation}\label{eq:39}
q=a+b\mathrm{i}+c\mathrm{j}+d\mathrm{k}\mapsto\left(
\begin{smallmatrix}
  a & -b & c & -d \\
  b & a & d & c  \\
  -c & -d & a & b \\
  d & -c & -b & a
\end{smallmatrix}
\right)\in\Hom_{\mathbb{R}}(\mathbb{R}^4,\mathbb{R}^4)).
\end{equation}
For proofs, see \cite{MR954184} and \cite{TRAUTMAN2006518}.

We proved the existence of a basis, in which the elements of the linear space $\Hom_{\mathbb{K}H}(\res^G_HL,L_0)$ are matrices over $\mathbb{K}'$ with $\dim_{\mathbb{K}'}L_0$ rows and $\dim_{\mathbb{K}'}L$ columns, in which only $n_L$ components are potentially non-zero. Each component is a square block-diagonal matrix and contains $\frac{\dim_{\mathbb{K}'}L_0}{\dim_{\mathbb{K}'}D}$ identical blocks. Each block is a square matrix with $\dim_{\mathbb{K}'}D$ rows. If $\mathbb{K}'=\mathbb{R}$ and $D=\mathbb{C}$, then each block has the form of the matrix in Equation~\eqref{eq:38}, if $\mathbb{K}'=\mathbb{R}$ and $D=\mathbb{H}$, then each block has the form of the matrix in Equation~\eqref{eq:39}.

Define the matrices $\{\,f^{ijk}\colon 1\leq i\leq n_L,1\leq j\leq\frac{\dim_{\mathbb{K}'}L_0}{\dim_{\mathbb{K}'}D},1\leq k\leq\dim_{\mathbb{K}'}D\,\}$ as follows. All the components of the matrix $f^{ijk}$ but the $i$th one are zero matrices. If $\dim_{\mathbb{K}'}D=1$, then each block contains the number $\frac{1}{\sqrt{\dim_{\mathbb{K}'}L_0}}$. If $\dim_{\mathbb{K}'}D=2$, then each block of the matrix $f^{ij1}$ has the form $\frac{1}{\sqrt{\dim_{\mathbb{K}'}L_0}}\left(
\begin{smallmatrix}
  1 & 0 \\
  0 & 1
\end{smallmatrix}
\right)$, and each block of the matrix $f^{ij2}$ the form $\frac{1}{\sqrt{\dim_{\mathbb{K}'}L_0}}\left(
\begin{smallmatrix}
  0 & -1 \\
  1 & 0
\end{smallmatrix}
\right)$. If $\dim_{\mathbb{K}'}D=4$, then each block of the matrix $f^{ij1}$ has the form $\frac{1}{\sqrt{\dim_{\mathbb{K}'}L_0}}\left(
\begin{smallmatrix}
  1 & 0 & 0 & 0 \\
  0 & 1 & 0 & 0  \\
  0 & 0 & 1 & 0 \\
  0 & 0 & 0 & 1
\end{smallmatrix}
\right)$, each block of the matrix $f^{ij2}$ the form $\frac{1}{\sqrt{\dim_{\mathbb{K}'}L_0}}\left(
\begin{smallmatrix}
  0 & -1 & 0 & 0 \\
  1 & 0 & 0 & 0  \\
  0 & 0 & 0 & 1 \\
  0 & 0 & -1 & 0
\end{smallmatrix}
\right)$, each block of the matrix $f^{ij3}$ the form $\frac{1}{\sqrt{\dim_{\mathbb{K}'}L_0}}\left(
\begin{smallmatrix}
  0 & 0 & 1 & 0 \\
  0 & 0 & 0 & 1  \\
  -1 & 0 & 0 & 0 \\
  0 & -1 & 0 & 0
\end{smallmatrix}
\right)$, and each block of the matrix $f^{ij4}$ the form $\frac{1}{\sqrt{\dim_{\mathbb{K}'}L_0}}\left(
\begin{smallmatrix}
  0 & 0 & 0 & -1 \\
  0 & 0 & 1 & 0  \\
  0 & -1 & 0 & 0 \\
  1 & 0 & 0 & 0
\end{smallmatrix}
\right)$. An orthonormal basis of the space $\Hom_{\mathbb{K}H}(\res^G_HL,L_0)$ is constructed.

Apply the Frobenius reciprocity isomorphism to the matrices $f^{ijk}$. According to the explicit construction of this isomorphism given in \cite[Chapter~III, Proposition~6.2]{MR1410059}, the matrix $f^{ijk}$ becomes the map $F^{ijk}\colon L\to C(G,L_0)$ that acts by $F^{ijk}(\mathbf{l})(g)=f^{ijk}(g^{-1}\cdot\mathbf{l})$ for $\mathbf{l}\in L$ and $g\in G$. The composition of the map \eqref{eq:40} and $F^{ijk}$ acts from $L$ to $L^2(E)$ and maps a vector $\mathbf{l}\in L$ to the cross-section $s^{ijk}_{\mathbf{l}}(gH)=g\cdot f^{ijk}(g^{-1}\cdot\mathbf{l})$ of the vector bundle $(E,G/H,\pi)$. In particular, for the vectors $\mathbf{e}_1$, \dots, $\mathbf{e}_{\dim L}$ of an orthonormal basis in $L$, their images
\[
{}_{L_0}Y_{ijkLm}(gH)=g\cdot f^{ijk}(g^{-1}\cdot\mathbf{e}_m)
\]
form an orthogonal basis in the isotypic subspace of the irreducible representation $L$ of the space $L^2(E)$. It is not necessarily orthonormal. By abuse of notation, we denote an orthonormal basis by the same symbol. The union of the above images over all representations $L\in\hat{G}_{\mathbb{K}}$ for which $\Hom_{\mathbb{K}G}(L,L^2(E))\neq\{0\}$, constitutes an orthonormal basis in the space $L^2(E)$.

If $L_0$ is reducible, then we decompose it into irreducible components and use \cite[Lemma~5.2.5]{MR0498996}, according to which the space $L^2(E)$ is the Hilbert direct sum of subspaces that correspond to each component, and construct an orthonormal basis in each subspace as above.

An alternative construction of the basis for the case of $\mathbb{K}=\mathbb{C}$, $G=\mathrm{SU}(2)$, and $H=\mathrm{U}(1)$ using the Euler angles, is given in \cite{MR825165}.

\providecommand{\noopsort}[1]{}

\printindex

\end{document}